\documentclass[10pt,letterpaper]{article}
\usepackage[margin=2cm]{geometry}
\usepackage[T1]{fontenc}
\usepackage[utf8]{inputenc}
\usepackage[english]{babel}
\usepackage{amsmath,amssymb,mathtools,amsthm}
\usepackage[cal=boondoxo,scr=boondoxo]{mathalpha}
\usepackage{etoolbox,tcolorbox,needspace}
\tcbuselibrary{skins,breakable}
\usepackage{microtype}
\usepackage{booktabs,tabularx,array,enumitem,bm}
\usepackage[numbers,sort&compress]{natbib}
\usepackage{xurl}
\usepackage{float}
\usepackage{hyperref}
\hypersetup{
  colorlinks=true,
  linkcolor=magenta,
  citecolor=green,
  urlcolor=blue,
  pdftitle={Conditional Fluctuations of Mean-Field Systems with Stable Common Noise},
  pdfauthor={Alexandre Autran}
}
\providecommand{\doi}[1]{\href{https://doi.org/#1}{doi:\nolinkurl{#1}}}
\allowdisplaybreaks
\definecolor{theoremframe}{rgb}{0,0,0.55}
\definecolor{theoremback}{rgb}{0.96,0.96,1}
\definecolor{propositionframe}{cmyk}{0.55,0,0,0.45}
\definecolor{propositionback}{cmyk}{0.08,0,0,0}
\definecolor{lemmaframe}{rgb}{0.20471,0.17177,0.32470}
\definecolor{lemmaback}{rgb}{0.94730,0.94290,0.96330}
\definecolor{corollaryframe}{rgb}{0.65059,0.35136,0.13452}
\definecolor{corollaryback}{rgb}{0.99294,0.95020,0.91922}

\tcbset{
  article statement/.style={
    enhanced,
    breakable,
    lines before break=6,
    boxrule=0.8pt,
    arc=1.4mm,
    outer arc=1.4mm,
    boxsep=0pt,
    left=2.5mm,
    right=2.5mm,
    top=3.6mm,
    bottom=2.3mm,
    before skip=9pt plus 2pt minus 1pt,
    after skip=5pt plus 1pt minus 1pt,
    fontupper=\normalfont\normalsize,
    fonttitle=\normalfont\normalsize\bfseries,
    coltitle=white,
    attach boxed title to top left={xshift=2mm,yshift=-2.4mm},
    boxed title style={
      boxrule=0pt,
      arc=0.8mm,
      outer arc=0.8mm,
      boxsep=0pt,
      left=2.8mm,
      right=2.8mm,
      top=1.2mm,
      bottom=1.2mm
    }
  }
}

\newcounter{innerthm}
\newtcolorbox{statementbox}[4]{
  article statement,
  before upper={\addtocounter{innerthm}{-1}\refstepcounter{innerthm}},
  colframe=#2,
  colback=#3,
  colbacktitle=#2,
  title={#1~\theinnerthm\ifstrempty{#4}{}{\enspace|\enspace#4}}
}
\newenvironment{theorem}[1][]{
  \stepcounter{innerthm}
  \begin{statementbox}{Theorem}{theoremframe}{theoremback}{#1}
}{\end{statementbox}}
\newenvironment{lemma}[1][]{
  \stepcounter{innerthm}
  \begin{statementbox}{Lemma}{lemmaframe}{lemmaback}{#1}
}{\end{statementbox}}
\newenvironment{corollary}[1][]{
  \stepcounter{innerthm}
  \begin{statementbox}{Corollary}{corollaryframe}{corollaryback}{#1}
}{\end{statementbox}}
\newenvironment{proposition}[1][]{
  \stepcounter{innerthm}
  \begin{statementbox}{Proposition}{propositionframe}{propositionback}{#1}
}{\end{statementbox}}
\theoremstyle{definition}
\newtheorem{innerassumption}[innerthm]{Assumption}
\newtheorem{innerexample}[innerthm]{Example}
\theoremstyle{remark}
\newtheorem*{remark*}{Remark}
\newtheorem{innerremark}[innerthm]{Remark}
\newenvironment{assumption}[1][]{\begin{innerassumption}[#1]}{\end{innerassumption}}
\newenvironment{remark}[1][]{\begin{innerremark}[#1]}{\end{innerremark}}

\newcolumntype{Y}{>{\raggedright\arraybackslash}X}
\newcommand{\N}{\mathbf N}
\newcommand{\Z}{\mathbf Z}
\newcommand{\R}{\mathbf R}

\newcommand{\Pp}{\mathbf P}
\newcommand{\Ee}{\mathbf E}
\newcommand{\Law}{\mathcal L}

\newcommand{\Prob}{\mathcal P}
\newcommand{\cF}{\mathcal F}
\newcommand{\cG}{\mathcal G}
\newcommand{\dd}{\,\mathrm d}
\newcommand{\ip}[2]{\left\langle #1,#2\right\rangle}
\newcommand{\ind}{\mathbf 1}
\newcommand{\trans}[2]{\tau_{#1}#2}

\newcommand{\byeq}[2]{\overset{\mathclap{\eqref{#1}}}{#2}}
\patchcmd{\abstract}{\quotation}{\quote}{}{}
\patchcmd{\endabstract}{\endquotation}{\endquote}{}{}
\title{Conditional Fluctuations of Mean-Field Systems\\ with Stable Common Noise}
\author{Alexandre Autran\footnote{Department of Mathematics, École normale supérieure de Rennes, France. \href{mailto:alexandre.autran@ens-rennes.fr}{\nolinkurl{alexandre.autran@ens-rennes.fr}}}}
\date{}
\newlength{\equationmaxwidth}
\makeatletter
\everydisplay\expandafter{\the\everydisplay
  \ifboolexpr{
    test {\ifdefstring{\@currenvir}{equation}}
    or test {\ifdefstring{\@currenvir}{align}}
  }{}{\ifdim\displaywidth>\equationmaxwidth
      \advance\displayindent by \dimexpr(\displaywidth-\equationmaxwidth)/2\relax
      \displaywidth=\equationmaxwidth
    \fi
  }}
\makeatother
\begin{document}\maketitle
\begin{abstract}
We study conditional fluctuations of smooth mean-field systems on a torus with symmetric stable common jumps of index $0<\alpha<1$. A positive-strip estimate makes empirical Poisson acceptance errors negligible at the sampling scale. Given the full Poisson master, the conditional laws converge in probability to a centered Gaussian kernel in a negative Sobolev path space. In dimension one, quantile coupling preserves this limit for microscopic marks with relative tail remainder $O(x^{-\rho})$, where $\rho>\tfrac{\alpha}{2}$. A second-order tail expansion identifies the response when this condition fails: it shifts the conditional Gaussian mean at criticality and, when nonzero, dominates sampling below criticality. We also identify the transition from a stable to a Brownian rank error. In a translation model, rank and clock couplings give the same microscopic laws but different path fluctuations. Rank coupling gives a functional Gaussian limit. Clock coupling admits no diverging additive path normalization at the limiting center.
\end{abstract}

\section{Introduction}\label{sec:introduction}
Heavy-tailed collective jumps can produce stable common noise in mean-field limits. The empirical error then combines sampling fluctuations and an error in the common signal. We study how their scales and path limits depend on the coupling with that signal. Under Poisson thinning, a common master supplies candidate jumps, and the particle system and its limit accept them at their empirical and conditional mean rates. For a stable index $0<\alpha<1$, a positive-strip estimate makes the acceptance error smaller than $N^{-\frac{1}{2}}$. The sampling error has a conditional functional Gaussian limit given the full master (Theorem~\ref{mf:clt}). In dimension one, we couple microscopic marks to stable jumps by matching their ranks. A relative tail remainder $O(x^{-\rho})$ preserves the sampling limit when $\rho>\tfrac{\alpha}{2}$ (Theorem~\ref{rk:transfer}). A second-order tail expansion gives the response at and below this threshold: at criticality it shifts the conditional Gaussian mean, and below criticality a nonzero response determines the leading scale (Theorem~\ref{cr:theorem}). A translation model shows that this threshold is sharp, while a translation-invariant spatial law can cancel the response. The rank error has another transition, at $\rho=1-\tfrac{\alpha}{2}$: below it the limiting correction is stable and shares the leading signal's Poisson atoms. At and above it the correction is independent Brownian noise, with a logarithmic normalization at equality. Finally, we compare rank and clock couplings of the same microscopic translation model. They have identical particle laws for each $N$, but different joint realizations with the limiting signal. Rank coupling gives a functional Gaussian limit at the limiting center. Clock coupling gives the same finite-dimensional limit, yet displaced jumps prevent every diverging additive path normalization there. Centering at the finite-system signal restores a functional Gaussian limit (Proposition~\ref{clk:same-law}).

\begin{table}[H]
\centering
\caption{The two thresholds for the rank error.}
\label{tab:boundaries}
\small
\vspace{5pt}
\begin{tabularx}{\linewidth}{@{}Y l Y@{}}
\toprule
Comparison & Boundary & Criterion\\
\midrule
Rank error versus sampling & $\rho=\tfrac{\alpha}{2}$, $\alpha<1$
 & $N^{-\frac{\rho}{\alpha}}=N^{-\frac{1}{2}}$.\\
Stable versus Brownian rank tangent & $\rho=1-\tfrac{\alpha}{2}$
 & Square integrability of the limiting quantile correction at large ranks.\\
\bottomrule
\end{tabularx}
\end{table}

\paragraph{Related work.} Stable common-noise limits and conditional propagation of chaos are established in \citet{LocherbachLoukianova2025Positive,LocherbachLoukianovaMarini2025Strong,LocherbachLoukianovaMarini2025Weak}. The strong result gives a quantitative coupling \citep[Theorem~2.14]{LocherbachLoukianovaMarini2025Strong}. The weak result uses a time-changed random walk. We normalize the empirical residual under a common Poisson master. Corollary~\ref{rk:exclusion} covers source exclusion. Conditional Gaussian fluctuations with a common factor are studied in \citet{BudhirajaKiraSaha2017,DelarueLackerRamanan2019,Nikolaev2026}. \citet{BudhirajaKiraSaha2017} allow empirical feedback in the factor, with a common jump intensity independent of the empirical measure \citep[Condition~2.3]{BudhirajaKiraSaha2017}, finite mark measures, moment conditions and nondegenerate private diffusion. Our selection has infinite stable activity and allows degenerate private diffusion. \citet{Nikolaev2026} treats Brownian common noise with less regular interactions. We use the Hilbert-space sampling method of \citet{FernandezMeleard1997,KurtzXiong2004}, then control acceptance errors by a positive-strip estimate and microscopic errors by quantile estimates. Both steps preserve convergence in probability of the conditional laws given the full master. Rank representations and stable invariance principles are developed in \citet{FergusonKlass1972,LePageWoodroofeZinn1981,LePagePodgorskiRyznar1997,Rosinski1990}, and Gaussian small-jump approximations in \citet{AsmussenRosinski2001}. Our rank tangent describes the residual of a fixed quantile coupling jointly with its leading signal. The clock analysis combines martingale central limit theorems and time-change theory \citep{EthierKurtz1986,Rebolledo1980,Whitt2002} to identify additive second-order errors and the obstruction to their path convergence.

\paragraph{Organization.} Section~\ref{mf:section} gives the conditional sampling limit under Poisson thinning. Section~\ref{rk:section} treats rank coupling, microscopic transfer and the endogenous response. Section~\ref{clk:section} studies clock coupling and compares the two constructions. Appendices~\ref{app:sampling}, \ref{app:rank} and~\ref{app:clock} contain the proofs and results.

\subsection{Notation and function spaces}
Each construction is defined on a probability space $(\Omega,\cF,\Pp)$. We write $\Law(X)$ for the law of a random variable $X$. For a Polish space $E$, $\Prob(E)$ is the set of Borel probability measures on $E$. For a finite signed Borel measure $\mu$, $|\mu|$ denotes its total variation measure. If $\phi:E\to\R$ is $|\mu|$-integrable, write $\mu(\phi)\coloneqq\int_E\phi\dd\mu$. We use the same notation for distributional pairings when the test function is in the domain of the distribution. Set $\N_0\coloneqq\{0,1,\ldots\}$ and $\N\coloneqq\N_0\setminus\{0\}$. For $d\in\N$, the torus is the quotient space $\mathbf T^d\coloneqq\R^d/\Z^d$, and $\mathbf T\coloneqq\mathbf T^1$. For a Euclidean domain or torus $U$ and $k\in\N_0$, we use the real-valued spaces $C^k(U)$ and $C^\infty(U)$. In $C_b^k(U)$, the function and its derivatives through order $k$ are bounded. The $C^k$ norm is the maximum of their supremum norms. Vector-valued spaces are defined componentwise, using the Euclidean norm. For a normed space $E$ and $T>0$, $C([0,T];E)$ and $D([0,T];E)$ are the spaces of continuous and c\`adl\`ag paths, respectively. We write $$\|x\|_{\infty,E}\coloneqq\sup_{t\in[0,T]}\|x_t\|_E.$$ For $1\le p\le\infty$, we use the spaces $L^p(\Pp;E)$ with measure $\Pp$ and $L^p(I;E)$ with Lebesgue measure on an interval $I\subset\R$. We omit subscripts when the space is clear. An $L^1$ label above an arrow denotes convergence in the $L^1$ norm. For the Skorokhod $J_1$ topology on $D([0,T];E)$, $x_n\xrightarrow[n\to\infty]{J_1}x$ if and only if there are increasing homeomorphisms $\lambda_n:[0,T]\to[0,T]$ such that

\[
 \|\lambda_n-\mathrm{id}\|_\infty+
 \|x_n\circ\lambda_n-x\|_{\infty,E}
 \xrightarrow[n\to\infty]{}0.
\]

Vector $J_1$ uses one time change for all coordinates. For a c\`adl\`ag path, $x_{t-}\coloneqq\lim_{u\uparrow t}x_u$ is its left limit, and $\Delta x_t\coloneqq x_t-x_{t-}$ for $t\in(0,T]$. We write $\Longrightarrow$ for convergence in law and $\xrightarrow{\mathbf{P}}$ for convergence in probability. For deterministic $b_N>0$, $X_N=O_{\Pp}(b_N)$ means that the family $\left(\tfrac{\|X_N\|}{b_N}\right)_{N\in\N}$ is tight. The notation $X_N=o_{\Pp}(b_N)$ means that this ratio converges to zero in probability as $N\to\infty$. Constants $C$ can change between estimates. They are independent of $N$ and carry subscripts for relevant dependencies.

\section{A conditional central limit theorem under Poisson thinning}
\label{mf:section}
\subsection{Model and conditioning}
Fix $0<T<\infty$, integers $d,d_0,q\ge1$, and $\bar f>0$. Equip $\mathbf T^d$ with its geodesic distance $d_{\mathbf T^d}$. For $\phi:\mathbf T^d\to\R$ and $\mu,\nu\in\Prob(\mathbf T^d)$, set $$\|\phi\|_{\mathrm{Lip}} \coloneqq\sup_{\substack{x,y\in\mathbf T^d\\x\ne y}} \frac{|\phi(x)-\phi(y)|}{d_{\mathbf T^d}(x,y)},$$ and $$W_1(\mu,\nu)\coloneqq \sup_{\substack{\phi:\mathbf T^d\to\R\\\|\phi\|_{\mathrm{Lip}}\le1}} (\mu-\nu)(\phi).$$ Let $\sigma\in\R^{d\times d_0}$ and $a\coloneqq\sigma\sigma^\top$. The drift depends on finitely many moments:

\[
 m:\left|\begin{aligned}
 \Prob(\mathbf T^d)&\longrightarrow\R^q\\
 \mu&\longmapsto(\mu(g_1),\ldots,\mu(g_q)).
 \end{aligned}\right.
\]

\begin{assumption}[Smooth coefficients]\label{mf:assumptions}
Fix $s>\tfrac d2+3$. The function $B:\mathbf T^d\times\R^q\to\R^d$ and its derivatives through order $\lceil s\rceil+2$ are bounded. Also, $g_1,\ldots,g_q,f\in C^{\lceil s\rceil+2}(\mathbf T^d)$ and $0\le f\le\bar f$.
\end{assumption}
Fix $0<\alpha<1$ and $c_{d,\alpha}>0$, and define $\nu_\alpha(\dd z)\coloneqq c_{d,\alpha}|z|^{-d-\alpha}\dd z$ for $z\in\R^d\setminus\{0\}$. Let $\mathcal N^0$ be a Poisson random measure on $[0,\infty)\times(\R^d\setminus\{0\})\times[0,\bar f]$ with intensity $\dd t\,\nu_\alpha(\dd z)\dd u$. Its completed right-continuous natural filtration $(\cG_t^0)_{t\ge0}$ records every candidate, including rejected ones. Let $(X_0^i)_{i\ge1}$ be i.i.d. with law $\mu_0\in\Prob(\mathbf T^d)$, and let $(W^i)_{i\ge1}$ be independent standard $\R^{d_0}$-valued Brownian motions. The sample, Brownian motions, and master $\mathcal N^0$ are mutually independent. Their usual augmented joint filtration $(\cF_t)_{t\ge0}$ preserves the Poisson compensator. For $N\ge1$ and $t\in[0,T]$, define the empirical measure by $$\mu_t^N\coloneqq\tfrac1N\sum_{i=1}^N\delta_{X_t^{i,N}}.$$ The particle system, with addition modulo $\Z^d$, is
\begin{subequations}\label{mf:model}

\begin{equation}
 \begin{aligned}
 X_t^{i,N}
 &=X_0^i+\int_0^t B(X_r^{i,N},m(\mu_r^N))\dd r
   +\sigma W_t^i+L_t^N,
 \end{aligned}
\end{equation}

where

\begin{equation}
 L_t^N\coloneqq\int_{[0,t]\times(\R^d\setminus\{0\})\times[0,\bar f]}
 z\ind_{\{u\le\mu_{r-}^N(f)\}}
 \mathcal N^0(\dd r,\dd z,\dd u)
\end{equation}

\end{subequations}
for $t\in[0,T]$ and $i\in\{1,\ldots,N\}$. Since $$\int_{\{z\in\R^d:0<|z|\le1\}}|z|\nu_\alpha(\dd z)<\infty$$ and $\nu_\alpha(\{z\in\R^d:|z|>1\})<\infty$, the mark integrals are uncompensated finite-variation sums. For an independent copy $(\bar X_0,W)$ of the private input, the reference flow satisfies $\mu_t=\Law(\bar X_t\mid\cG_t^0)$ for $t\in[0,T]$, where
\begin{subequations}\label{mf:reference}

\begin{equation}
 \begin{aligned}
 \bar X_t
 &=\bar X_0+\int_0^t B(\bar X_r,m(\mu_r))\dd r
   +\sigma W_t+L_t,
 \end{aligned}
\end{equation}

with

\begin{equation}
 L_t\coloneqq\int_{[0,t]\times(\R^d\setminus\{0\})\times[0,\bar f]}
 z\ind_{\{u\le\mu_{r-}(f)\}}
 \mathcal N^0(\dd r,\dd z,\dd u).
\end{equation}

\end{subequations}
Both identities hold for $t\in[0,T]$. Include the private copy in $\cF_t$. The signals $L^N,L$ take values in $\R^d$ and the particle states in $\mathbf T^d$. Proposition~\ref{mf:wellposed} in Appendix~\ref{app:sampling} gives existence and uniqueness for both systems.
\subsection{Fluctuations and their conditional law}
Let $\mathcal D'(\mathbf T^d)$ be the continuous dual of $C^\infty(\mathbf T^d)$, the space of real distributions. For $\beta\in\R$, set
\[
 H^\beta\coloneqq
 \left\{\xi\in\mathcal D'(\mathbf T^d):
 \sum_{k\in\Z^d}(1+|k|^2)^\beta|\xi(e_k)|^2<\infty\right\},
\]
with $$\|\xi\|_{H^\beta}^2\coloneqq \sum_{k\in\Z^d}(1+|k|^2)^\beta|\xi(e_k)|^2.$$ Here $e_k:\mathbf T^d\to\mathbb C$ is defined by $e_k(x)\coloneqq e^{2\pi i k\cdot x}$ for $k\in\Z^d$. Fourier coefficients use complexification. The spaces and fields are real. Write $\xi(\phi)$ for the $H^{-\beta}$--$H^\beta$ pairing. The moment map extends linearly to $H^{-s}$ by $m(\xi)\coloneqq(\xi(g_1),\ldots,\xi(g_q))$. For $z\in\R^d$, translation is the isometry
\[
 \tau_z:\left|\begin{aligned}
 H^\beta&\longrightarrow H^\beta\\
 \xi&\longmapsto[\phi\longmapsto\xi(\phi(\cdot+z))].
 \end{aligned}\right.
\]
The bracketed map acts on $\phi\in H^{-\beta}$. For $\mu\in\Prob(\mathbf T^d)$, define $\mathcal A_\mu,\mathcal B_\mu:H^s\to H^{s-2}$ by 
\[
\mathcal A_\mu\phi\coloneqq B(\cdot,m(\mu))\cdot\nabla\phi+\frac12\operatorname{tr}\left(aD^2\phi\right),\qquad 
 \mathcal B_\mu\phi\coloneqq\mathcal A_\mu\phi+
\sum_{j=1}^qg_j\,\mu\left(\partial_{m_j}B(\cdot,m(\mu))\cdot\nabla\phi\right).
\]

Both operator norms are bounded uniformly in $\mu\in\Prob(\mathbf T^d)$. For locally square-integrable real martingales $M,N$, $\ip{M}{N}$ denotes predictable covariation, and $\langle M\rangle\coloneqq\ip{M}{M}$. For a Polish space $E$ with bounded compatible metric $d_E$ and a function $F:E\to\R$, define $\|F\|_\infty\coloneqq\sup_{x\in E}|F(x)|$ and $$\operatorname{Lip}_{d_E}(F)\coloneqq \sup_{\substack{x,y\in E\\x\ne y}} \frac{|F(x)-F(y)|}{d_E(x,y)}.$$ The bounded-Lipschitz distance is

\[
 d_{\mathrm{BL}}(P,Q)\coloneqq
 \sup_{\substack{F:E\to\R\\
 \|F\|_\infty+\operatorname{Lip}_{d_E}(F)\le1}}
 |P(F)-Q(F)|,\qquad P,Q\in\Prob(E).
\]

Convergence of conditional kernels $K_N,K$ given $\cG_T^0$ in this distance, in probability, implies $$\Ee[ZK_N(F)]\xrightarrow[N\to\infty]{}\Ee[ZK(F)]$$ for every bounded $\cG_T^0$-measurable $Z$ and bounded Lipschitz $F:E\to\R$. This is stable convergence \citep{AldousEagleson1978}. The label $\mathrm{st}(\cG_T^0)$ above an arrow denotes stable convergence relative to $\cG_T^0$, meaning joint convergence in law with every bounded $\cG_T^0$-measurable random variable.
\begin{theorem}[Conditional functional fluctuations]\label{mf:clt}
Under Assumption~\ref{mf:assumptions}, let $0<\alpha<1$ and set $\eta^N\coloneqq\sqrt N(\mu^N-\mu)$. Then

\begin{equation}\label{mf:stable-clt}
 \eta^N\xRightarrow[N\to\infty]{\mathrm{st}(\cG_T^0)}\eta
 \quad\text{in }D([0,T];H^{-s})
\end{equation}

with the Skorokhod $J_1$ topology. Conditionally on $\cG_T^0$, the limit is centered Gaussian, starts from the $\mu_0$-Brownian bridge, and satisfies, for $t\in[0,T]$ and $\phi\in C^\infty(\mathbf T^d)$,

\begin{equation}\label{mf:limit-equation}
 \begin{aligned}
 \eta_t(\phi)
 =\ &\eta_0(\phi)+M_t(\phi)
   +\int_0^t\eta_r(\mathcal B_{\mu_r}\phi)\dd r\\
 &+
 \int_{[0,t]\times(\R^d\setminus\{0\})\times[0,\bar f]}
 \ind_{\{u\le\mu_{r-}(f)\}}
 \eta_{r-}\left(\phi(\cdot+z)-\phi\right)
 \mathcal N^0(\dd r,\dd z,\dd u).
 \end{aligned}
\end{equation}

Here $M$ is a continuous, conditionally centered Gaussian $H^{-s}$-valued martingale, conditionally independent of $\eta_0$, with

\begin{equation}\label{mf:bracket}
 \ip{M(\phi)}{M(\psi)}_t
 =\int_0^t\mu_r(\nabla\phi^\top a\nabla\psi)\dd r,
 \qquad \phi,\psi\in C^\infty(\mathbf T^d).
\end{equation}

The jump sum is absolutely convergent for smooth tests. For almost every master, \eqref{mf:limit-equation} has a unique law among $H^{-s}$-valued c\`adl\`ag solutions with this joint conditional law of $(\eta_0,M)$. If $K_N$ and $K$ are the conditional laws of $\eta^N$ and $\eta$ given $\cG_T^0$, then

\begin{equation}\label{mf:kernel-convergence}
 d_{\mathrm{BL}}(K_N,K)\xrightarrow[N\to\infty]{\mathbf{P}}0.
\end{equation}

Let $\widetilde\mu^N$ be the empirical law with its own empirical drift, the same initial sample and Brownian motions, and the prescribed reference signal $L$. Then

\begin{equation}\label{mf:signal-rate}
 \begin{aligned}
 \sup_{t\in[0,T]}|L_t^N-L_t|
 +\sup_{t\in[0,T]}W_1(\mu_t^N,\widetilde\mu_t^N) &=O_{\Pp}\left(N^{-\frac1{2\alpha}}\right).
\end{aligned}
\end{equation}

\end{theorem}
The initial bridge has covariance

\[
 \Ee[\eta_0(\phi)\eta_0(\psi)]
 =\mu_0(\phi\psi)-\mu_0(\phi)\mu_0(\psi),
 \qquad \phi,\psi\in C^\infty(\mathbf T^d).
\]

An accepted atom $(t,z,u)$ translates the field by $\eta_t=\tau_z\eta_{t-}$. A rejected atom leaves it unchanged. The limit contains no derivative of the acceptance intensity.
\paragraph{Proof idea.} Split the fluctuation into sampling with the reference signal and the effect of differing acceptance labels: $\mu^N-\mu=(\widetilde\mu^N-\mu)+(\mu^N-\widetilde\mu^N)$. With the master fixed, translation by $-L$ gives continuous interacting diffusions. Their empirical fields satisfy the Gaussian sampling limit in Lemma~\ref{mf:prescribed}. The threshold error is of order $N^{-\frac12}$, and the positive-strip estimate bounds the total size of disagreement jumps by $O_{\Pp}\left(N^{-\frac1{2\alpha}}\right)$. Since $\alpha<1$, their contribution vanishes after multiplication by $\sqrt N$. Conditional coupling then gives \eqref{mf:kernel-convergence}. Appendix~\ref{app:sampling} gives the proofs and covariance formula \eqref{mf:covariance}.

\section{Rank coupling and microscopic transfer}\label{rk:section}
Matching jumps by tail rank identifies the common-signal error and its contribution to empirical fluctuations.
\subsection{An inverse-tail criterion and the rank trichotomy}\label{rk:trichotomy-section}
Fix $0<\alpha<2$ for this subsection. Let $R$ be a nonnegative random variable and let $e$ be independent of $R$ with law $\nu_e\coloneqq \tfrac{\delta_{-1}+\delta_1}{2}$. Set $U\coloneqq eR$. Assume that the tail function $\overline F:[0,\infty)\to[0,1]$ satisfies

\begin{equation*}
 \overline F(x)\coloneqq \Pp(R>x)
 =cx^{-\alpha}\{1+dx^{-\rho}+o(x^{-\rho})\},
 \qquad x\to\infty,
\end{equation*}

where $c>0$, $\rho>0$, and $d\in\R$. Define the generalized tail inverse by

\[
 Q:\left|\begin{aligned}
 (0,\infty)&\longrightarrow [0,\infty)\\
 v&\longmapsto 
 \begin{cases}
 \inf\{x\ge0:\overline F(x)\le v\},&0<v<1,\\
 0,&v\ge1.
 \end{cases}
 \end{aligned}\right.
\]

For $N\in\N$, define $g_N:(0,\infty)\to[0,\infty)$, $g:(0,\infty)\to(0,\infty)$, and $h:(0,\infty)\to\R$ by
\begin{equation}\label{rk:maps}
g_N(y)\coloneqq N^{-\frac{1}{\alpha}}Q\left(\frac{y}{N}\right), \qquad g(y)\coloneqq \left(\frac{c}{y}\right)^{\frac{1}{\alpha}}, \quad\text{and}\quad h(y)\coloneqq \frac{d}{\alpha}\left(\frac{c}{y}\right)^{\frac{1-\rho}{\alpha}}.
\end{equation}
Let $\Pi$ be a Poisson measure on $[0,\infty)\times(0,\infty)\times\{-1,1\}$ with intensity $\dd t\,\dd y\,\nu_e(\dd e)$, and let $\widetilde\Pi\coloneqq \Pi-\dd t\,\dd y\,\nu_e(\dd e)$. The sigma-field $\sigma(\Pi)$ is generated by the counts of $\Pi$ on all Borel sets of finite intensity. For every finite-valued Borel function $a:(0,\infty)\to\R$ with $\int_1^\infty a(y)^2\dd y<\infty$, define the càdlàg process $I(a)\in D([0,T];\R)$ by

\[
\begin{aligned}
 I_t(a)&\coloneqq {}\int_{(0,t]\times(0,1]\times\{-1,1\}}
            ea(y)\,\Pi(\dd v,\dd y,\dd e) +\underset{M\to\infty}{L^2\text{-}\lim}
     \int_{(0,t]\times(1,M]\times\{-1,1\}}
            ea(y)\,\widetilde\Pi(\dd v,\dd y,\dd e).
\end{aligned}
\]

The first integral is a finite sum. Doob's inequality makes the second Cauchy for the norm $(\Ee[\sup_{0\le t\le T}|\cdot|^2])^{\frac{1}{2}}$, so its limit has a càdlàg version. The fixed cutoff makes $I$ linear. Set $J^N\coloneqq I(g_N)$ and $S\coloneqq I(g)$. For every bounded Borel $\varphi:\R\to\R$,

\[
 \int_0^N\int_{\{-1,1\}}\varphi(eg_N(y))\nu_e(\dd e)\dd y
 =N\Ee\left[\varphi\left(N^{-\frac{1}{\alpha}}U\right)\right].
\]

Thus $J^N$ is a rate-$N$ compound Poisson process, allowing zero marks, and $S$ is symmetric $\alpha$-stable. Its Lévy measure $\nu_S$ satisfies $\nu_S(\{z\in\R:|z|>x\})=cx^{-\alpha}$ for $x>0$, and its characteristic exponent is $$\log\Ee[e^{iuS_t}]=-tc\alpha|u|^\alpha\int_0^\infty(1-\cos z)z^{-1-\alpha}\dd z$$ for $t\ge0$ and $u\in\R$. The integral is finite and positive.
\begin{assumption}[Inverse-tail control]\label{rk:inverse}
Put $r_N\coloneqq N^{-\frac{\rho}{\alpha}}$ and $\rho_c\coloneqq 1-\tfrac{\alpha}{2}$. Assume $r_N^{-1}(g_N-g)\xrightarrow[N\to\infty]{} h$ uniformly on each compact subset of $(0,\infty)$, together with the condition corresponding to $\rho$:
\begin{enumerate}[label=(\roman*)]
\item If $\rho<\rho_c$, for every $R>0$,

\begin{equation}\label{rk:inverse-l2}
 \int_R^\infty|r_N^{-1}(g_N-g)-h|^2\dd y\xrightarrow[N\to\infty]{}0.
\end{equation}

\item If $\rho=\rho_c$, put $q_N\coloneqq r_N\sqrt{\log N}$ and $\sigma_c^2\coloneqq \tfrac{cd^2}{\alpha^2}$ for $N\ge2$. For every $R>0$ and $\varepsilon>0$,

\begin{equation}\label{rk:critical-bracket}
 q_N^{-2}\int_R^\infty(g_N-g)^2\dd y\xrightarrow[N\to\infty]{}\sigma_c^2,
\end{equation}

and

\begin{equation}\label{rk:critical-lindeberg}
 q_N^{-2}\int_R^\infty(g_N-g)^2
  \ind_{\{|g_N-g|>\varepsilon q_N\}}\dd y\xrightarrow[N\to\infty]{}0.
\end{equation}

\item If $\rho>\rho_c$, the Borel function

\[
 k:\left|\begin{aligned}
 (0,\infty)&\longrightarrow \R\\
 v&\longmapsto Q(v)-\left(\frac{c}{v}\right)^{\frac{1}{\alpha}}
 \end{aligned}\right.
\]

belongs to $L^2((0,\infty),\dd v)$.
\end{enumerate}
\end{assumption}
Assumption~\ref{rk:inverse} follows from the Hall expansion \eqref{rk:strong-hall} and eventual continuity and strict decrease of the tail. See Lemma~\ref{rk:hall-sufficient} in Appendix~\ref{app:rank-inverse}.
\begin{theorem}[Rank tangent]\label{rk:tangent}
Under Assumption~\ref{rk:inverse}, the following limits hold jointly with $S$ in $D([0,T];\R^2)$ with the vector $J_1$ topology.
\begin{enumerate}[label=(\roman*)]
\item If $0<\rho<\rho_c$, let $Z\coloneqq I(h)$. Then

\begin{equation}\label{rk:stable-tangent}
 \sup_{0\le t\le T}|r_N^{-1}(J_t^N-S_t)-Z_t|
          \xrightarrow[N\to\infty]{\mathbf{P}}0.
\end{equation}

For $d\ne0$, $Z$ is a nondegenerate symmetric stable process with index $\kappa\coloneqq \tfrac{\alpha}{1-\rho}\in(\alpha,2)$ and Lévy measure $\nu_Z(\dd z)=\frac{c\kappa}{2}\left|\frac{d}{\alpha}\right|^\kappa |z|^{-1-\kappa}\dd z$. If $d=0$, $Z=0$ at this scale.
\item If $\rho=\rho_c$, then $(S,q_N^{-1}(J^N-S))\xRightarrow[N\to\infty]{}(S,\sigma_c B^{\rm rk})$, where $B^{\rm rk}$ is standard Brownian motion independent of $S$.
\item If $\rho>\rho_c$, then $\left(S,N^{\frac{1}{\alpha}-\frac{1}{2}}(J^N-S)\right)\xRightarrow[N\to\infty]{}(S,\sigma_Q B^{\rm rk})$, where $\sigma_Q^2\coloneqq \int_0^\infty k(v)^2\dd v$.
\end{enumerate}
In all three regimes, $\|J^N-S\|_\infty=O_{\Pp}(b_N)$, where

\begin{equation}\label{rk:rate}
 b_N\coloneqq 
 \begin{cases}
 N^{-\frac{\rho}{\alpha}},&\rho<1-\frac{\alpha}{2},\\
 N^{-\frac{2-\alpha}{2\alpha}}\sqrt{1+\log N},&\rho=1-\frac{\alpha}{2},\\
 N^{-\frac{2-\alpha}{2\alpha}},&\rho>1-\frac{\alpha}{2}.
 \end{cases}
\end{equation}

\end{theorem}
\begin{remark*}
In the first regime, $Z$ shares the Poisson atoms of $S$. The variance $\sigma_Q^2$ depends on the full quantile, whereas the critical variance $\sigma_c^2$ depends only on the Hall coefficients.
\end{remark*}
\paragraph{Proof idea.} On each fixed rank block, inverse-tail convergence gives convergence at the finitely many Poisson atoms. The remaining ranks converge in mean square below the rank threshold and satisfy a martingale central limit theorem at and above it. Splitting the ranks also proves independence of the Brownian limits from $S$. Appendix~\ref{app:rank-tangent} gives the proof. For $\alpha<1$, $b_N=o\left(N^{-\frac{1}{2}}\right)$ if and only if $\rho>\tfrac{\alpha}{2}$. Since $\tfrac{\alpha}{2}<1-\tfrac{\alpha}{2}$, the sampling threshold lies inside the stable rank regime.
\subsection{Nearly stable microscopic marks}\label{rk:microscopic-section}
Let $0<\alpha<1$ and impose Assumption~\ref{mf:assumptions} in dimension one. Let $U$ be symmetric, $c\coloneqq \nu_\alpha(\{z\in\R:|z|>1\})$, and assume

\begin{equation}\label{rk:tail-bound}
 \Pp(|U|>x)=cx^{-\alpha}\{1+O(x^{-\rho})\},
                 \qquad x\to\infty,
\end{equation}

for some $\rho>0$. Atoms are allowed. Use $Q,g_N,g$ from \eqref{rk:maps} and the Poisson master $\mathcal P^0$ on $[0,\infty)\times(0,\infty)\times\{-1,1\}\times[0,\bar f]$ with intensity $\dd t\,\dd y\,\nu_e(\dd e)\dd u$. The bijection

\[
 \left|\begin{aligned}
 (0,\infty)\times\{-1,1\}&\longrightarrow \R\setminus\{0\}\\
 (y,e)&\longmapsto e\left(\frac{c}{y}\right)^{\frac{1}{\alpha}}
 \end{aligned}\right.
\]

has inverse $z\mapsto(c|z|^{-\alpha},\operatorname{sgn}z)$ and pushes $\dd y\,\nu_e(\dd e)$ to $\nu_\alpha$. It therefore preserves the master sigma-field $\mathcal G_T^0$. The image of the restriction to $0<y<N$ under $(y,e)\mapsto eg_N(y)$ is

\begin{equation}\label{rk:micro-measure}
 \nu_N(\dd z)\coloneqq N\Law\left(N^{-\frac{1}{\alpha}}U\right)(\dd z),
\end{equation}

with zero marks irrelevant. Since $g_N(y)=0$ for $y\ge N$, this restriction does not change the microscopic signal below. In \eqref{mf:model}, retain the initial sample, Brownian motions and drift, and replace the signal by

\begin{equation}\label{rk:micro-signal}
 L_t^{N,\mathrm{mic}}
 \coloneqq \int_{\substack{(0,t]\times(0,\infty)\\\times\{-1,1\}\times[0,\bar f]}}
 eg_N(y)\ind_{\{u\le\mu_{v-}^{N,\mathrm{mic}}(f)\}}
              \mathcal P^0(\dd v,\dd y,\dd e,\dd u),
\end{equation}

where $\mu_t^{N,\mathrm{mic}}\coloneqq N^{-1}\sum_{i=1}^N\delta_{X_t^{i,N,\mathrm{mic}}} \in\Prob(\mathbf T)$. Its accepted rate is $N\mu_{t-}^{N,\mathrm{mic}}(f)$ and its marks have law $\Law\left(N^{-\frac{1}{\alpha}}U\right)$. The stable-master system uses $g$ and its own threshold. Both signals belong to $D([0,T];\R)$.
\begin{theorem}[Microscopic transfer]\label{rk:transfer}
Suppose Assumption~\ref{mf:assumptions} holds in dimension one and \eqref{rk:tail-bound} holds with $\rho>\tfrac{\alpha}{2}$. Under the coupling above,

\begin{equation}\label{rk:comparison-rate}
\begin{aligned}
 \sup_{0\le t\le T}W_1(\mu_t^{N,\mathrm{mic}},\mu_t^N)
   +\sup_{0\le t\le T}|L_t^{N,\mathrm{mic}}-L_t^N| &=O_{\Pp}(b_N)=o_{\Pp}\left(N^{-\frac{1}{2}}\right).
\end{aligned}
\end{equation}

Consequently, with the limit of Theorem~\ref{mf:clt}, $\sqrt N(\mu^{N,\mathrm{mic}}-\mu) \xRightarrow[N\to\infty]{\mathrm{st}(\mathcal G_T^0)}\eta$. This convergence holds in $D([0,T];H^{-s}(\mathbf T))$ with $J_1$. The conditional laws given $\mathcal G_T^0$ converge in probability to the same Gaussian kernel.
\end{theorem}
\paragraph{Proof idea.} Lemma~\ref{rk:localized} bounds the squared rank error after removing a block of large marks. The acceptance-strip estimate and synchronous particle coupling then give \eqref{rk:comparison-rate}. Sobolev duality transfers the empirical limit. See Appendix~\ref{app:rank-microscopic}. The same conclusions hold when each accepted event excludes one source chosen proportionally to its rate. Corollary~\ref{rk:exclusion} gives the construction and its $O_{\Pp}(N^{-1})$ comparison bound.
\subsection{The critical and subcritical endogenous regimes}\label{cr:section}
Work in dimension one under Assumption~\ref{mf:assumptions}, so $s>\tfrac{7}{2}$. Assume \eqref{rk:strong-hall} and the eventual continuity and strict monotonicity in Lemma~\ref{rk:hall-sufficient}, with $0<\alpha<1$ and $0<\rho\le\tfrac{\alpha}{2}$. Set $r_N\coloneqq N^{-\frac{\rho}{\alpha}}$. Use the microscopic coupling \eqref{rk:micro-signal} and reference $(\mu,L)$ from \eqref{mf:reference}. For $t\in(0,T]$, set

\[
 \iota_t:\left|\begin{aligned}
 [0,\bar f]&\longrightarrow \{0,1\}\\
 u&\longmapsto \ind_{\{u\le\mu_{t-}(f)\}}.
 \end{aligned}\right.
\]

With $h$ from \eqref{rk:maps} and compensated master $\widetilde{\mathcal P}^0\coloneqq \mathcal P^0-\dd t\dd y\nu_e(\dd e)\dd u$, define

\begin{equation*}
\begin{aligned}
Z_t\coloneqq {}&\int_{\substack{(0,t]\times(0,1]\times\{-1,1\}\times[0,\bar f]}}
             eh(y)\iota_v(u)\,\mathcal P^0(\dd v,\dd y,\dd e,\dd u)
 \\
 &\qquad\qquad+\underset{M\to\infty}{L^2\text{-}\lim}
   \int_{\substack{(0,t]\times(1,M]\\\times\{-1,1\}\times[0,\bar f]}}
             eh(y)\iota_v(u)\,\widetilde{\mathcal P}^{0}
                              (\dd v,\dd y,\dd e,\dd u).
\end{aligned}
\end{equation*}

Since $\rho\le\tfrac{\alpha}{2}<1-\tfrac{\alpha}{2}$, $\int_1^\infty h(y)^2\dd y =\frac{d^2c^{\frac{2(1-\rho)}{\alpha}}} {\alpha(2-2\rho-\alpha)}<\infty$, where $d$ is the tail coefficient in \eqref{rk:strong-hall}. Thus $Z\in D([0,T];\R)$ is adapted and master-measurable. The process $Z$ need not be Lévy or have finite variation. Let $\Ee_0$ denote expectation over the private initial values and Brownian motions with the full master fixed. For a reference particle $\bar X$, define $V\in D([0,T];\R)$ and $\mathbf c\coloneqq (c_1,\ldots,c_q)\in D([0,T];\R^q)$ by
\begin{subequations}\label{cr:response}

\begin{equation}
\begin{aligned}
 V_t&={}Z_t+\int_0^t B_x(\bar X_u,m(\mu_u))V_u\dd u +\sum_{j=1}^q\int_0^t B_{m_j}(\bar X_u,m(\mu_u))c_j(u)\dd u,
\end{aligned}
\end{equation}

where

\begin{equation}
 c_j(u)\coloneqq \Ee_0[g_j'(\bar X_u)V_u],
 \qquad j\in\{1,\ldots,q\}.
\end{equation}

\end{subequations}
Here $B_x\coloneqq \partial_xB$ and $B_{m_j}\coloneqq \partial_{m_j}B$. Define

\[
 \chi_t:\left|\begin{aligned}
 H^s(\mathbf T)&\longrightarrow \R\\
 \phi&\longmapsto \Ee_0\left[\phi'(\bar X_t)V_t\right].
 \end{aligned}\right.
\]

The response equation has a unique solution and defines a master-measurable field $\chi\in D([0,T];H^{-s}(\mathbf T))$, with $\|V\|_\infty+\|\mathbf c\|_\infty+\|\chi\|_{\infty,H^{-s}}\le C_T\|Z\|_\infty$. Appendix~\ref{app:rank-response} gives the proofs and response estimates.
\begin{theorem}[Endogenous rank response]\label{cr:theorem}
Under the preceding assumptions,

\begin{equation}\label{cr:signal-limit}
 \left\|r_N^{-1}(L^{N,\mathrm{mic}}-L)-Z\right\|_\infty
                                      \xrightarrow[N\to\infty]{\mathbf{P}}0.
\end{equation}

Let $\widetilde\mu^N$ be the auxiliary empirical law driven by the prescribed signal $L$, as in Theorem~\ref{mf:clt}. Then

\begin{equation}\label{cr:field-expansion}
 \sup_{0\le t\le T}\left\|
 \begin{aligned}
 r_N^{-1}(\mu_t^{N,\mathrm{mic}}-\mu_t)
 -\frac{N^{-\frac{1}{2}}}{r_N}\sqrt N(\widetilde\mu_t^N-\mu_t)-\chi_t
 \end{aligned}
 \right\|_{H^{-s}}\xrightarrow[N\to\infty]{\mathbf{P}}0.
\end{equation}

For $\rho<\tfrac{\alpha}{2}$, $r_N^{-1}(\mu^{N,\mathrm{mic}}-\mu)\xrightarrow[N\to\infty]{\mathbf{P}}\chi$ uniformly in $H^{-s}(\mathbf T)$. At $\rho=\tfrac{\alpha}{2}$,

\begin{equation}\label{cr:critical-limit}
 \sqrt N(\mu^{N,\mathrm{mic}}-\mu)
   \xRightarrow[N\to\infty]{\mathrm{st}(\mathcal G_T^0)}\eta+\chi,
\end{equation}

in $D([0,T];H^{-s}(\mathbf T))$ with $J_1$, where $\eta$ is the centered conditional Gaussian field of Theorem~\ref{mf:clt}. In both regimes, the conditional laws given the master converge in probability to the indicated kernels. At criticality this kernel is Gaussian with mean $\chi$ and the covariance of $\eta$. The field limits hold jointly with $(L,r_N^{-1}(L^{N,\mathrm{mic}}-L))$, whose limit is $(L,Z)$, with a common $J_1$ time change.
\end{theorem}
\paragraph{Proof idea.} The acceptance mismatch is smaller than $r_N$, so the signal has tangent $Z$. Differentiating the finite-particle system with respect to this prescribed signal gives the response $\chi$ with a quadratic remainder. A conditional law of large numbers for the particle responses yields \eqref{cr:field-expansion}. The prescribed-signal central limit theorem supplies the sampling term. See Appendix~\ref{app:rank-response}. If $\Pp(\chi_t\ne0)>0$ at a deterministic time $t>0$, the subcritical expansion shows that the root-$N$ field is not tight at that time. The next model has a nonzero response and a critical kernel, proving sharpness of $\rho>\tfrac{\alpha}{2}$ for the class. The correction vanishes if $d=0$ or $\chi=0$, as for a uniform spatial law under translation.
\subsection{Sharpness of the sampling threshold in a translation model}\label{rk:sharpness-section}
Constant rates make the response explicit. Let $0<\alpha<1$, $f\equiv\lambda>0$, $B=0$, and $\sigma=0$. Let $(Y_i)_{i\ge1}$ be independent with common law $\nu\in\Prob(\mathbf T)$, independent of $\Pi$. Assume $\nu(\dd x)=p(x)\dd x$ for a nonconstant $p\in C^\infty(\mathbf T)$. Set $\nu_N\coloneqq \tfrac1N\sum_{i=1}^N\delta_{Y_i}$ and $\mathcal B_N\coloneqq \sqrt N(\nu_N-\nu)$. Also set $L_t^N\coloneqq J_{\lambda t}^N$ and $L_t\coloneqq S_{\lambda t}$ for $t\in[0,T]$. Define
\begin{subequations}\label{rk:translated-model}

\begin{equation}
 X_t^{i,N}\coloneqq Y_i+L_t^N\pmod{\mathbf Z},
 \qquad i\in\{1,\ldots,N\}.
\end{equation}

The empirical and reference laws are

\begin{equation}
 \mu_t^N\coloneqq \tau_{L_t^N}\nu_N,
\end{equation}

and

\begin{equation}
 \mu_t\coloneqq \tau_{L_t}\nu.
\end{equation}

\end{subequations}
With $\partial\nu$ denoting the distributional derivative, $\partial_a(\tau_a\nu)=-\tau_a\partial\nu$. Fix $s>\tfrac{3}{2}$. Let $\mathcal B$ be the centered Gaussian $H^{-s}(\mathbf T)$-valued $\nu$-bridge, independent of $\Pi$. For $\varphi,\psi\in H^s(\mathbf T)$, it is characterized by

\begin{equation}\label{rk:bridge-covariance}
 \Ee\left[\mathcal B(\varphi)\mathcal B(\psi)\right]
       =\nu(\varphi\psi)-\nu(\varphi)\nu(\psi).
\end{equation}

Set $\Gamma_t\coloneqq \tau_{L_t}\mathcal B$ for all $\rho>0$. When $\rho\le\tfrac{\alpha}{2}$, let $Z$ be the process of Theorem~\ref{rk:tangent}(i) and define $\Xi_t\coloneqq -Z_{\lambda t}\tau_{L_t}\partial\nu$. When $d\ne0$, $Z$ is a nondegenerate symmetric stable process with index $\kappa=\tfrac{\alpha}{1-\rho}<2$.
\begin{corollary}[Three sampling regimes]\label{rk:sharpness}
Assume \eqref{rk:strong-hall} and the monotonicity hypothesis of Lemma~\ref{rk:hall-sufficient}, with $d\ne0$. For every $T<\infty$,

\begin{equation*}
 \begin{cases}
 N^{\frac{\rho}{\alpha}}(\mu^N-\mu)
       \xRightarrow[N\to\infty]{\mathrm{st}(\sigma(\Pi))}\Xi,
                                  &0<\rho<\frac{\alpha}{2},\\[2pt]
 \sqrt N(\mu^N-\mu)
       \xRightarrow[N\to\infty]{\mathrm{st}(\sigma(\Pi))}\Gamma+\Xi,
                                  &\rho=\frac{\alpha}{2},\\[2pt]
 \sqrt N(\mu^N-\mu)
       \xRightarrow[N\to\infty]{\mathrm{st}(\sigma(\Pi))}\Gamma,
                                  &\rho>\frac{\alpha}{2}.
 \end{cases}
\end{equation*}

These are functional limits in $D([0,T];H^{-s}(\mathbf T))$ with $J_1$, jointly with $L$ using a common time change. In each regime the conditional laws given $\sigma(\Pi)$ converge in probability to the indicated kernel. At criticality, conditionally on the master, the limit is a Gaussian field with mean $\Xi$ and the covariance of $\Gamma$. At every fixed $t>0$, $\|\Xi_t\|_{H^{-s}}>0$ almost surely in the first two regimes. In particular the root-$N$ field is not tight at that time if $\rho<\tfrac{\alpha}{2}$.
\end{corollary}
\paragraph{Proof idea.} Split the empirical error into a translated sample bridge and the displacement of the smooth reference density. Taylor expansion of the translation gives $\Xi$, and the rank rate determines which term survives. Appendix~\ref{app:rank-sharpness} proves the functional and conditional limits. For $\nu(\dd x)=(1+\varepsilon\cos(2\pi x))\dd x$, $0<|\varepsilon|<1$, the critical shift satisfies $\Xi_t(\sin(2\pi\cdot))=\pi\varepsilon Z_{\lambda t}\cos(2\pi L_t)$ and $\Xi_t(\cos(2\pi\cdot))=-\pi\varepsilon Z_{\lambda t}\sin(2\pi L_t)$. Thus these two tests detect every nonzero displacement. For uniform $\nu$, $\partial\nu=0$ and the shift vanishes.

\section{A clock-coupled counterpoint}\label{clk:section}
We compare two empirical centers for a finite-state population translated by a stable process at a random clock, then construct rank and grid couplings with identical microscopic laws. Their fluctuation limits depend on the coupling and on the chosen center.
\subsection{The population and its clock}
Fix $T>0$, an integer $m\ge2$, and distinct $y_1,\ldots,y_m\in\R$. Let $(e_a)_{a=1}^m$ be the standard basis of $\R^m$, and define $v_{ab}\coloneqq e_b-e_a$ for $1\le a,b\le m$, $a\ne b$. The state simplex is $\Delta_m\coloneqq\{p\in[0,1]^m:\sum_{a=1}^mp_a=1\}$, and its tangent space is $\mathsf T\coloneqq\{g\in\R^m:\sum_{a=1}^mg_a=0\}$. Let $\mathcal U\subset\R^m$ be an open neighborhood of $\Delta_m$, and let $q_{ab}\in C_b^2(\mathcal U;[0,\infty))$ for $1\le a,b\le m$, $a\ne b$. Define the event rates $r_{ab}:\Delta_m\to[0,\infty)$, the population drift $F:\Delta_m\to\mathsf T$, the total event rate $\lambda:\Delta_m\to[0,\infty)$ and the population covariance density $Q:\Delta_m\to\R^{m\times m}$ by the following formulas for $p\in\Delta_m$:
\begin{equation}\label{clk:rates}
\begin{alignedat}{2}
 r_{ab}(p)&\coloneqq p_aq_{ab}(p),
 &\qquad F(p)&\coloneqq\sum_{\substack{1\le a,b\le m\\a\ne b}}v_{ab}r_{ab}(p),\\
 \lambda(p)&\coloneqq\sum_{\substack{1\le a,b\le m\\a\ne b}}r_{ab}(p),
 &\qquad Q(p)&\coloneqq\sum_{\substack{1\le a,b\le m\\a\ne b}}v_{ab}v_{ab}^{\top}r_{ab}(p).
\end{alignedat}
\end{equation}
For $p_0\in\Delta_m$, let $p\in C^1([0,T];\Delta_m)$ solve $\dot p_t=F(p_t)$ with $p(0)=p_0$. Assume $\lambda_*\coloneqq\inf_{t\in[0,T]}\lambda(p_t)>0$. Given independent unit-rate Poisson processes $(P_{ab})_{1\le a,b\le m,\,a\ne b}$, define, for $N\in\N$ and $t\in[0,T]$,
\begin{subequations}\label{clk:chain}

\begin{equation}
 C_{ab}^N(t)\coloneqq P_{ab}\left(N\int_0^t r_{ab}(p_u^N)\dd u\right).
\end{equation}

These counts determine the population through

\begin{equation}
 p_t^N\coloneqq p_0^N+\frac1N\sum_{\substack{1\le a,b\le m\\a\ne b}}v_{ab}C_{ab}^N(t).
\end{equation}

The empirical event clock is

\begin{equation}
 T_t^N\coloneqq\frac1N\sum_{\substack{1\le a,b\le m\\a\ne b}}C_{ab}^N(t),
\end{equation}

and its deterministic limit is

\begin{equation}
 A_t\coloneqq\int_0^t\lambda(p_u)\dd u,\qquad t\in[0,T].
\end{equation}

\end{subequations}
Here $p_0^N\in\Delta_m\cap N^{-1}\Z^m$ is independent of the Poisson processes, and $G_0^N\coloneqq\sqrt N(p_0^N-p_0)\xRightarrow[N\to\infty]{}G_0$, where $G_0$ is a centered Gaussian vector in $\mathsf T$. Bounded rates ensure nonexplosion, and the factor $p_a$ keeps $p^N$ in $\Delta_m$. Derivatives of $F$ and $\lambda$ act on $\mathsf T$. Let $S$ be a symmetric $\alpha$-stable L\'evy process, $0<\alpha<2$, independent of the population array and its driving Poisson processes, with $\Ee[e^{iuS_t}]=e^{-t|u|^\alpha}$ for $t\ge0$ and $u\in\R$. We write $\sigma(S)\coloneqq\sigma(S_t:t\ge0)$ for the sigma-field generated by the whole path of $S$. Set $L_t^N\coloneqq S_{T_t^N},\, t\in[0,T]$, and define the limiting signal by $L_t\coloneqq S_{A_t},\, t\in[0,T]$. The grid increments are $U_k^N\coloneqq N^{\frac1\alpha} \left(S_{\frac{k}{N}}-S_{\frac{k-1}{N}}\right),\, k\in\N$. For each $N$, the $U_k^N$ are independent with law $S_1$, independent of the chain, and $L_t^N=N^{-\frac1\alpha}\sum_{k=1}^{NT_t^N}U_k^N$. We use the $J_1$ topology on all path spaces below unless stated otherwise.
\begin{lemma}[Joint population and clock fluctuations]\label{clk:joint-clt}
Set $G^N\coloneqq\sqrt N(p^N-p)$ and $H^N\coloneqq\sqrt N(T^N-A)$. Then

\[
 (G^N,H^N)\xRightarrow[N\to\infty]{}(G,H)
 \quad\text{in }D([0,T];\mathsf T\times\R).
\]

The limit is continuous, Gaussian, independent of $S$, and solves

\begin{equation}\label{clk:limit-sde}
 \begin{aligned}
 \dd G_t&=DF(p_t)G_t\dd t
 +\sum_{\substack{1\le a,b\le m\\a\ne b}}v_{ab}\sqrt{r_{ab}(p_t)}\dd B_t^{ab},\\
 \dd H_t&=D\lambda(p_t)G_t\dd t
 +\sum_{\substack{1\le a,b\le m\\a\ne b}}\sqrt{r_{ab}(p_t)}\dd B_t^{ab}.
 \end{aligned}
\end{equation}

The initial condition is $(G(0),H(0))=(G_0,0)$, and the $B^{ab}$ are independent standard Brownian motions, independent of $G_0$.
\end{lemma}
The shared counts generally correlate $G$ and $H$. The proof and covariance equations are in Appendix~\ref{app:clock}. In particular, $\operatorname{Var}(H_t)>0$ for every $t\in(0,T]$.
\subsection{Stable increments at a fluctuating clock}
\begin{proposition}[The fixed-time clock tangent]\label{clk:tangent}
For $k\in\N$ and deterministic times $0<t_1<\cdots<t_k\le T$, set $Z_j\coloneqq|H_{t_j}|^{\frac1\alpha}\xi_j$, where $\xi_1,\ldots,\xi_k$ are independent standard symmetric $\alpha$-stable variables, independent of $(G,H,S)$. Then, stably relative to $\sigma(S)$,

\begin{equation*}
 \begin{aligned}
 \left((G_{t_j}^N,H_{t_j}^N)_{j=1}^k,
   \left(N^{\frac1{2\alpha}}(L_{t_j}^N-L_{t_j})\right)_{j=1}^k\right) &\xRightarrow[N\to\infty]{}
   \left((G_{t_j},H_{t_j})_{j=1}^k,(Z_j)_{j=1}^k\right).
\end{aligned}
\end{equation*}

For $u_1,\ldots,u_k\in\R$,

\[
 \begin{aligned}
 \Ee\left[e^{i\sum_{j=1}^ku_jZ_j}\middle|G,H,S\right]
 &=\Ee\left[e^{i\sum_{j=1}^ku_jZ_j}\middle|H\right] =\exp\left\{-\sum_{j=1}^k|H_{t_j}||u_j|^\alpha\right\}.
\end{aligned}
\]

\end{proposition}
\begin{remark*}
The $Z_j$ are independent given $H$. Their random scales can make them dependent without conditioning. The proof removes increments in shrinking disjoint neighborhoods of the deterministic clock times, then uses stable scaling. Appendix~\ref{app:clock} gives the proof.
\end{remark*}
\subsection{Two empirical centers}
For a finite signed measure $\zeta$ on $\R$, define translation by $(\tau_\ell\zeta)(\phi)\coloneqq\zeta(\phi(\cdot+\ell))$ for $\ell\in\R$ and bounded continuous $\phi:\R\to\R$. Set, for $t\in[0,T]$,
\begin{equation}\label{clk:centers}
\nu_t^N\coloneqq\sum_{a=1}^m p_{t,a}^N\delta_{y_a}, \qquad \nu_t\coloneqq\sum_{a=1}^m p_{t,a}\delta_{y_a}, \qquad \mu_t^N\coloneqq\tau_{L_t^N}\nu_t^N, \qquad \mu_t^{[N]}\coloneqq\tau_{L_t^N}\nu_t, \quad\text{and}\quad \mu_t^L\coloneqq\tau_{L_t}\nu_t.
\end{equation}
For $r\in\R$, let $H^r(\R)$ be the completion of the Schwartz space $\mathcal S(\R)$ under $$\|\phi\|_{H^r(\R)}^2\coloneqq\frac1{2\pi} \int_\R(1+\xi^2)^r|\widehat\phi(\xi)|^2\dd\xi,$$ where $$\widehat\phi(\xi)\coloneqq\int_\R e^{-ix\xi}\phi(x)\dd x$$ for $\xi\in\R$. The Fourier transform extends by duality to tempered distributions. Translation extends to an isometry on every $H^r(\R)$, with $\widehat{\tau_\ell\zeta}(\xi)=e^{-i\ell\xi}\widehat\zeta(\xi)$. In particular, $\delta_y\in H^{-s}(\R)$ for $s>\tfrac12$.
\begin{theorem}[Functional fluctuations at the finite-$N$ center]\label{clk:finite-center}
For every $s>\tfrac12$, stably relative to $\sigma(S)$,

\begin{equation}\label{clk:finite-center-limit}
 \left(L^N,\sqrt N(\mu^N-\mu^{[N]})\right)
 \xRightarrow[N\to\infty]{}(L,\mathcal G^L)
\end{equation}
in $D([0,T];\R)\times D([0,T];H^{-s}(\R))$, where $$\mathcal G_t^L\coloneqq\tau_{L_t}\sum_{a=1}^mG_{t,a}\delta_{y_a}$$ for $t\in[0,T]$. Given $L$, the field is centered Gaussian. Define $a_\phi(\ell)\coloneqq(\phi(y_a+\ell))_{a=1}^m\in\R^m$ for $\ell\in\R$ and $\phi\in H^s(\R)$. For $\phi,\psi\in H^s(\R)$ and $r,t\in[0,T]$,

\[
 \begin{aligned}
 \operatorname{Cov}(\mathcal G_r^L(\phi),\mathcal G_t^L(\psi)\mid L) &=a_\phi(L_r)^\top\operatorname{Cov}(G_r,G_t)a_\psi(L_t).
\end{aligned}
\]

\end{theorem}
Aligning the large jumps of $S$ under the two clocks proves $L^N\xrightarrow[N\to\infty]{\mathbf{P}}L$ in $J_1$. Translation continuity then transfers the population CLT. Appendix~\ref{app:clock} gives the proof and the following two results.
\begin{corollary}[Three regimes at the limiting center]\label{clk:regimes}
Fix the times in Proposition~\ref{clk:tangent} and $s>\tfrac32$. Let $Z_j\coloneqq|H_{t_j}|^{\frac1\alpha}\xi_j$ and $\mathcal C_j\coloneqq-Z_j\partial_x\mu_{t_j}^L\in H^{-s}(\R)$. Here $\partial_x\zeta$ is the distributional derivative of a finite signed measure $\zeta$, defined by $(\partial_x\zeta)(\phi)\coloneqq-\zeta(\phi')$ for smooth compactly supported $\phi:\R\to\R$. In $(H^{-s}(\R))^k$, stably relative to $\sigma(S)$,
\begin{subequations}\label{clk:regimes-limit}
for $0<\alpha<1$,

\begin{equation}
 \left(\sqrt N(\mu_{t_j}^N-\mu_{t_j}^L)\right)_{j=1}^k
 \xRightarrow[N\to\infty]{}(\mathcal G_{t_j}^L)_{j=1}^k;
\end{equation}

for $\alpha=1$,

\begin{equation}
 \left(\sqrt N(\mu_{t_j}^N-\mu_{t_j}^L)\right)_{j=1}^k
 \xRightarrow[N\to\infty]{}(\mathcal G_{t_j}^L+\mathcal C_j)_{j=1}^k;
\end{equation}

and for $1<\alpha<2$,

\begin{equation}
 \left(N^{\frac1{2\alpha}}(\mu_{t_j}^N-\mu_{t_j}^L)\right)_{j=1}^k
 \xRightarrow[N\to\infty]{}(\mathcal C_j)_{j=1}^k.
\end{equation}

\end{subequations}
For $\alpha>1$, the root-$N$ field is not tight at any $t\in(0,T]$.
\end{corollary}
The empirical and clock terms have scales $N^{-\frac12}$ and $N^{-\frac1{2\alpha}}$. At $\alpha=1$ both survive and can remain dependent: Appendix~\ref{app:clock} gives the conditional Gaussian--Cauchy law and a two-state example. For real paths, $M_1$ is the Skorokhod topology defined by uniform comparison of continuous order-preserving parametrizations of the completed graph, including the vertical segment at every jump \citep{Whitt2002}.
\begin{proposition}[No additive path normalization at the limiting center]\label{clk:no-path}
For every $T>0$ and $s>\tfrac12$, there are almost surely positive random constants $c_L,c_\mu$ such that $\liminf_{N\to\infty}\sup_{t\in[0,T]}|L_t^N-L_t|\ge c_L$ and $\liminf_{N\to\infty}\sup_{t\in[0,T]} \|\mu_t^N-\mu_t^L\|_{H^{-s}}\ge c_\mu$. For every deterministic $a_N\xrightarrow[N\to\infty]{}\infty$, set $\mathcal D^{N,a}\coloneqq a_N(L^N-L)$ and $\mathcal E^{N,a}\coloneqq a_N(\mu^N-\mu^L)$. Almost surely,
\begin{subequations}\label{clk:no-path-bound}

\begin{equation}
 \sup_{t\in[0,T]}|\mathcal D_t^{N,a}|\xrightarrow[N\to\infty]{}\infty
\end{equation}

and

\begin{equation}
 \sup_{t\in[0,T]}\|\mathcal E_t^{N,a}\|_{H^{-s}}
       \xrightarrow[N\to\infty]{}\infty.
\end{equation}

\end{subequations}
The family $(\mathcal D^{N,a})_{N\in\N}$ is tight in neither $J_1$ nor $M_1$, and $(\mathcal E^{N,a})_{N\in\N}$ is not $J_1$-tight.
\end{proposition}
\begin{remark*}
Thus the first line of \eqref{clk:regimes-limit} has no functional extension at the limiting center. A jump of $L$ almost surely occurs at no chain event time, leaving a jump of fixed size in the difference. Centering by the finite-$N$ signal removes this obstruction.
\end{remark*}
\subsection{The same microscopic law under two couplings}
We compare two couplings of a torus translation model. The torus $\mathbf T=\R/\Z$ and its Sobolev spaces are as in Section~\ref{rk:section}. Translations are understood modulo one.
\begin{proposition}[Identical microscopic laws, different path limits]\label{clk:same-law}
Fix $T>0$, $0<\alpha<1$, $\lambda>0$, $0<|\varepsilon|<1$, and $s>\tfrac12$. Let $(Y_i)_{i\in\N}$ be independent on $\mathbf T$ with common law $\nu(\dd x)\coloneqq(1+\varepsilon\cos(2\pi x))\dd x$, and set $\nu^N\coloneqq N^{-1}\sum_{i=1}^N\delta_{Y_i}$. There are rank and grid couplings, denoted by $\mathrm r$ and $\mathrm g$, in which $L^N$ is a compound Poisson process of rate $N\lambda$ with marks $N^{-\frac1\alpha}U$, where $\Ee[e^{iuU}]=e^{-|u|^\alpha}$ for $u\in\R$, independent of $(Y_i)_{i\in\N}$. In each construction, let $L_t\coloneqq S_{\lambda t}$ for a standard symmetric $\alpha$-stable process $S$ independent of the sample, and set $\mu_t^N\coloneqq\tau_{L_t^N}\nu^N$ and $\mu_t\coloneqq\tau_{L_t}\nu$ for $t\in[0,T]$. For each $N$, the particle systems $(Y_i+L^N)_{i=1}^N$ have identical laws. In the rank coupling,

\begin{equation}\label{clk:same-law-rank}
 \sqrt N(\mu^{N,\mathrm r}-\mu^{\mathrm r})
 \xRightarrow[N\to\infty]{}\Gamma
\end{equation}

in $D([0,T];H^{-s}(\mathbf T))$, where $\Gamma_t\coloneqq\tau_{L_t}\mathcal B$ and $\mathcal B$ is the $\nu$-Brownian bridge of \eqref{rk:bridge-covariance}, independent of $S$. In the grid coupling, $L_t^{N,\mathrm g}\coloneqq S_{\frac{K_t^N}{N}}$, where $K^N$ is a rate-$N\lambda$ Poisson process independent of $S$. The limit \eqref{clk:same-law-rank} holds at every finite collection of deterministic times, and

\begin{equation}\label{clk:same-law-grid-center}
 \sqrt N(\mu^{N,\mathrm g}-\tau_{L^{N,\mathrm g}}\nu)
       \xRightarrow[N\to\infty]{}\Gamma
\end{equation}

in $D([0,T];H^{-s}(\mathbf T))$. All these limits are stable relative to $\sigma(S)$ in their respective constructions. For every deterministic $a_N\xrightarrow[N\to\infty]{}\infty$,

\begin{equation}\label{clk:same-law-obstruction}
 \sup_{t\in[0,T]}a_N\|\mu_t^{N,\mathrm g}-\mu_t^{\mathrm g}\|_{H^{-s}}
       \xrightarrow[N\to\infty]{}\infty
       \quad\text{almost surely}.
\end{equation}

Thus the grid-coupled differences normalized by $a_N$ are not $J_1$-tight at the limiting center.
\end{proposition}
\begin{remark*}
The microscopic laws agree for each $N$, but their joint laws with $L$ and their joint realizations across $N$ depend on the coupling. Rank coupling controls the signal error uniformly below the empirical scale. Grid coupling moves jump times and retains the path obstruction. Appendix~\ref{app:clock} proves these assertions using a stable-tail estimate.
\end{remark*}

\section{Conclusions}
Poisson ranks and random clocks lead to different fluctuation problems even when the microscopic particle laws agree. Under thinning with $0<\alpha<1$, disagreement jumps are smaller than sampling fluctuations, so the conditional Gaussian limit is unchanged. Quantile errors can survive at the sampling scale: a second-order tail term shifts the Gaussian mean at criticality and determines the leading response below it. Clock errors act through displaced jumps. Their finite-dimensional limits do not give an additive path limit at the limiting center, whereas centering at the finite-system signal restores Gaussian path fluctuations.

\appendix
\setcounter{innerthm}{0}
\renewcommand{\theinnerthm}{\thesection.\arabic{innerthm}}
\numberwithin{innerthm}{section}
\numberwithin{equation}{section}
\section{Proof of the conditional sampling limit}\label{app:sampling}
We use the notation of Section~\ref{mf:section}.
\subsection{Sobolev estimates and prescribed-signal sampling}
For $s>\tfrac d2+1$, $x,y\in\mathbf T^d$, and $\mu,\widetilde\mu\in\Prob(\mathbf T^d)$, the embedding $H^s\hookrightarrow C^1(\mathbf T^d)$ yields
\begin{subequations}\label{mf:dirac-bound}

\begin{equation}
 \|\delta_x-\delta_y\|_{H^{-s}}
 \le C_s d_{\mathbf T^d}(x,y)
\end{equation}

and

\begin{equation}
 \|\mu-\widetilde\mu\|_{H^{-s}}
 \le C_sW_1(\mu,\widetilde\mu).
\end{equation}

\end{subequations}
Indeed, $|\phi(x)-\phi(y)|\le C_s\|\phi\|_{H^s}d_{\mathbf T^d}(x,y)$, and the second bound follows by integrating against a coupling. For $\mu\in\Prob(\mathbf T^d)$, the adjoint is

\[
 \mathcal A_\mu^*:\left|\begin{aligned}
 H^{-(s-2)}&\longrightarrow H^{-s}\\
 \xi&\longmapsto [\phi\longmapsto\xi(\mathcal A_\mu\phi)].
 \end{aligned}\right.
\]

For $v\in C^{\lceil s\rceil+1}(\mathbf T^d;\R^d)$ and $\xi\in H^{-(s-1)}$, define $\nabla\cdot(v\xi)\in H^{-s}$ by $[\nabla\cdot(v\xi)](\phi)\coloneqq-\xi(v\cdot\nabla\phi)$ for $\phi\in H^s$. Fix $\ell\in D([0,T];\R^d)$ with $\ell_0=0$. Let $\widetilde X^{i,N}$ and $\mu^\ell$ solve the empirical and nonlinear equations driven by $\ell$. In this subsection, $N\ge1$, $i\in\{1,\ldots,N\}$, $j\in\{1,\ldots,q\}$, and time variables belong to $[0,T]$. Define the translated states $Y_t^{i,N}\coloneqq\widetilde X_t^{i,N}-\ell_t$ and their empirical law $\nu_t^N\coloneqq\tau_{-\ell_t}\widetilde\mu_t^N$. The translated reference law is $\nu_t\coloneqq\tau_{-\ell_t}\mu_t^\ell$, and its empirical fluctuation is $\zeta_t^N\coloneqq\sqrt N(\nu_t^N-\nu_t)$. The paths $Y^{i,N}$ are continuous and solve $\dd Y_t^{i,N}=\widehat B_t(Y_t^{i,N},\nu_t^N)\dd t+\sigma\dd W_t^i$, where

\[
 \widehat B_t:\left|\begin{aligned}
 \mathbf T^d\times\Prob(\mathbf T^d)&\longrightarrow \R^d\\
 (y,\nu)&\longmapsto B(y+\ell_t,m(\tau_{\ell_t}\nu)),
 \end{aligned}\right.
\]

and $g_{j,t}:\mathbf T^d\to\R$ is defined by $g_{j,t}(y)\coloneqq g_j(y+\ell_t)$. These coefficients are measurable in time, with spatial bounds independent of $\ell$.
\begin{lemma}[Prescribed-signal sampling]\label{mf:prescribed}
Under Assumption~\ref{mf:assumptions}, put $r\coloneqq s-2>\tfrac d2+1$. For every $p\ge2$,

\begin{equation}\label{mf:sampling-moments}
 \sup_{N\ge1}\Ee\left[\sup_{t\in[0,T]}
 \|\sqrt N(\widetilde\mu_t^N-\mu_t^\ell)\|_{H^{-r}}^p\right]
 \le C_{p,T},
\end{equation}

with $C_{p,T}$ independent of $\ell$. The fields $\zeta^N$ converge in law in $C([0,T];H^{-s})$ to the Gaussian solution of the co-moving linearized equation. Hence $\sqrt N(\widetilde\mu^N-\mu^\ell)$ converges in law in $D([0,T];H^{-s})$ to its translation by $\ell$. If $\ell$ is measurable with respect to an environment independent of the initial sample and Brownian motions, these conclusions hold for conditional laws for almost every environment.
\end{lemma}
\begin{proof}
\emph{Moment bounds.} Choose $p_*>\max\{p,2\}$. Couple $Y^{i,N}$ with i.i.d. nonlinear copies $\bar Y^i$ of law $\nu$, using the same initial states and Brownian motions. In paired Euclidean lifts, put $$Z_{j,t}^N\coloneqq\frac1N\sum_{i=1}^N g_{j,t}(\bar Y_t^i)-\nu_t(g_{j,t}).$$ The maximal paired distance is $\Delta_N(t)\coloneqq\max_{1\le i\le N}\sup_{u\in[0,t]}|Y_u^{i,N}-\bar Y_u^i|$. Since the Brownian terms cancel,

\[
 |\nu_t^N(g_{j,t})-\nu_t(g_{j,t})|
 \le\|g_j\|_{\mathrm{Lip}}\Delta_N(t)+|Z_{j,t}^N|
\]

and

\[
 \Delta_N(t)\le C\int_0^t\Delta_N(u)\dd u
                  +C\sum_{j=1}^q\int_0^t|Z_{j,u}^N|\dd u.
\]

Gronwall's inequality gives $\Delta_N(T)\le C_T\sum_{j=1}^q\int_0^T|Z_{j,u}^N|\dd u$. The summands defining $Z_{j,t}^N$ are centered, bounded, and i.i.d., so

\begin{equation}\label{mf:empirical-moment}
 \Ee\left[|Z_{j,t}^N|^{p_*}\right]\le C_{p_*}N^{-\frac{p_*}{2}}.
\end{equation}

Jensen's inequality gives

\begin{equation}\label{mf:particle-comparison}
 \begin{aligned}
 \Ee\left[\Delta_N(T)^{p_*}\right]
 &\le C_{p_*,T}\sum_{j=1}^q\int_0^T\Ee\left[|Z_{j,u}^N|^{p_*}\right]\dd u \ \ \byeq{mf:empirical-moment}{\le}\ \ C_{p_*,T}N^{-\frac{p_*}{2}}.
\end{aligned}
\end{equation}

For the independent-copy empirical field $\Gamma_t^N\coloneqq N^{-\frac12}\sum_{i=1}^N(\delta_{\bar Y_t^i}-\nu_t)$, bounded drift and Brownian increments imply $\Ee\left[|\bar Y_t^i-\bar Y_u^i|^{p_*}\right] \le C_{p_*,T}|t-u|^{\frac{p_*}{2}}$. For centered i.i.d. variables $\Xi_1,\ldots,\Xi_N$ in a real Hilbert space with finite $p_*$th moment, symmetrization and the Hilbert-space moment inequality yield

\[
 \begin{aligned}
 \Ee\left[\left\|N^{-\frac12}\sum_{i=1}^N\Xi_i\right\|^{p_*}\right]
 &\le C_{p_*}(\Ee[\|\Xi_1\|^2])^{\frac{p_*}{2}} +C_{p_*}N^{1-\frac{p_*}{2}}\Ee[\|\Xi_1\|^{p_*}].
\end{aligned}
\]

For the centered $H^{-r}$-valued Dirac increments, this gives
\begin{subequations}\label{mf:independent-increments}

\begin{equation}
 \Ee[\|\Gamma_t^N-\Gamma_u^N\|_{H^{-r}}^{p_*}]
 \ \ \byeq{mf:dirac-bound}{\le}\ \ C_{p_*,T}|t-u|^{\frac{p_*}{2}}
\end{equation}

and

\begin{equation}
 \Ee[\|\Gamma_0^N\|_{H^{-r}}^{p_*}]\le C_{p_*}.
\end{equation}

\end{subequations}
For $n\in\{0,1,\ldots\}$ and $k\in\{0,\ldots,2^n\}$, set $t_{n,k}\coloneqq Tk2^{-n}$. Then dyadic approximation and Minkowski give

\begin{align*}
 \left\|\sup_{t\in[0,T]}\|\Gamma_t^N\|_{H^{-r}}\right\|_{L^{p_*}(\Omega)} &\le\|\Gamma_0^N\|_{L^{p_*}(\Omega;H^{-r})} +\sum_{n=0}^\infty\left(\sum_{k=1}^{2^n}
 \Ee[\|\Gamma_{t_{n,k}}^N-\Gamma_{t_{n,k-1}}^N\|_{H^{-r}}^{p_*}]
 \right)^{\frac1{p_*}}\\
 \ \ &\byeq{mf:independent-increments}{\le}\ \ 
 C+C\sqrt T\sum_{n=0}^\infty2^{-n\left(\frac12-\frac1{p_*}\right)}<\infty.
\end{align*}

Finally,

\[
 \sup_{t\in[0,T]}\|\zeta_t^N\|_{H^{-r}}
 \ \ \byeq{mf:dirac-bound}{\le}\ \ 
 \sup_{t\in[0,T]}\|\Gamma_t^N\|_{H^{-r}}
       +C_r\sqrt N\Delta_N(T).
\]

Together with \eqref{mf:particle-comparison} and translation isometry, this proves \eqref{mf:sampling-moments}.
\paragraph{Tightness and the limit equation.} Define the co-moving generator

\[
 \widehat{\mathcal A}_{t,\nu}:\left|\begin{aligned}
 H^s&\longrightarrow H^{s-2}\\
 \phi&\longmapsto \widehat B_t(\cdot,\nu)\cdot\nabla\phi
                         +\frac12\operatorname{tr}\left(aD^2\phi\right).
 \end{aligned}\right.
\]

It\^o's formula gives the $H^{-s}$ fluctuation drift

\begin{equation}\label{mf:exact-drift}
 \widehat{\mathcal A}_{t,\nu_t^N}^*\zeta_t^N
 -\sum_{j=1}^q\zeta_t^N(g_{j,t})\nabla\cdot(c_{j,t}^N\nu_t),
\end{equation}

where $c_{j,t}^N:\mathbf T^d\to\R^d$ is defined by

\[
 \begin{aligned}
 c_{j,t}^N(y)\coloneqq\int_0^1
 \partial_{m_j}B\left(y+\ell_t,m(\tau_{\ell_t}[(1-\theta)\nu_t+\theta\nu_t^N])\right)\dd\theta.
 \end{aligned}
\]

The martingale part is

\[
 \widehat M_t^N(\phi)
 \coloneqq \frac1{\sqrt N}\sum_{i=1}^N\int_0^t
 \nabla\phi(Y_u^{i,N})^\top\sigma\dd W_u^i,
 \qquad\phi\in C^\infty(\mathbf T^d).
\]

For $\phi,\psi\in C^\infty(\mathbf T^d)$, its covariation is

\[
 \ip{\widehat M^N(\phi)}{\widehat M^N(\psi)}_t
 =\int_0^t\nu_u^N(\nabla\phi^\top a\nabla\psi)\dd u.
\]

For $r>\tfrac d2+1$ and $\zeta\in H^{-r}$, Sobolev embedding gives $|\zeta(g_{j,t})|\le C\|\zeta\|_{H^{-r}}$ and $\|\nabla\cdot(c_{j,t}^N\nu_t)\|_{H^{-r}}\le C$. The covariance trace per unit time is bounded by $4\pi^2\sum_{k\in\Z^d}(1+|k|^2)^{-r}k^\top ak<\infty$. Thus $\widehat M^N$ is a continuous $H^{-r}$-valued martingale, and the drift in \eqref{mf:exact-drift} has $H^{-s}$ norm at most $C\|\zeta_t^N\|_{H^{-r}}$. For $\delta>0$, stopping times $\theta\in[0,T]$, and $0\le h\le\delta$, $$\Ee\left[\|\zeta_{(\theta+h)\wedge T}^N-\zeta_\theta^N\|_{H^{-s}}^2\right] \le C_T(\delta+\delta^2).$$ The moment bound, the compact embedding $H^{-r}\hookrightarrow H^{-s}$, and these increment estimates give tightness with continuous limits. The Hilbert-space initial CLT gives the $\mu_0$-Brownian bridge. For $\phi\in C^\infty(\mathbf T^d)$, define

\[
 h_{j,t}^\phi:\left|\begin{aligned}
 \mathbf T^d&\longrightarrow \R\\
 y&\longmapsto \partial_{m_j}B(y+\ell_t,m(\tau_{\ell_t}\nu_t))
                       \cdot\nabla\phi(y).
 \end{aligned}\right.
\]

Taylor's formula in the moment vector yields

\[
 \begin{aligned}
 \sqrt N\{\nu_t^N(\widehat{\mathcal A}_{t,\nu_t^N}\phi)
                    -\nu_t(\widehat{\mathcal A}_{t,\nu_t}\phi)\} &=\zeta_t^N(\widehat{\mathcal A}_{t,\nu_t}\phi)
   +\sum_{j=1}^q\zeta_t^N(g_{j,t})\nu_t(h_{j,t}^\phi)+R_N(t,\phi),
\end{aligned}
\]

where

\[
 \begin{aligned}
 |R_N(t,\phi)|
 &\le C_\phi\|\zeta_t^N\|_{H^{-r}}
                  |m(\tau_{\ell_t}(\nu_t^N-\nu_t))| +C_\phi\sqrt N|m(\tau_{\ell_t}(\nu_t^N-\nu_t))|^2 \le C_\phi N^{-\frac12}\|\zeta_t^N\|_{H^{-r}}^2.
\end{aligned}
\]

Hence

\[
 \begin{aligned}
 \Ee\left[\int_0^T|R_N(t,\phi)|\dd t\right]
 &\le C_\phi N^{-\frac12}\int_0^T\Ee[\|\zeta_t^N\|_{H^{-r}}^2]\dd t \ \ \byeq{mf:sampling-moments}{\le}\ \ C_{\phi,T}N^{-\frac12}.
\end{aligned}
\]

For $\phi,\psi\in C^\infty(\mathbf T^d)$, set $H\coloneqq\nabla\phi^\top a\nabla\psi\in C^\infty(\mathbf T^d)$. Then

\begin{align*}
 \Ee\left[\sup_{t\in[0,T]}\left|\int_0^t(\nu_u^N-\nu_u)(H)\dd u\right|\right] &\le\|H\|_{\mathrm{Lip}}T\Ee[\Delta_N(T)] +\int_0^T\Ee\left[\left|\frac1N\sum_{i=1}^N H(\bar Y_u^i)-\nu_u(H)\right|\right]\dd u\\
 \ \ &\byeq{mf:particle-comparison}{\le}\ \ C_{\phi,\psi,T}N^{-\frac12}.
\end{align*}

For every $k\in\N$ and $\phi_1,\ldots,\phi_k\in C^\infty(\mathbf T^d)$, apply the continuous-martingale functional CLT \citep[Theorem~II.2]{Rebolledo1980}, together with the Cram\'er--Wold argument, to the finite-dimensional projections $(\widehat M^N(\phi_j))_{j=1}^k$. Continuity makes the jump condition automatic. The preceding estimate gives uniform convergence in probability of the brackets, whose limits are deterministic because $\ell$ is fixed. The covariance trace bound and compact embedding above give tightness of $\widehat M^N$ in $C([0,T];H^{-s})$. Thus the projected limits identify a centered Gaussian martingale $\widehat M$ in this space, with covariance $$\ip{\widehat M(\phi)}{\widehat M(\psi)}_t=\int_0^t\nu_u(\nabla\phi^\top a\nabla\psi)\dd u.$$ Write $\zeta_0\in H^{-s}$ for the initial Gaussian bridge. To obtain joint convergence, take a bounded continuous function $\Phi:H^{-s}\to(0,\infty)$ and change measure with density $\tfrac{\Phi(\zeta_0^N)}{\Ee[\Phi(\zeta_0^N)]}$. This density is measurable at time zero, so $\widehat M^N$ remains a martingale. The initial CLT makes these densities uniformly bounded for large $N$, so bracket convergence is preserved and the same martingale CLT applies. Hence, for every bounded continuous $\Psi:C([0,T];H^{-s})\to\R$,

\[
 \Ee\left[\Phi(\zeta_0^N)\Psi(\widehat M^N)\right]
 \xrightarrow[N\to\infty]{}
 \Ee\left[\Phi(\zeta_0)\right]\Ee\left[\Psi(\widehat M)\right].
\]

Together with tightness, this factorization proves joint convergence of $(\zeta_0^N,\widehat M^N)$ to $(\zeta_0,\widehat M)$ and independence of the two limit coordinates. Every limit solves

\[
 \begin{aligned}
 \zeta_t(\phi)
 &=\zeta_0(\phi)+\widehat M_t(\phi)
   +\int_0^t\zeta_u(\widehat{\mathcal A}_{u,\nu_u}\phi)\dd u +\int_0^t\sum_{j=1}^q\zeta_u(g_{j,u})\nu_u(h_{j,u}^\phi)\dd u.
\end{aligned}
\]

\paragraph{Uniqueness and translation.} We use backward duality. For $0\le u\le v\le T$, let

\[
 P_{u,v}:\left|\begin{aligned}
 C^1(\mathbf T^d)&\longrightarrow C^1(\mathbf T^d)\\
 \phi&\longmapsto [y\longmapsto\Ee_{u,y}\left[\phi(\bar Y_v)\right]],
 \end{aligned}\right.
\qquad
 \widehat{\mathcal K}_u:\left|\begin{aligned}
 C^1(\mathbf T^d)&\longrightarrow C^1(\mathbf T^d)\\
 \phi&\longmapsto \displaystyle\sum_{j=1}^qg_{j,u}\nu_u(h_{j,u}^\phi).
 \end{aligned}\right.
\]

Here $\Ee_{u,y}$ denotes expectation for the diffusion with fixed law flow $\nu$ started from $y$ at time $u$. Both operators are uniformly bounded on $C^1$. For $\phi\in C^\infty(\mathbf T^d)$ and $t\in[0,T]$, solve

\begin{equation}\label{mf:backward}
 v_u=P_{u,t}\phi+\int_u^tP_{u,w}\widehat{\mathcal K}_w v_w\dd w,
 \qquad 0\le u\le t.
\end{equation}

For $n\in\{0,1,\ldots\}$, the $n$th Picard term is bounded in $C^1$ by $\tfrac{C_T^{n+1}(t-u)^n\|\phi\|_{C^1}}{n!}$. Differentiating the diffusion flow and the integral equation gives

\[
 \sup_{u\in[0,t]}\|v_u\|_{C^k}\le C_T\|\phi\|_{C^k},
 \qquad 1\le k\le\lceil s\rceil+2.
\]

Consequently, $v\in L^\infty([0,t];H^{s+2})\cap W^{1,\infty}([0,t];H^s)$. Here $W^{1,\infty}([0,t];H^s)$ consists of $H^s$-valued functions whose value and weak time derivative are essentially bounded. The terminal value is $v_t=\phi$, and $\partial_uv_u=-(\widehat{\mathcal A}_{u,\nu_u}+\widehat{\mathcal K}_u)v_u$ for almost every $u\in[0,t]$. The derivative follows from the integral evolution identity for $P_{u,w}$. Extend $v$ constantly beyond $[0,t]$ and smooth it in time and space. The resulting tests $v^{(n)}$ satisfy

\[
 \begin{aligned}
 \|v^{(n)}-v\|_{L^1([0,t];H^{s+2})}
 +\|v^{(n)}-v\|_{C([0,t];H^s)} +\|\partial_uv^{(n)}-\partial_uv\|_{L^1([0,t];H^s)}
 &\xrightarrow[n\to\infty]{}0.
\end{aligned}
\]

Apply the product rule to $\zeta_u(v_u^{(n)})$. The drift passes to the limit through the uniform operator bound $H^{s+2}\to H^s$, and the martingale through It\^o isometry and $H^s\hookrightarrow C^1$. For any càdlàg $H^{-s}$-valued solution, the same limit argument applies after localizing its bounded path norm. Thus $\zeta_t(\phi)=\zeta_0(v_0)+\int_0^t\dd\widehat M_u(v_u)$. The integral $\int_0^t\dd\widehat M_u(v_u)$ is the scalar stochastic integral defined by the $H^{-s}$--$H^s$ pairing with the time-dependent test $v_u$. This identifies the law, proves Gaussianity and uniqueness, and gives convergence of the full sequence. Finally,

\[
 \begin{aligned}
 \sup_{t\in[0,T]}\|\tau_{\ell_t}z_n(t)-\tau_{\ell_t}z(t)\|_{H^{-s}} &=\sup_{t\in[0,T]}\|z_n(t)-z(t)\|_{H^{-s}}
\end{aligned}
\]

for $z_n,z\in C([0,T];H^{-s})$. Translating back therefore gives convergence in the original frame. Conditioning on an independent environment fixes $\ell$ and leaves the private-input law unchanged.
\end{proof}
\subsection{Positive disagreement strips}
The strip estimates use the unconditioned filtration $(\cF_t)_{t\ge0}$. Write $\omega_{d-1}$ for the surface area of the unit sphere in $\R^d$.
\begin{lemma}[Positive strip estimate]\label{mf:strip}
Let $a,b$ be predictable $[0,\bar f]$-valued processes on $[0,T]$ in a filtration preserving the compensator $\dd t\,\nu_\alpha(\dd z)\dd u$. For $t\in[0,T]$, define the integrated threshold error $A_t\coloneqq\int_0^t|a_r-b_r|\dd r$, the disagreement variation

\[
 \begin{aligned}
 V_t\coloneqq\int_{[0,t]\times(\R^d\setminus\{0\})\times[0,\bar f]}
 |z|\,|\ind_{\{u\le a_r\}}-\ind_{\{u\le b_r\}}|
 \mathcal N^0(\dd r,\dd z,\dd u),
 \end{aligned}
\]

and

\[
 \begin{aligned}
 K_t\coloneqq\int_{[0,t]\times(\R^d\setminus\{0\})\times[0,\bar f]}
 z(\ind_{\{u\le a_r\}}-\ind_{\{u\le b_r\}})\mathcal N^0(\dd r,\dd z,\dd u).
 \end{aligned}
\]

For every stopping time $\tau\in[0,T]$ and $x,v>0$,
\begin{subequations}\label{mf:strip-estimate}

\begin{equation}
 \Pp(V_\tau>x,\ A_\tau\le v)\le Cvx^{-\alpha}
\end{equation}

and

\begin{equation}
 \sup_{t\in[0,\tau]}|K_t|\le V_\tau.
\end{equation}

\end{subequations}
The estimate is uniform for positive Borel mark measures $\nu_N$ on $\R^d\setminus\{0\}$ satisfying, for every $x>0$ and $N\ge1$,
\begin{subequations}\label{mf:tail-bounds}

\begin{equation}
 \nu_N(\{z\in\R^d:|z|>x\})\le Cx^{-\alpha}
\end{equation}

and

\begin{equation}
 \int_{\{z\in\R^d:0<|z|\le x\}}|z|\nu_N(\dd z)\le Cx^{1-\alpha}.
\end{equation}

\end{subequations}
\end{lemma}
\begin{proof}
Polar coordinates give

\[
 \begin{aligned}
 \nu_\alpha(\{z\in\R^d:|z|>x\})
 &=c_{d,\alpha}\omega_{d-1}\int_x^\infty r^{-1-\alpha}\dd r =\frac{c_{d,\alpha}\omega_{d-1}}{\alpha}x^{-\alpha}.
\end{aligned}
\]

Similarly,

\[
 \begin{aligned}
 \int_{\{z\in\R^d:0<|z|\le x\}}|z|\nu_\alpha(\dd z)
 &=c_{d,\alpha}\omega_{d-1}\int_0^x r^{-\alpha}\dd r =\frac{c_{d,\alpha}\omega_{d-1}}{1-\alpha}x^{1-\alpha}.
\end{aligned}
\]

Fix $\varepsilon>0$ and stop at $\widehat\tau\coloneqq \tau\wedge\inf\{t\in[0,T]:A_t>v+\varepsilon\}\wedge T$. Let $B_x$ count disagreement marks with $|z|>\tfrac x2$ up to $\widehat\tau$, and let $S_x$ sum the sizes of the remaining disagreement marks. Since $\int_0^{\bar f}|\ind_{\{u\le a_r\}}-\ind_{\{u\le b_r\}}|\dd u =|a_r-b_r|$, continuity of $A$ and predictable compensation imply
\begin{subequations}\label{mf:strip-compensation}

\begin{equation}
 \Ee[B_x]\le(v+\varepsilon)\nu_\alpha\left(\left\{z\in\R^d:|z|>\tfrac x2\right\}\right)
\end{equation}

and

\begin{equation}
 \Ee[S_x]\le(v+\varepsilon)
 \int_{\left\{z\in\R^d:0<|z|\le\tfrac x2\right\}}|z|\nu_\alpha(\dd z).
\end{equation}

\end{subequations}
On $\{A_\tau\le v\}$, stopping has no effect before $\tau$. Consequently,

\begin{align*}
 \Pp(V_\tau>x,A_\tau\le v)
 &\le\Pp(B_x\ge1)+\Pp(S_x>x)
 \le\Ee\left[B_x\right]+x^{-1}\Ee\left[S_x\right] \\ &\byeq{mf:strip-compensation}{\le}\ \ (v+\varepsilon)
 \nu_\alpha\left(\left\{z\in\R^d:|z|>\tfrac x2\right\}\right)
 +\frac{v+\varepsilon}{x}
 \int_{\left\{z\in\R^d:0<|z|\le\tfrac x2\right\}}|z|\nu_\alpha(\dd z) \byeq{mf:tail-bounds}{\le}\ \  C(v+\varepsilon)x^{-\alpha}.
\end{align*}

Let $\varepsilon\downarrow0$. The bound $\sup_{t\in[0,\tau]}|K_t|\le V_\tau$ is pathwise, and the calculation uses only \eqref{mf:tail-bounds}.
\end{proof}
\begin{lemma}[Closing the strip estimate]\label{mf:bootstrap}
Let $(\varepsilon_N)_{N\ge1}$ be a deterministic positive sequence with $\varepsilon_N\xrightarrow[N\to\infty]{}0$, let $R_N=O_{\Pp}(\varepsilon_N)$, and let $(A^N,V^N)$ start from zero, with $V^N\ge0$ and $A^N$ adapted, continuous, and nondecreasing. Suppose \eqref{mf:strip-estimate} holds uniformly for stopping times of the same filtration and

\begin{equation}\label{mf:bootstrap-hypothesis}
 A_t^N\le R_N+CV_t^N,\qquad t\in[0,T].
\end{equation}

Then $A_T^N=O_{\Pp}(\varepsilon_N)$ and $V_T^N=O_{\Pp}\left(\varepsilon_N^{\frac1\alpha}\right) =o_{\Pp}(\varepsilon_N)$. The conclusion also holds for an almost surely finite random $C$ independent of $N$.
\end{lemma}
\begin{proof}
Replace $C$ by $1\vee C$. For $M>0$, set $\tau_N\coloneqq \inf\{t\in[0,T]:A_t^N\ge2M\varepsilon_N\}\wedge T$. On $\{A_T^N>2M\varepsilon_N,\ R_N\le M\varepsilon_N\}$, continuity and \eqref{mf:bootstrap-hypothesis} give $A_{\tau_N}^N=2M\varepsilon_N$. Then

\[
 CV_{\tau_N}^N\ \ \byeq{mf:bootstrap-hypothesis}{\ge}\ \ 
 A_{\tau_N}^N-R_N\ge M\varepsilon_N.
\]

Therefore

\begin{align*}
 \Pp(A_T^N>2M\varepsilon_N)
 &\le\Pp(R_N>M\varepsilon_N)+\Pp\!\left(V_{\tau_N}^N>\frac{M\varepsilon_N}{2C},\,
                   A_{\tau_N}^N\le2M\varepsilon_N\right) \\ &\byeq{mf:strip-estimate}{\le}\ \ \Pp(R_N>M\varepsilon_N)
      +C'(M\varepsilon_N)^{1-\alpha}.
\end{align*}

For $R>0$, the strip estimate also gives

\[
 \begin{aligned}
 \Pp\left(V_T^N>R\varepsilon_N^{\frac1\alpha}\right)
\ \ &\byeq{mf:strip-estimate}{\le}\ \ 
 \Pp(A_T^N>2M\varepsilon_N)+C'MR^{-\alpha}.
 \end{aligned}
\]

Taking $\limsup_{N\to\infty}$ in the first bound and then $M\to\infty$ proves tightness of $\tfrac{A_T^N}{\varepsilon_N}$. In the last bound choose $M$ first and then $R$, giving tightness of $\tfrac{V_T^N}{\varepsilon_N^{\frac1\alpha}}$. Finally,

\[
 \frac{V_T^N}{\varepsilon_N}
 =\varepsilon_N^{\frac1\alpha-1}
      \frac{V_T^N}{\varepsilon_N^{\frac1\alpha}}
 \xrightarrow[N\to\infty]{\mathbf{P}}0.
\]

For random $C$, the same estimates on $\{C\le C_0\}$ have an additional error $\Pp(C>C_0)$. Let $C_0\to\infty$.
\end{proof}
\subsection{Construction without a cutoff}
\begin{proposition}[Well-posedness]\label{mf:wellposed}
Under Assumption~\ref{mf:assumptions} and $0<\alpha<1$, \eqref{mf:model} has a pathwise unique adapted solution, and \eqref{mf:reference} has a unique adapted conditional-law flow. Restricting the master to $|z|>\varepsilon$, on the same probability space and for each fixed $N$, gives

\begin{equation*}
 \begin{aligned}
 \max_{1\le i\le N}\sup_{t\in[0,T]}
 d_{\mathbf T^d}(X_t^{i,N,\varepsilon},X_t^{i,N}) +\sup_{t\in[0,T]}|L_t^{N,\varepsilon}-L_t^N|
 &\xrightarrow[\varepsilon\downarrow0]{\mathbf{P}}0.
\end{aligned}
\end{equation*}

For the reference system,

\begin{equation*}
 \begin{aligned}
 \sup_{t\in[0,T]}W_1(\mu_t^\varepsilon,\mu_t)
 +\sup_{t\in[0,T]}|L_t^\varepsilon-L_t| &\xrightarrow[\varepsilon\downarrow0]{\mathbf{P}}0.
\end{aligned}
\end{equation*}

\end{proposition}
\begin{proof}
Above each cutoff, there are finitely many candidate translations, with Lipschitz diffusions between them. Between candidates the reference law solves the Lipschitz McKean--Vlasov equation. At each candidate its pre-jump law fixes the acceptance label. Define the discarded variation by

\[
 J_T^\varepsilon\coloneqq
 \int_{[0,T]\times\{z\in\R^d:0<|z|\le\varepsilon\}\times[0,\bar f]}
 |z|\mathcal N^0(\dd t,\dd z,\dd u).
\]

Its expectation is

\[
 \begin{aligned}
 \Ee[J_T^\varepsilon]
 &=T\bar f\int_{\{z\in\R^d:0<|z|\le\varepsilon\}}|z|\nu_\alpha(\dd z) =\frac{T\bar f c_{d,\alpha}\omega_{d-1}}{1-\alpha}
        \varepsilon^{1-\alpha}.
\end{aligned}
\]

Couple $\varepsilon'<\varepsilon$ synchronously and set $$D_t\coloneqq\allowbreak \max_{1\le i\le N}\sup_{u\in[0,t]} d_{\mathbf T^d}(X_u^{i,N,\varepsilon},X_u^{i,N,\varepsilon'}).$$ We have $J_T^\varepsilon\downarrow0$ almost surely as $\varepsilon\downarrow0$. For $t\in[0,T]$, define the threshold discrepancy by $A_t\coloneqq \int_0^t|\mu_r^{N,\varepsilon}(f) -\mu_r^{N,\varepsilon'}(f)|\dd r$. Let $V_t$ be the disagreement variation of their acceptance labels on $|z|>\varepsilon$. The omitted marks have total size at most $J_T^\varepsilon$, so synchronous coupling gives $$D_t\le J_T^\varepsilon+V_t+C\int_0^tD_r\dd r\quad\text{and}\quad A_t\le\|f\|_{\mathrm{Lip}}\int_0^tD_r\dd r.$$ Hence $D_t+A_t\le C_T(J_T^\varepsilon+V_t),\, t\in[0,T]$, and $$\sup_{u\in[0,t]}|L_u^{N,\varepsilon}-L_u^{N,\varepsilon'}| \le J_T^\varepsilon+V_t.$$ Since $J_T^\varepsilon=O_{\Pp}(\varepsilon^{1-\alpha})$, lemmas~\ref{mf:strip} and \ref{mf:bootstrap} give $A_T=O_{\Pp}(\varepsilon^{1-\alpha})$, $V_T=O_{\Pp}\left(\varepsilon^{\frac{1-\alpha}{\alpha}}\right)$, and $D_T=O_{\Pp}(\varepsilon^{1-\alpha})$, uniformly in $\varepsilon'<\varepsilon$. Thus the particle and signal paths are Cauchy in probability for the uniform metric. For the reference, use $$D_t\coloneqq \Ee\left[\sup_{u\in[0,t]} d_{\mathbf T^d}(\bar X_u^\varepsilon,\bar X_u^{\varepsilon'}) \middle|\cG_T^0\right].$$ The cutoff flows depend only on past master atoms and independent private input, so this conditional quantity is $\cG_t^0$-measurable, bounds $\sup_{u\in[0,t]}W_1(\mu_u^\varepsilon,\mu_u^{\varepsilon'})$, and satisfies the same Gronwall estimate. The reference flows and signals are therefore Cauchy as well. Choose a subsequence with summable error probabilities. Its paths converge uniformly almost surely to adapted càdlàg limits. Lipschitz continuity lets us pass to the limit in the drift integrals. For each $\delta>0$, compensation gives

\[
 \Ee\left[\int_{[0,T]\times\{z\in\R^d:|z|>\delta\}\times[0,\bar f]}
 \ind_{\{u=\mu_{t-}(f)\}}\mathcal N^0(\dd t,\dd z,\dd u)\right]=0,
\]

and the same holds for the empirical threshold. Hence no candidate above $\delta$ lies on the limiting threshold. There are finitely many such candidates, and uniform convergence makes all their acceptance labels agree eventually. Marks below $\delta$ contribute at most $J_T^\delta\xrightarrow[\delta\downarrow0]{}0$, which identifies the limiting jump sums. To identify the conditional law, construct every cutoff copy from the same independent initial state and Brownian motion. Boundedness of the torus distance and the conditional coupling bound imply that these copies are Cauchy in $L^1$ for the uniform distance. Their limit $\bar X$ has an adapted càdlàg version. For bounded Lipschitz $\phi:\mathbf T^d\to\R$ and fixed $t\in[0,T]$,

\[
 \begin{aligned}
 \Ee\left[\left|\Ee[\phi(\bar X_t^\varepsilon)-\phi(\bar X_t)
                   \mid\cG_t^0]\right|\right] &\le\Ee[|\phi(\bar X_t^\varepsilon)-\phi(\bar X_t)|]
 \xrightarrow[\varepsilon\downarrow0]{}0.
\end{aligned}
\]

Also $\mu_t^\varepsilon(\phi)\xrightarrow[\varepsilon\downarrow0]{L^1}\mu_t(\phi)$, proving $\mu_t=\Law(\bar X_t\mid\cG_t^0)$. Independence of future master increments gives the same kernel when conditioning on $\cG_T^0$. A countable determining class, rational times, and càdlàg versions make this identity simultaneous in $t$. The Cauchy property gives convergence of the full cutoff family. For two solutions with identical input, the comparison becomes $D_t+A_t\le C_TV_t$. At $\tau_a\coloneqq \inf\{t\in[0,T]:A_t\ge a\}\wedge T$, continuity implies

\[
 \begin{aligned}
 \Pp(A_T>a)
 &\le\Pp\left(V_{\tau_a}>\tfrac a{2C_T},\ A_{\tau_a}\le a\right) \ \ \byeq{mf:strip-estimate}{\le}\ \ C_T'a^{1-\alpha}.
\end{aligned}
\]

Let $a\downarrow0$ to obtain $A_T=0$ almost surely. Compensation then eliminates all disagreement marks above every $\delta>0$. Letting $\delta\downarrow0$ gives $V_T=D_T=0$. The reference-law comparison gives uniqueness in the same way.
\end{proof}
\subsection{Conditional convergence and identification of the limit}
\begin{proof}[Proof of Theorem~\ref{mf:clt}]
Since $L$ is $\cG_T^0$-measurable and independent of the sample, Lemma~\ref{mf:prescribed} gives almost sure convergence of the conditional laws of $\widetilde\eta^N\coloneqq\sqrt N(\widetilde\mu^N-\mu)$. Set $R_N\coloneqq\int_0^T|\widetilde\mu_r^N(f)-\mu_r(f)|\dd r$. For $t\in[0,T]$, define $$A_t^N\coloneqq\int_0^t|\mu_{r-}^N(f)-\mu_{r-}(f)|\dd r$$ and let $V^N$ be the disagreement variation for $L^N,L$. For $p\ge2$ and $r\coloneqq s-2$, conditional Jensen gives

\begin{equation*}
 \begin{aligned}
 \Ee\left[R_N^p\middle|\cG_T^0\right]
 &\le T^{p-1}\|f\|_{H^r}^p\int_0^T
 \Ee\left[\|\widetilde\mu_u^N-\mu_u\|_{H^{-r}}^p
                    \middle|\cG_T^0\right]\dd u \ \ \byeq{mf:sampling-moments}{\le}\ \ C_{p,T}N^{-\frac p2}.
\end{aligned}
\end{equation*}

Synchronous coupling, with $$D_t^N\coloneqq \frac1N\sum_{i=1}^N\sup_{u\in[0,t]} d_{\mathbf T^d}(X_u^{i,N},\widetilde X_u^{i,N}),$$ gives $$D_t^N\le V_t^N+C\int_0^tD_r^N\dd r\quad\text{and}\quad A_t^N\le R_N+\|f\|_{\mathrm{Lip}}\int_0^tD_r^N\dd r.$$ Thus $D_t^N\le C_TV_t^N$ and $A_t^N\le R_N+C_TV_t^N$. Lemma~\ref{mf:bootstrap} with $\varepsilon_N\coloneqq N^{-\frac12}$ gives
\begin{equation}\label{mf:strip-rate}
A_T^N=O_{\Pp}\left(N^{-\frac12}\right) \quad\text{and}\quad V_T^N=O_{\Pp}\left(N^{-\frac1{2\alpha}}\right).
\end{equation}
This proves \eqref{mf:signal-rate}. Moreover,
\begin{equation}\label{mf:field-equivalence}
 \begin{aligned}
 \sup_{t\in[0,T]}\|\eta_t^N-\widetilde\eta_t^N\|_{H^{-s}}
\ \ &\byeq{mf:dirac-bound}{\le}\ \ C_s\sqrt N D_T^N \le C_{s,T}\sqrt N V_T^N \ \ \byeq{mf:strip-rate}{=}\ \ O_{\Pp}\left(N^{\frac{\alpha-1}{2\alpha}}\right)
 \xrightarrow[N\to\infty]{\mathbf{P}}0.
\end{aligned}
\end{equation}

Let $\widetilde K_N\coloneqq \Law(\widetilde\eta^N\mid\cG_T^0)$ and choose a bounded compatible $J_1$ metric $d$ dominated by truncated uniform distance. Conditional coupling gives $d_{\mathrm{BL}}(K_N,\widetilde K_N) \le\Ee\left[d(\eta^N,\widetilde\eta^N)\middle|\cG_T^0\right]$. By boundedness of $d$, $\Ee[d(\eta^N,\widetilde\eta^N)\mid\cG_T^0] \xrightarrow[N\to\infty]{\substack{\eqref{mf:field-equivalence}\\ L^1}}0$. Since $d_{\mathrm{BL}}(\widetilde K_N,K)\xrightarrow[N\to\infty]{}0$ almost surely, the triangle inequality proves \eqref{mf:kernel-convergence} and \eqref{mf:stable-clt}. Write $\zeta$ for the co-moving Gaussian limit and set $\eta_t\coloneqq\trans{L_t}{\zeta_t}$. For $\phi\in C^\infty(\mathbf T^d)$ and $z\in\R^d$, the bound $\|\phi(\cdot+z)-\phi\|_{H^s}\le C|z|\|\phi\|_{H^{s+1}}$, bounded $H^{-s}$ paths, and $\sum_{t\in(0,T]}|\Delta L_t|<\infty$ give absolute convergence of the jump sum. Fix the master and truncate the accepted jumps of $L$, without recomputing labels, to obtain $L^{[\delta]}$. Writing $\operatorname{Var}_{[0,T]}$ for total variation,

\begin{equation}\label{mf:accepted-truncation}
 \begin{aligned}
 \|L-L^{[\delta]}\|_\infty
 &\le\operatorname{Var}_{[0,T]}(L-L^{[\delta]}) =\sum_{\substack{t\in(0,T]\\|\Delta L_t|\le\delta}}|\Delta L_t|
 \xrightarrow[\delta\downarrow0]{}0.
\end{aligned}
\end{equation}

Apply the product rule to $\phi(\cdot+L_t^{[\delta]})$ in the continuous co-moving equation. Uniform $H^{s+2}$ convergence of these tests and conditional isometry let us pass to the limit in the drift and martingale terms. For $a,b,z\in\R^d$,

\[
 \begin{aligned}
 \|\phi(\cdot+a+z)-\phi(\cdot+a)
       -\phi(\cdot+b+z)+\phi(\cdot+b)\|_{H^s} &\le C|z||a-b|\|\phi\|_{H^{s+2}}.
\end{aligned}
\]

The errors from retained and omitted jumps are therefore bounded by

\[
 \begin{aligned}
 &C_\phi\sup_{t\in[0,T]}\|\zeta_t\|_{H^{-s}}\left[
 \|L-L^{[\delta]}\|_\infty\operatorname{Var}_{[0,T]}(L)\right.\\
 &\hspace{42mm}\left.+\operatorname{Var}_{[0,T]}(L-L^{[\delta]})\right]
 \underset{\delta\downarrow0}{\ \ \byeq{mf:accepted-truncation}{\longrightarrow}\ \ }0.
 \end{aligned}
\]

This proves \eqref{mf:limit-equation}, with $M_t(\phi)\coloneqq\int_0^t\dd\widehat M_r\left(\phi(\cdot+L_{r-})\right)$, where $\widehat M$ is the co-moving martingale. Conditional isometry gives \eqref{mf:bracket}. Conversely, testing any solution against $\phi(\cdot-L_t^{[\delta]})$ cancels retained jumps and bounds the omitted terms by

\[
 C_\phi\sup_{t\in[0,T]}\|\eta_t\|_{H^{-s}}
       \operatorname{Var}_{[0,T]}(L-L^{[\delta]})\xrightarrow[\delta\downarrow0]{}0.
\]

The same test convergence and conditional isometry show that $\trans{-L}{\eta}$ solves the continuous co-moving equation. Backward duality \eqref{mf:backward} gives uniqueness for the prescribed joint conditional law of $(\eta_0,M)$.
\end{proof}
\subsection{Covariance and finite cutoffs}
For $t\in[0,T]$ and $\phi\in C^\infty(\mathbf T^d)$, let $\widehat v^{t,\phi}$ solve \eqref{mf:backward} with terminal value $\phi(\cdot+L_t)$. For $0\le r\le t$, define

\[
 v_r^{t,\phi}:\left|\begin{aligned}
 \mathbf T^d&\longrightarrow \R\\
 x&\longmapsto \widehat v_r^{t,\phi}(x-L_r).
 \end{aligned}\right.
\]

Thus $v_t^{t,\phi}=\phi$ and $v_{r-}^{t,\phi}(x)=v_r^{t,\phi}(x+z)$ at accepted atoms. For $t,u\in[0,T]$ and $\phi,\psi\in C^\infty(\mathbf T^d)$, backward duality and conditional It\^o isometry give

\begin{equation}\label{mf:covariance}
 \begin{aligned}
 \operatorname{Cov}(\eta_t(\phi),\eta_u(\psi)\mid\cG_T^0) &=\mu_0(v_0^{t,\phi}v_0^{u,\psi})
   -\mu_0(v_0^{t,\phi})\mu_0(v_0^{u,\psi}) +\int_0^{t\wedge u}
 \mu_r\left((\nabla v_r^{t,\phi})^\top
                  a\nabla v_r^{u,\psi}\right)\dd r.
\end{aligned}
\end{equation}

Here $\mu_r=\trans{L_r}{\nu_r}$, and the backward tests are deterministic when the full master is fixed.
\begin{remark}[Finite cutoff and the absence of an intensity derivative]
\label{mf:cutoff-remark}
At fixed cutoff $\varepsilon>0$, the conditional-kernel CLT holds almost surely for $0<\alpha<2$. Let $K_\varepsilon\in\N_0$ be the number of master atoms with $|z|>\varepsilon$ in $[0,T]$, and enumerate them as $(\tau_k,z_k,u_k)$, $k\in\{1,\ldots,K_\varepsilon\}$. Compensation gives $\Delta_k\coloneqq|u_k-\mu_{\tau_k-}^\varepsilon(f)|>0$ for every candidate. The two empirical systems agree until their first differing label, so \eqref{mf:sampling-moments} gives

\[
 \begin{aligned}
 \Pp(\text{a differing label before }T\mid\cG_T^0) \ \ &\byeq{mf:sampling-moments}{\le}\ \ 
 \frac{C_T}{N}\sum_{k=1}^{K_\varepsilon}\Delta_k^{-2}
 \xrightarrow[N\to\infty]{}0
 \quad\text{almost surely}.
\end{aligned}
\]

The cutoff acceptance map is locally constant almost surely. Without cutoff, \eqref{mf:strip-rate} gives a negligible weighted discrepancy for $\alpha<1$.
\end{remark}

\section{Rank estimates and response proofs}\label{app:rank}
\subsection{Hall tails and inverse estimates}\label{app:rank-inverse}
\begin{lemma}[A sufficient Hall condition]\label{rk:hall-sufficient}
Suppose $\overline F$ is continuous and strictly decreasing for all sufficiently large arguments and, for some $\epsilon>0$,

\begin{equation}\label{rk:strong-hall}
 \overline F(x)=cx^{-\alpha}
       \{1+dx^{-\rho}+O(x^{-\rho-\epsilon})\},\qquad x\to\infty.
\end{equation}

Then Assumption~\ref{rk:inverse} holds. With $\eta\coloneqq \min\{\rho,\epsilon\}$,

\begin{equation}\label{rk:inverse-expansion}
\begin{aligned}
 Q(v)&=\left(\frac{c}{v}\right)^{\frac{1}{\alpha}}
   +\frac d\alpha\left(\frac{c}{v}\right)^{\frac{1-\rho}{\alpha}} +O\left(\left(\frac{c}{v}\right)^{\frac{1-\rho-\eta}{\alpha}}\right),
 \qquad v\downarrow0.
\end{aligned}
\end{equation}

\end{lemma}
\begin{proof}
Set $x_v\coloneqq \left(\tfrac{c}{v}\right)^{\frac{1}{\alpha}}$ and $\delta_v\coloneqq \tfrac{Q(v)}{x_v}-1$. Regular variation gives $\delta_v\xrightarrow[v\downarrow0]{}0$. Eventual continuity and strict monotonicity give $\overline F(Q(v))=v$ for small $v$, hence

\[
 1=(1+\delta_v)^{-\alpha}
 \left\{1+dx_v^{-\rho}(1+\delta_v)^{-\rho}
           +O(x_v^{-\rho-\epsilon})\right\}.
\]

Expanding both factors gives

\[
 0=-\alpha\delta_v+dx_v^{-\rho}
    +O(\delta_v^2+x_v^{-\rho}|\delta_v|+x_v^{-\rho-\epsilon}).
\]

Absorbing the terms containing $|\delta_v|$ first gives $\delta_v=O(x_v^{-\rho})$. Substitution then yields

\[
 \delta_v-\frac d\alpha x_v^{-\rho}
 =O(x_v^{-2\rho}+x_v^{-\rho-\epsilon})=O(x_v^{-\rho-\eta}),
\]

which proves \eqref{rk:inverse-expansion}. For some $v_0\in(0,1)$ and $R\le y\le Nv_0$,

\begin{equation}\label{rk:inverse-remainder}
 |r_N^{-1}(g_N-g)(y)-h(y)|
 \ \ \byeq{rk:inverse-expansion}{\le}\ \ 
 C N^{-\frac{\eta}{\alpha}}y^{-\frac{1-\rho-\eta}{\alpha}}.
\end{equation}

Use $k(v)\coloneqq Q(v)-\left(\tfrac{c}{v}\right)^{\frac{1}{\alpha}}$ for $v>0$ in all three regimes and put $a\coloneqq \tfrac{2(1-\rho)}{\alpha}$. If $a>1$,

\[
 N^{-\frac{2\eta}{\alpha}}\int_R^{Nv_0}y^{-a+\frac{2\eta}{\alpha}}\dd y
 =\begin{cases}
 O_R\left(N^{-\frac{2\eta}{\alpha}}\right),&-a+\frac{2\eta}{\alpha}<-1,\\
 O_R\left(N^{-\frac{2\eta}{\alpha}}\log N\right),&-a+\frac{2\eta}{\alpha}=-1,\\
 O_R(N^{1-a}),&-a+\frac{2\eta}{\alpha}>-1
 \end{cases}
 \xrightarrow[N\to\infty]{}0.
\]

On the remaining ranks, $y=Nv$ gives

\[
 \int_{Nv_0}^\infty|r_N^{-1}(g_N-g)|^2\dd y
 =N^{1-a}\int_{v_0}^\infty k(v)^2\dd v\xrightarrow[N\to\infty]{}0.
\]

The integral is finite since $Q$ is bounded on $[v_0,1]$ and zero afterward. Also $\int_{Nv_0}^\infty h(y)^2\dd y\xrightarrow[N\to\infty]{}0$, proving \eqref{rk:inverse-l2}. If $a=1$, then $h(y)=\left(\tfrac{d\sqrt c}{\alpha}\right)y^{-\frac{1}{2}}$, and

\[
\begin{aligned}
 \int_R^{Nv_0}|h(y)|\,|r_N^{-1}(g_N-g)(y)-h(y)|\dd y
 \ \ &\byeq{rk:inverse-remainder}{\le}\ \ C N^{-\frac{\eta}{\alpha}}
       \int_R^{Nv_0}y^{-1+\frac{\eta}{\alpha}}\dd y =O_R(1),
\end{aligned}
\]

and

\[
\begin{aligned}
 \int_R^{Nv_0}|r_N^{-1}(g_N-g)(y)-h(y)|^2\dd y
 \ \ &\byeq{rk:inverse-remainder}{\le}\ \ C N^{-\frac{2\eta}{\alpha}}
       \int_R^{Nv_0}y^{-1+\frac{2\eta}{\alpha}}\dd y =O_R(1).
\end{aligned}
\]

Consequently,

\[
\begin{aligned}
 r_N^{-2}\int_R^{Nv_0}(g_N-g)^2\dd y
 &=\frac{cd^2}{\alpha^2}\log\left(\frac{Nv_0}{R}\right)+O_R(1) =\sigma_c^2\log N+O_R(1).
\end{aligned}
\]

The remaining integral is $O(r_N^2)$, and

\[
 \sup_{y\ge R}\frac{|g_N(y)-g(y)|}{q_N}
 =O_R\left((\log N)^{-\frac{1}{2}}\right)\xrightarrow[N\to\infty]{}0.
\]

This proves \eqref{rk:critical-bracket}. The Lindeberg indicator in \eqref{rk:critical-lindeberg} eventually vanishes. The argument includes $d=0$. If $a<1$, then $k(v)^2\ \ \byeq{rk:inverse-expansion}{=}\ \ O(v^{-a})$ near zero. Together with the preceding tail bound, this yields $k\in L^2(0,\infty)$.
\end{proof}
\subsection{Proof of the rank trichotomy}\label{app:rank-tangent}
\begin{proof}[Proof of Theorem~\ref{rk:tangent}]
In the first regime put $e_N\coloneqq r_N^{-1}(g_N-g)-h$. For $R>0$, let $(t_j,y_j,e_j)_{j=1}^{K_R}$ be the Poisson atoms in $(0,T]\times(0,R]\times\{-1,1\}$. This collection is finite and, when nonempty, has a positive minimum rank almost surely. Symmetry gives

\[
 \sup_{0\le t\le T}|I_t(e_N\ind_{(0,R]})|
 \le\sum_{j=1}^{K_R}|e_N(y_j)|\xrightarrow[N\to\infty]{}0
 \quad\text{almost surely}.
\]

On $y>R\ge1$, Doob's inequality gives

\begin{align*}
 \Ee\left[\sup_{0\le t\le T}|I_t(e_N\ind_{(R,\infty)})|^2\right]
 &\le4T\int_R^\infty e_N(y)^2\dd y =4T\int_R^\infty|r_N^{-1}(g_N-g)(y)-h(y)|^2\dd y \xrightarrow[N\to\infty]{\eqref{rk:inverse-l2}}0.
\end{align*}

This proves \eqref{rk:stable-tangent}, hence joint uniform convergence with the unchanged coordinate $S$. Moreover,

\[
 h\in L^2(R,\infty)\quad\Longleftrightarrow\quad
 2\frac{1-\rho}{\alpha}>1\qquad(d\ne0),
\]

and, for $z>0$, $\int_0^\infty\ind_{\{|h(y)|>z\}}\dd y =c\left|\frac{d}{\alpha}\right|^\kappa z^{-\kappa}$, which identifies the Lévy measure. For $t\in[0,T]$ and $u,v\in\R$,

\begin{equation*}
 \log\Ee\left[e^{i(uS_t+vZ_t)}\right]
    =t\int_0^\infty\left[\cos(ug(y)+vh(y))-1\right]\dd y.
\end{equation*}

The integral converges absolutely: $g,h$ are square-integrable at infinity and the integrand is bounded near zero. At criticality, fix $R\ge1$. The martingale $M^{N,R}\coloneqq q_N^{-1}I((g_N-g)\ind_{(R,\infty)})$ satisfies

\[
 \langle M^{N,R}\rangle_t
   =tq_N^{-2}\int_R^\infty(g_N-g)^2\dd y\xrightarrow[N\to\infty]{} t\sigma_c^2.
\]

Equation \eqref{rk:critical-lindeberg} is its Lindeberg condition, so the martingale functional CLT applies \citep{EthierKurtz1986,Rebolledo1980}. For the finite-rank term, local uniform convergence at its finitely many atoms gives

\[
 \left\|q_N^{-1}I((g_N-g)\ind_{(0,R]})\right\|_\infty
 =\frac{r_N}{q_N}O_{\Pp}(1)=o_{\Pp}(1).
\]

Define $S^{\le R}\coloneqq I(g\ind_{(0,R]})$, independent of $M^{N,R}$. Then

\begin{align*}
 \Ee\left[\sup_{0\le t\le T}|S_t-S_t^{\le R}|^2\right]
 &\le4T\int_R^\infty g(y)^2\dd y \ \ \byeq{rk:maps}{=}\ \ 
 4Tc^{\frac{2}{\alpha}}\int_R^\infty y^{-\frac{2}{\alpha}}\dd y =4Tc^{\frac{2}{\alpha}}\frac{R^{1-\frac{2}{\alpha}}}{\frac{2}{\alpha}-1}
 \xrightarrow[R\to\infty]{}0.
\end{align*}

Let $\Phi,\Psi:D([0,T];\R)\to\R$ be bounded Lipschitz functions for a bounded compatible $J_1$ metric dominated by the truncated uniform distance. For each fixed $R\ge1$, independence gives

\[
\begin{aligned}
 \Ee\left[\Phi(S^{\le R})\Psi(M^{N,R})\right]
 =\Ee\left[\Phi(S^{\le R})\right]\Ee\left[\Psi(M^{N,R})\right]
 &\xrightarrow[N\to\infty]{}
 \Ee\left[\Phi(S^{\le R})\right]\Ee\left[\Psi(\sigma_cB^{\rm rk})\right].
\end{aligned}
\]

For fixed $R$, the vanishing finite-rank term lets us replace $M^{N,R}$ by $q_N^{-1}(J^N-S)$. The preceding $L^2$ bound controls the error from replacing $S^{\le R}$ by $S$, uniformly in $N$. Letting $R\to\infty$ therefore yields

\[
 \Ee\left[\Phi(S)\Psi\left(q_N^{-1}(J^N-S)\right)\right]
 \xrightarrow[N\to\infty]{}
 \Ee\left[\Phi(S)\right]\Ee\left[\Psi(\sigma_cB^{\rm rk})\right].
\]

The marginal convergences give tightness in the product of the two path spaces. This factorization identifies every subsequential joint limit as $(S,\sigma_cB^{\rm rk})$ with independent coordinates. A continuous second coordinate allows the same $J_1$ time change for both coordinates. In the third regime, writing $y=Nv$ yields

\[
\begin{aligned}
 N^{\frac{1}{\alpha}-\frac{1}{2}}(g_N(Nv)-g(Nv))
 &=N^{-\frac{1}{2}}\left(Q(v)-\left(\frac{c}{v}\right)^{\frac{1}{\alpha}}\right) =N^{-\frac{1}{2}}k(v).
\end{aligned}
\]

Thus the normalized error has bracket

\[
 tN^{\frac{2}{\alpha}-1}\int_0^\infty(g_N-g)^2\dd y
 =t\int_0^\infty k(v)^2\dd v=t\sigma_Q^2.
\]
Dominated convergence gives its Lindeberg condition: $$t\int_0^\infty k(v)^2\ind_{\{|k(v)|>\varepsilon\sqrt N\}}\dd v \xrightarrow[N\to\infty]{}0.$$ Since $g_N-g\in L^2(0,\infty)$, symmetry identifies the error with its fully compensated integral. Local uniform convergence and $N^{\frac{1}{\alpha}-\frac{1}{2}}r_N\xrightarrow[N\to\infty]{}0$ remove each fixed rank block. The same truncation proves independence and vector $J_1$ convergence.
\end{proof}
\subsection{Microscopic transfer and source exclusion}\label{app:rank-microscopic}
Throughout this subsection, $0<\alpha<1$ and $\nu_N$ denotes the microscopic mark measure in \eqref{rk:micro-measure}.
\begin{lemma}[Localized quantile estimate]\label{rk:localized}
Under \eqref{rk:tail-bound}, for each $R>1$, with $y_R\coloneqq cR^{-\alpha}$,

\begin{equation}\label{rk:localized-estimate}
 \int_{y_R}^\infty|g_N(y)-g(y)|^2\dd y\le C_R b_N^2,
\end{equation}

where $b_N$ is given by \eqref{rk:rate}. Uniformly in $N$,
\begin{subequations}\label{rk:uniform-tails}

\begin{equation}
 \nu_N(\{z\in\R:|z|>x\})\le Cx^{-\alpha},
 \qquad x>0,
\end{equation}

and

\begin{equation}
 \int_{\{z\in\R:|z|\le x\}}|z|\nu_N(\dd z)\le Cx^{1-\alpha},
 \qquad x>0.
\end{equation}

\end{subequations}
\end{lemma}
\begin{proof}
For sufficiently large $K$ and small $v>0$, $\overline F\left(\left(\frac{c}{v}\right)^{\frac{1}{\alpha}}\left(1+Kv^{\frac{\rho}{\alpha}}\right)\right)\ \ \byeq{rk:tail-bound}{\le}\ \ v$, and $v\ \ \byeq{rk:tail-bound}{<}\ \ \overline F\left(\left(\frac{c}{v}\right)^{\frac{1}{\alpha}}\left(1-Kv^{\frac{\rho}{\alpha}}\right)\right)$. The definition of $Q$ therefore yields

\begin{equation}\label{rk:quantile-bound}
 Q(v)=\left(\frac{c}{v}\right)^{\frac{1}{\alpha}}\left\{1+O\left(v^{\frac{\rho}{\alpha}}\right)\right\},\qquad v\downarrow0.
\end{equation}

Choose $v_0\in(0,1)$ on which this bound holds. For $0<y\le Nv_0$,

\[
\begin{aligned}
 |g_N(y)-g(y)|
 &=N^{-\frac{1}{\alpha}}\left|Q\left(\frac{y}{N}\right)-\left(\frac{cN}{y}\right)^{\frac{1}{\alpha}}\right| \ \ \byeq{rk:quantile-bound}{\le}\ \ 
 C N^{-\frac{\rho}{\alpha}}y^{\frac{\rho-1}{\alpha}}.
\end{aligned}
\]

Since $Q$ is bounded on $[v_0,1]$ and vanishes on $[1,\infty)$,

\begin{align*}
 \int_{y_R}^\infty|g_N-g|^2\dd y
 &\le C N^{-\frac{2\rho}{\alpha}}
          \int_{y_R}^{Nv_0}y^{\frac{2(\rho-1)}{\alpha}}\dd y
       +C N^{1-\frac{2}{\alpha}}\\
 &\le C_R\begin{cases}
 N^{-\frac{2\rho}{\alpha}},&\rho<1-\frac{\alpha}{2},\\
 N^{1-\frac{2}{\alpha}}(1+\log N),&\rho=1-\frac{\alpha}{2},\\
 N^{1-\frac{2}{\alpha}},&\rho>1-\frac{\alpha}{2}
 \end{cases}
 \ \ \byeq{rk:rate}{=}\ \ C_R b_N^2,
\end{align*}

for large $N$, proving \eqref{rk:localized-estimate}. Increase the constant so that the tail bound holds for all positive arguments. Then

\begin{align*}
 \nu_N(\{z\in\R:|z|>x\})
 &=N\Pp\left(|U|>N^{\frac{1}{\alpha}}x\right) \le Cx^{-\alpha}.
\end{align*}

Consequently,

\begin{align*}
 \int_{\{z\in\R:|z|\le x\}}|z|\nu_N(\dd z)
 &\le\int_0^x\nu_N(\{z\in\R:|z|>v\})\dd v \le C\int_0^xv^{-\alpha}\dd v =\frac C{1-\alpha}x^{1-\alpha}.
\end{align*}

\end{proof}
\begin{proof}[Proof of Theorem~\ref{rk:transfer}]
Write the signal difference as a stable-mark disagreement term plus

\[
 R_t^N\coloneqq \int_{\substack{(0,t]\times(0,\infty)\\\times\{-1,1\}\times[0,\bar f]}}e(g_N-g)(y)
       \ind_{\{u\le\mu_{v-}^{N,\mathrm{mic}}(f)\}}
              \mathcal P^0(\dd v,\dd y,\dd e,\dd u).
\]

The integral is absolutely summable since $\alpha<1$. Restricting to $y\ge y_R$ gives $R^{N,R}$, an $L^2$ martingale by predictability and symmetry, with

\[
\begin{aligned}
 \langle R^{N,R}\rangle_t
 &=\int_0^t\mu_{v-}^{N,\mathrm{mic}}(f)\dd v
       \int_{y_R}^\infty(g_N-g)^2\dd y \ \ \byeq{rk:localized-estimate}{\le}\ \ t\bar f C_Rb_N^2.
\end{aligned}
\]

The removed block has intensity $T\bar f cR^{-\alpha}$. Hence

\begin{equation}\label{rk:rank-error-probability}
\begin{aligned}
 \Pp\left(\sup_{0\le t\le T}|R_t^N|>Mb_N\right)
 &\le T\bar f cR^{-\alpha}
    +\frac{4\Ee[\langle R^{N,R}\rangle_T]}{M^2b_N^2} \le T\bar f cR^{-\alpha}+\frac{4T\bar f C_R}{M^2}.
\end{aligned}
\end{equation}

Letting $R$, then $M$, tend to infinity gives $\|R^N\|_\infty=O_{\Pp}(b_N)$. Put $$A_t^N\coloneqq \int_0^t |\mu_{v-}^{N,\mathrm{mic}}(f)-\mu_{v-}^N(f)|\dd v$$ and let $V_t^N$ sum $g(y)$ over candidates accepted by exactly one threshold. By predictability,

\begin{equation*}
 \Pp(V_\tau^N>x,\ A_\tau^N\le v) \overset{\text{Lemma~\ref{mf:strip}}}{\le} Cvx^{-\alpha}
                 \qquad(\tau\le T).
\end{equation*}

Use paired real lifts with the same initial values and Brownian paths, and set $$D_t^N\coloneqq \frac1N\sum_{i=1}^N\sup_{0\le r\le t} |X_r^{i,N,\mathrm{mic}}-X_r^{i,N}|.$$ The Brownian terms cancel, while Lipschitz continuity gives $$D_t^N\le\|R^N\|_{\infty,[0,t]}+V_t^N+C\int_0^tD_r^N\dd r\quad\text{and}\quad A_t^N\le\|f'\|_\infty\int_0^tD_r^N\dd r.$$ Gronwall's inequality yields $$D_t^N+A_t^N \le C_T\left(\sup_{0\le v\le t}|R_v^N|+V_t^N\right).$$ Thus $A_T^N\overset{\text{Lemma~\ref{mf:bootstrap}}}{=} O_{\Pp}(b_N)$ and $V_T^N \overset{\text{Lemma~\ref{mf:bootstrap}}}{=} O_{\Pp}\left(b_N^{\frac{1}{\alpha}}\right)=o_{\Pp}(b_N)$. Thus the comparison error is $O_{\Pp}(b_N)$. Moreover,

\[
 \sqrt N b_N=
 \begin{cases}
 N^{\frac{1}{2}-\frac{\rho}{\alpha}},&\rho<1-\frac{\alpha}{2},\\
 N^{\frac{\alpha-1}{\alpha}}\sqrt{1+\log N},&\rho=1-\frac{\alpha}{2},\\
 N^{\frac{\alpha-1}{\alpha}},&\rho>1-\frac{\alpha}{2}
 \end{cases}
 \xrightarrow[N\to\infty]{}0.
\]

For $\lambda,\lambda'\in\Prob(\mathbf T)$, Sobolev embedding and Kantorovich duality give

\[
\begin{aligned}
 \|\lambda-\lambda'\|_{H^{-s}}
 &=\sup_{\substack{\phi\in H^s(\mathbf T)\\\|\phi\|_{H^s}\le1}}|(\lambda-\lambda')(\phi)| \le\sup_{\substack{\phi\in H^s(\mathbf T)\\\|\phi\|_{H^s}\le1}}\|\phi'\|_\infty W_1(\lambda,\lambda') \le C_sW_1(\lambda,\lambda').
\end{aligned}
\]

The stable limit then follows from the comparison and Slutsky's theorem. The conditional bounded--Lipschitz distance between the two field laws is at most

\[
 \Ee\left[\min\{1,C\sqrt N
       \sup_{0\le t\le T}W_1(\mu_t^{N,\mathrm{mic}},\mu_t^N)\}
           \middle|\mathcal G_T^0\right]
 \xrightarrow[N\to\infty]{L^1(\Pp)}0.
\]

Equation \eqref{mf:kernel-convergence} completes the proof.
\end{proof}
\begin{corollary}[Source exclusion without reset]\label{rk:exclusion}
Suppose that the hypotheses of Theorem~\ref{rk:transfer} hold. Independently mark each Poisson candidate with a uniform random variable in $(0,1)$, independent of the original master, initial sample and Brownian motions. At an event accepted by the exclusion system, use this mark to choose one source $j$ with probabilities $$\frac{f(X_{t-}^{j,N,\mathrm{exclude}})} {\sum_{i=1}^N f(X_{t-}^{i,N,\mathrm{exclude}})}, \, j=1,\ldots,N$$ and omit only that particle from the common translation. No event is accepted when the denominator is zero. Keep the same candidates, initial sample and Brownian motions as in the all-recipient system, with each system using its own pre-jump acceptance threshold. Then

\[
 \sup_{0\le t\le T}W_1(\mu_t^{N,\mathrm{exclude}},
                    \mu_t^{N,\mathrm{mic}})=O_{\Pp}(N^{-1}),
\]

and the empirical CLT and conditional-kernel conclusion of Theorem~\ref{rk:transfer} remain valid. Conditioning is on the original master sigma-field $\mathcal G_T^0$, which excludes the source-selection marks.
\end{corollary}
\begin{proof}
Enlarge the filtration by the source marks revealed before the current time. Independent marking multiplies the candidate compensator by Lebesgue measure on $(0,1)$. Its projection remains unchanged. Both source probabilities and acceptance thresholds are predictable. At a shared event the excluded source adds at most $|z|\wedge D$ to the paired distance, where $D\coloneqq \operatorname{diam}(\mathbf T)$. Thus the total mean cost is bounded by

\[
 E_N\coloneqq \frac1N\int_{[0,T]\times\R\times[0,\bar f]}
             (|z|\wedge D)\,\mathcal N_N^0(\dd t,\dd z,\dd u).
\]

Here $\mathcal N_N^0$ is the image of the master restricted to $0<y<N$ under the microscopic mark map, and has mark measure $\nu_N$. Consequently,

\begin{align*}
 \Ee\left[E_N\right]
 &=\frac{T\bar f}{N}\int_0^D\nu_N(\{z\in\R:|z|>v\})\dd v \ \ \byeq{rk:uniform-tails}{\le}\ \ 
 \frac{T\bar f C D^{1-\alpha}}{(1-\alpha)N}=O(N^{-1}).
\end{align*}

Let

\[
 D_t^N\coloneqq \frac1N\sum_{i=1}^N\sup_{0\le r\le t}
 d_{\mathbf T}(X_r^{i,N,\mathrm{exclude}},X_r^{i,N,\mathrm{mic}}),
\]

and $A_t^N\coloneqq \int_0^t |\mu_{r-}^{N,\mathrm{exclude}}(f)-\mu_{r-}^{N,\mathrm{mic}}(f)|\dd r$, and let $V_t^N$ sum $|z|$ over candidates accepted by exactly one system. Common translations preserve $d_{\mathbf T}$. The Lipschitz drift, synchronous Brownian coupling and preceding event cost give $D_t^N+A_t^N\le C_T(E_N+V_t^N)$ for $t\le T$. Thus $$A_T^N \overset{\substack{\eqref{rk:uniform-tails}\\ \text{Lemmas~\ref{mf:strip},~\ref{mf:bootstrap}}}}{=} O_{\Pp}(N^{-1})\quad\text{and}\quad V_T^N\overset{\substack{\eqref{rk:uniform-tails}\\ \text{Lemmas~\ref{mf:strip},~\ref{mf:bootstrap}}}}{=}O_{\Pp}\left(N^{-\frac{1}{\alpha}}\right)=o_{\Pp}(N^{-1}).$$ 

Thus $D_T^N=O_{\Pp}(N^{-1})$, which proves the comparison. The conditional coupling estimate from Theorem~\ref{rk:transfer}, which integrates over the source marks, transfers both conclusions.
\end{proof}
\subsection{Endogenous response estimates}\label{app:rank-response}
\subsubsection{Well-posedness of the response}
For \eqref{cr:response}, bounded derivatives imply that successive Picard iterates differ after $n\in\N_0$ steps by at most $C\|Z\|_\infty\tfrac{(Ct)^n}{n!}$ for $t\in[0,T]$. This proves convergence. The same estimate for two solutions proves uniqueness. Gronwall and Sobolev embedding yield $\|V\|_\infty+\|\mathbf c\|_\infty\le C_T\|Z\|_\infty$. For $\phi\in H^s(\mathbf T)$, they also give $$|\chi_t(\phi)|\le C_T\|Z\|_\infty\|\phi'\|_\infty \le C_T\|Z\|_\infty\|\phi\|_{H^s}.$$ Hence $\chi_t\in H^{-s}(\mathbf T)$, with $\|\chi\|_{\infty,H^{-s}}\le C_T\|Z\|_\infty$. Both $\mathbf c$ and $\chi$ are master-measurable.
\subsubsection{Signal and field limits}
\begin{proof}[Proof of Theorem~\ref{cr:theorem}]\leavevmode
\paragraph{Signal comparison.} Set $D_N\coloneqq r_N^{-1}(g_N-g)$ and $\iota_t^N(u)\coloneqq \ind_{\{u\le\mu_{t-}^{N,\mathrm{mic}}(f)\}}$. Split the signal difference as $L_t^{N,\mathrm{mic}}-L_t=F_t^N+K_t^N$, where

\[
 F_t^N\coloneqq \int_{\substack{(0,t]\times(0,\infty) \times\{-1,1\}\times[0,\bar f]}}e(g_N-g)(y)\iota_v^N(u)
                  \mathcal P^0(\dd v,\dd y,\dd e,\dd u),
\]

and

\[
 K_t^N\coloneqq \int_{\substack{(0,t]\times(0,\infty) \times\{-1,1\}\times[0,\bar f]}}eg(y)(\iota_v^N-\iota_v)(u)
                  \mathcal P^0(\dd v,\dd y,\dd e,\dd u).
\]

Let

\[
 V_t^N\coloneqq \int_{\substack{(0,t]\times(0,\infty) \times\{-1,1\}\times[0,\bar f]}}
 g(y)|\iota_v^N-\iota_v|(u)\dd\mathcal P^0,\qquad A_t^N\coloneqq \int_0^t|\mu_{u-}^{N,\mathrm{mic}}(f)-\mu_{u-}(f)|\dd u.
\]
Then $\|K^N\|_{\infty,[0,t]}\le V_t^N$. Here $b_N=r_N$. The predictable-threshold bound gives $\|F^N\|_\infty \byeq{rk:rank-error-probability}{=}\ O_{\Pp}(r_N)$. Put $R_N\coloneqq \sup_{0\le t\le T}|\widetilde\mu_t^N(f)-\mu_t(f)|.$ Thus $R_N \overset{\text{Lemma~\ref{mf:prescribed}}}{=}O_{\Pp}\left(N^{-\frac{1}{2}}\right)$, and synchronous coupling gives $$A_t^N\le C_T\{R_N+\|F^N\|_{\infty,[0,t]}+V_t^N\}.$$ Since $N^{-\frac{1}{2}}\le r_N$,
\begin{subequations}\label{cr:strip-rate}

\begin{equation}
 A_T^N \overset{\text{Lemmas~\ref{mf:strip},~\ref{mf:bootstrap}}}{=} O_{\Pp}(r_N),
\end{equation}

and

\begin{equation}
 V_T^N \overset{\text{Lemmas~\ref{mf:strip},~\ref{mf:bootstrap}}}{=}O_{\Pp}\left(r_N^{\frac{1}{\alpha}}\right)=o_{\Pp}(r_N).
\end{equation}

\end{subequations}
Since

\[
 \sup_{0\le t\le T}|\mu_t^{N,\mathrm{mic}}(f)-\mu_t(f)|
 \le R_N+C_T(\|F^N\|_\infty+V_T^N),
\]

we also obtain

\begin{equation}\label{cr:threshold-rate}
 \sup_{0\le t\le T}|\mu_t^{N,\mathrm{mic}}(f)-\mu_t(f)|=O_{\Pp}(r_N).
\end{equation}

Fix $R>1$. The block $0<y\le R$ contains finitely many candidates before $T$, with a positive minimum rank almost surely. Predictable compensation gives

\[
 \Ee\left[\int_{\substack{(0,T]\times(0,R] \times\{-1,1\}\times[0,\bar f]}}
       \ind_{\{u=\mu_{v-}(f)\}}\dd\mathcal P^0\right]=0.
\]

Thus, if $(t_j,y_j,e_j,u_j)_{j=1}^{K_R}$ are these candidates,

\[
 \Delta_R\coloneqq \min_{1\le j\le K_R}|u_j-\mu_{t_j-}(f)|>0
 \quad\text{almost surely},
\]

with $\Delta_R\coloneqq 1$ when $K_R=0$. Their labels coincide whenever $\sup_{0\le t\le T}|\mu_t^{N,\mathrm{mic}}(f)-\mu_t(f)|<\Delta_R$. By \eqref{cr:threshold-rate} this event has probability tending to one. On it, the uniform error between their rescaled sums is at most $$\sum_{j=1}^{K_R}|D_N(y_j)-h(y_j)|\xrightarrow[N\to\infty]{}0.$$ On $y>R$, symmetry centers both integrals, and Doob's inequality gives

\[
 \Ee\left[\|r_N^{-1}F^{N,>R}\|_\infty^2\right]
       \le4T\bar f\int_R^\infty D_N(y)^2\dd y,
\]

and $\Ee\left[\|Z^{>R}\|_\infty^2\right] \le4T\bar f\int_R^\infty h(y)^2\dd y$. The superscript $>R$ denotes restriction to those ranks. We have

\begin{align*}
 \int_R^\infty D_N(y)^2\dd y
 &\le2\int_R^\infty|D_N(y)-h(y)|^2\dd y
      +2\int_R^\infty h(y)^2\dd y,
\end{align*}

so

\begin{align*}
 \limsup_{N\to\infty}\int_R^\infty D_N(y)^2\dd y
\ \ &\byeq{rk:inverse-l2}{\le}\ \ 
 2\int_R^\infty h(y)^2\dd y\xrightarrow[R\to\infty]{}0.
\end{align*}

Thus the tail bounds vanish as $N\to\infty$ and then $R\to\infty$. Together with the finite-rank comparison, this gives $\|r_N^{-1}F^N-Z\|_\infty\xrightarrow[N\to\infty]{\mathbf{P}}0$. Finally, by this convergence and \eqref{cr:strip-rate},

\[
\begin{aligned}
 \|r_N^{-1}(L^{N,\mathrm{mic}}-L)-Z\|_\infty &\le\|r_N^{-1}F^N-Z\|_\infty+r_N^{-1}V_T^N \ \ =\ \ o_{\Pp}(1),
\end{aligned}
\]

which proves \eqref{cr:signal-limit}.
\paragraph{Finite-particle variation.} Fix $\ell,z\in D([0,T];\R)$ with $\ell_0=z_0=0$ and $\delta\in\R$. Use the same private inputs for systems $X^{i,N,\delta z}$ and $X^{i,N,0}$ driven by $\ell+\delta z$ and $\ell$, and let $\mu_t^{N,\delta z}\coloneqq N^{-1}\sum_{i=1}^N\delta_{X_t^{i,N,\delta z}}$. Define $V^{i,N}[z]\in D([0,T];\R)$ and $\chi^N[z]\in D([0,T];H^{-s}(\mathbf T))$ by

\begin{equation*}
\begin{aligned}
 V_t^{i,N}[z]&={}z_t
 +\int_0^t B_x(X_u^{i,N,0},m(\mu_u^{N,0}))V_u^{i,N}[z]\dd u +\sum_{j=1}^q\int_0^t B_{m_j}(X_u^{i,N,0},m(\mu_u^{N,0}))c_j^N(u)\dd u.
\end{aligned}
\end{equation*}

Here

\begin{equation*}
 c_j^N(u)\coloneqq \frac1N\sum_{i=1}^N g_j'(X_u^{i,N,0})V_u^{i,N}[z],
 \qquad j\in\{1,\ldots,q\},
\end{equation*}

and $$\chi_t^N[z]\coloneqq -\frac1N\sum_{i=1}^N V_t^{i,N}[z]\partial_x\delta_{X_t^{i,N,0}}.$$ Using paired Euclidean lifts, put $d_i\coloneqq X^{i,N,\delta z}-X^{i,N,0}\in D([0,T];\R)$ for $i\in\{1,\ldots,N\}$ and $$\Delta m_j\coloneqq N^{-1}\sum_{i=1}^N[g_j(X^{i,N,\delta z})-g_j(X^{i,N,0})]$$ for $j\in\{1,\ldots,q\}$. Write $\Delta m\coloneqq (\Delta m_1,\ldots,\Delta m_q)\in D([0,T];\R^q)$. Bounded first derivatives and Gronwall give
\begin{subequations}\label{cr:first-variation-bound}

\begin{equation}
 \max_{1\le i\le N}\|V^{i,N}[z]\|_\infty\le C_T\|z\|_\infty,
\end{equation}

and

\begin{equation}
 \max_{1\le i\le N}\|d_i\|_\infty\le C_T|\delta|\|z\|_\infty.
\end{equation}

\end{subequations}
For $u\in[0,T]$, write $m_u^0\coloneqq m(\mu_u^{N,0})$. Taylor's formula defines the real remainders $\varepsilon_j(u)$ and $\varepsilon_i^B(u)$ through

\begin{align*}
 \Delta m_j(u)
 &=\frac1N\sum_{i=1}^N g_j'(X_u^{i,N,0})d_i(u)+\varepsilon_j(u).
\end{align*}

For the drift,

\begin{align*}
 B(X_u^{i,N,\delta z},m_u^0+\Delta m(u))-B(X_u^{i,N,0},m_u^0)
 =\ & B_x(X_u^{i,N,0},m_u^0)d_i(u) \\\qquad &+\sum_{j=1}^q B_{m_j}(X_u^{i,N,0},m_u^0)\Delta m_j(u)
   +\varepsilon_i^B(u).
\end{align*}

The remainders satisfy

\begin{align*}
 \max_{1\le j\le q}|\varepsilon_j(u)|+\max_{1\le i\le N}|\varepsilon_i^B(u)|
 &\le C\max_{1\le i\le N}|d_i(u)|^2 \ \ \byeq{cr:first-variation-bound}{\le}\ \ C_T\delta^2\|z\|_\infty^2.
\end{align*}

For $E_t\coloneqq \max_{1\le i\le N}\sup_{0\le u\le t}|d_i(u)-\delta V_u^{i,N}[z]|$, the preceding bounds give
\begin{subequations}\label{cr:particle-taylor}

\begin{equation}
 E_t\le C\int_0^t E_u\dd u+C_T\delta^2\|z\|_\infty^2.
\end{equation}

Gronwall's inequality therefore yields

\begin{equation}
 E_T\le C_T\delta^2\|z\|_\infty^2.
\end{equation}

\end{subequations}
All constants are independent of $N$, $\ell$ and the sample. These pathwise estimates hold for random $z$ with bounded paths. For $s>\tfrac{5}{2}$, the map $x\mapsto\delta_x$ from $\mathbf T$ to $H^{-s}(\mathbf T)$ has bounded first and second derivatives $-\partial_x\delta_x$ and $\partial_x^2\delta_x$. In particular,

\[
 \delta_{x+d}-\delta_x+d\partial_x\delta_x
 =d^2\int_0^1(1-\theta)\partial_x^2\delta_{x+\theta d}\dd\theta,
\]

and $\sup_{x\in\mathbf T}\|\partial_x^2\delta_x\|_{H^{-s}}<\infty$. Averaging gives

\begin{equation}\label{cr:finite-taylor}
\begin{aligned}
 \|\mu^{N,\delta z}-\mu^{N,0}-\delta\chi^N[z]\|_{\infty,H^{-s}}
 &\le C\max_{1\le i\le N}\|d_i\|_\infty^2+CE_T \ \ \overset{\mathclap{\eqref{cr:first-variation-bound}\text{--}\eqref{cr:particle-taylor}}}{\le}\ \ C_T\delta^2\|z\|_\infty^2.
\end{aligned}
\end{equation}

The linear response defines a bounded linear map

\[
 \chi^N:\left|\begin{aligned}
 \{z\in D([0,T];\R):z_0=0\}&\longrightarrow D([0,T];H^{-s}(\mathbf T))\\
 z&\longmapsto \chi^N[z],
 \end{aligned}\right.
\]

when both spaces carry the uniform norm. In particular,

\begin{equation}\label{cr:response-lipschitz}
 \|\chi^N[z]-\chi^N[z']\|_{\infty,H^{-s}}
                         \le C_T\|z-z'\|_\infty.
\end{equation}

\paragraph{Response law of large numbers.} Take $\ell=L$ and $z=Z$, fix the master, and restrict to $\|Z\|_\infty\le K$. Couple $X^{i,N,0}=\widetilde X^{i,N}$ with reference particles $\bar X^i$ using the same private inputs. The solutions $\bar V^i$ of \eqref{cr:response} make $(\bar X^i,\bar V^i)_{i\ge1}$ conditionally independent and identically distributed, with $\|\bar V^i\|_\infty\le C_TK$. Set

\[
\begin{aligned}
 \mathcal E_N(u)&\coloneqq {}\sum_{j=1}^q
 \left|\frac1N\sum_{i=1}^N g_j(\bar X_u^i)-\mu_u(g_j)\right| +\sum_{j=1}^q\left|\frac1N\sum_{i=1}^N g_j'(\bar X_u^i)\bar V_u^i-c_j(u)\right|.
\end{aligned}
\]

Subtracting the response equations and applying Gronwall gives

\begin{equation}\label{cr:response-pairing}
\begin{aligned}
 \Ee_0\left[\frac1N\sum_{i=1}^N\|V^{i,N}[Z]-\bar V^i\|_\infty\right] &\le C_{K,T}\Ee_0\left[\frac1N\sum_{i=1}^N
       \|\widetilde X^{i,N}-\bar X^i\|_\infty\right] +C_{K,T}\int_0^T\Ee_0[\mathcal E_N(u)]\dd u\\
\ \ &\byeq{mf:particle-comparison}{\le}\ \ C_{K,T}N^{-\frac{1}{2}}.
\end{aligned}
\end{equation}

Each centered average in $\mathcal E_N(u)$ has conditional second moment at most $\tfrac{C_K}{N}$. Since $x\mapsto\partial_x\delta_x$ is Lipschitz into $H^{-s}$, it remains to control the centered mean of $-\bar V_t^i\partial_x\delta_{\bar X_t^i}$ uniformly in time. Set $Y^i\coloneqq \bar X^i-L\in C([0,T];\mathbf T)$ and $R^i\coloneqq \bar V^i-Z\in C([0,T];\R)$. Then $\dd Y_t^i=b_t^i\dd t+\sigma\dd W_t^i$, where $b_t^i\coloneqq B(\bar X_t^i,m(\mu_t))$. The path $R^i$ is absolutely continuous, and its almost-everywhere derivative $\dot R^i$ satisfies $\|R^i\|_\infty+\|\dot R^i\|_\infty\le C_{K,T}$. In the Hilbert space $\mathcal H\coloneqq H^{-(s-1)}(\mathbf T)$, put $A_t^{i,0}\coloneqq \delta_{Y_t^i}$ and $A_t^{i,1}\coloneqq R_t^i\delta_{Y_t^i}$. Since $s-1>\tfrac{5}{2}$, Itô's formula gives

\[
 \dd A_t^{i,j}=D_t^{i,j}\dd t+G_t^{i,j}\dd W_t^i,
 \qquad j\in\{0,1\}.
\]

The drift terms are $D_t^{i,0}\coloneqq -b_t^i\partial_x\delta_{Y_t^i}+\tfrac12a\partial_x^2\delta_{Y_t^i}$ and $D_t^{i,1}\coloneqq \dot R_t^i\delta_{Y_t^i}+R_t^iD_t^{i,0}$. The diffusion terms are defined by $G_t^{i,0}v\coloneqq -(\sigma v)\partial_x\delta_{Y_t^i}$ for $v\in\R^{d_0}$ and $G_t^{i,1}\coloneqq R_t^iG_t^{i,0}$. Here $a=\sigma\sigma^\top\in[0,\infty)$ and, for each $i,j,t$, $G_t^{i,j}:\R^{d_0}\to\mathcal H$ has Hilbert--Schmidt norm $\|G_t^{i,j}\|_{\mathrm{HS}}$. Uniformly for $i\in\N$ and $j\in\{0,1\}$, and for almost every $t\in[0,T]$, $\|D_t^{i,j}\|_{\mathcal H}+\|G_t^{i,j}\|_{\mathrm{HS}}\le C_{K,T}$. Set $$\overline D_t^{N,j}\coloneqq \tfrac1N\sum_{i=1}^N(D_t^{i,j}-\Ee_0[D_t^{i,j}])\quad\text{and}\quad M_t^{N,j}\coloneqq \tfrac1N\sum_{i=1}^N\int_0^tG_u^{i,j}\dd W_u^i.$$ Conditional independence and Doob's inequality yield

\begin{align*}
 \Ee_0[\|\overline D_t^{N,j}\|_{\mathcal H}^2]
 &=\frac1{N^2}\sum_{i=1}^N\Ee_0[
       \|D_t^{i,j}-\Ee_0[D_t^{i,j}]\|_{\mathcal H}^2] \le\frac{C_{K,T}}N,
\end{align*}

and

\begin{align*}
 \Ee_0\left[\sup_{0\le t\le T}\|M_t^{N,j}\|_{\mathcal H}^2\right]
 &\le\frac4{N^2}\sum_{i=1}^N\Ee_0\left[\int_0^T\|G_u^{i,j}\|_{\mathrm{HS}}^2\dd u\right] \le\frac{C_{K,T}}N.
\end{align*}

The centered initial average has the same variance bound for $j=0$ and is zero for $j=1$. Moreover,

\[
\begin{aligned}
 \Ee_0\left[\sup_{0\le t\le T}\left\|\int_0^t\overline D_u^{N,j}\dd u\right\|_{\mathcal H}^2\right]
 &\le T\int_0^T\Ee_0[\|\overline D_u^{N,j}\|_{\mathcal H}^2]\dd u \le\frac{C_{K,T}}N.
\end{aligned}
\]

Combining the initial, drift and martingale estimates yields

\begin{equation}\label{cr:weighted-lln}
 \Ee_0\left[\sup_{0\le t\le T}\left\|
       \frac1N\sum_{i=1}^N A_t^{i,j}-\Ee_0\left[A_t^{1,j}\right]
                         \right\|_{H^{-(s-1)}}\right]
                    \le C_{K,T}N^{-\frac{1}{2}},\qquad j=0,1.
\end{equation}

Now

\[
 -\frac1N\sum_{i=1}^N\bar V_t^i\partial_x\delta_{\bar X_t^i}
 =-\partial_x\tau_{L_t}\left(
          Z_t\frac1N\sum_{i=1}^N A_t^{i,0}+\frac1N\sum_{i=1}^N A_t^{i,1}\right),
\]

and $\partial_x:H^{-(s-1)}(\mathbf T)\to H^{-s}(\mathbf T)$ is bounded. Letting $K\to\infty$ gives

\begin{equation}\label{cr:response-lln}
 \|\chi^N[Z]-\chi\|_{\infty,H^{-s}}\xrightarrow[N\to\infty]{\substack{\eqref{cr:response-pairing},\,\eqref{cr:weighted-lln}\\ \mathbf{P}}}0.
\end{equation}

Also $\chi_t=-\partial_x\tau_{L_t}(Z_t\Ee_0\left[A_t^{1,0}\right]+\Ee_0\left[A_t^{1,1}\right])$ is càdlàg because its conditional means are continuous and $L,Z$ are càdlàg. Thus $\chi\in D([0,T];H^{-s}(\mathbf T))$.
\paragraph{Field expansion and conditional limits.} Apply \eqref{cr:finite-taylor} with $\ell=L$, $\delta=r_N$ and $z_N\coloneqq r_N^{-1}(L^{N,\mathrm{mic}}-L)$. Then $\|z_N-Z\|_\infty\xrightarrow[N\to\infty]{\substack{\eqref{cr:signal-limit}\\ \mathbf{P}}}0$ and $\|z_N\|_\infty=O_{\Pp}(1)$. Hence

\[
\begin{aligned}
 \|r_N^{-1}(\mu^{N,\mathrm{mic}}-\widetilde\mu^N)-\chi\|_{\infty,H^{-s}}
 &\le \|r_N^{-1}(\mu^{N,\mathrm{mic}}-\widetilde\mu^N)-\chi^N[z_N]\|_{\infty,H^{-s}}\\
 &\quad+\|\chi^N[z_N]-\chi^N[Z]\|_{\infty,H^{-s}}+\|\chi^N[Z]-\chi\|_{\infty,H^{-s}}\\
 \ \ &\overset{\mathclap{\eqref{cr:finite-taylor},\,\eqref{cr:response-lipschitz}}}{\le}\ \
 C_Tr_N\|z_N\|_\infty^2+C_T\|z_N-Z\|_\infty+\|\chi^N[Z]-\chi\|_{\infty,H^{-s}}
 \xrightarrow[N\to\infty]{\substack{\eqref{cr:signal-limit},\,\eqref{cr:response-lln}\\ \mathbf{P}}}0.
\end{aligned}
\]

This is \eqref{cr:field-expansion}. For $\rho<\tfrac{\alpha}{2}$, $\tfrac{N^{-\frac{1}{2}}}{r_N}\xrightarrow[N\to\infty]{}0$ and \eqref{mf:sampling-moments} removes the sampling term. At criticality, Lemma~\ref{mf:prescribed} gives conditional convergence of $\tau_{-L}\sqrt N(\widetilde\mu^N-\mu)$ in $C([0,T];H^{-s})$ for almost every master. For each such master, the map

\[
 \left|\begin{aligned}
 C([0,T];H^{-s}(\mathbf T))&\longrightarrow D([0,T];\R^2\times H^{-s}(\mathbf T))\\
 z&\longmapsto (L,Z,\tau_Lz+\chi)
 \end{aligned}\right.
\]

is continuous into vector $J_1$: the uniform distance between two outputs is at most $\|z-z'\|_{\infty,H^{-s}}$. Let $\varepsilon_N$ be the uniform error in \eqref{cr:field-expansion} plus $\|z_N-Z\|_\infty$. Then $\Ee[1\wedge\varepsilon_N]\xrightarrow[N\to\infty]{}0$, and conditional expectation gives $$\Ee[1\wedge\varepsilon_N\mid\mathcal G_T^0]\xrightarrow[N\to\infty]{L^1(\Pp)}0.$$ The conditional bounded--Lipschitz distance between the microscopic field and this approximation is bounded by a constant times $\Ee\left[1\wedge\varepsilon_N\middle|\mathcal G_T^0\right]$. This proves \eqref{cr:critical-limit} and its joint version. Below criticality, the same argument gives the conditional Dirac kernel at $\chi$.
\end{proof}
\subsection{Translation model and sharpness}\label{app:rank-sharpness}
\begin{proof}[Proof of Corollary~\ref{rk:sharpness}]
The variables $\delta_{Y_i}-\nu$ are centered, independent and identically distributed in $H^{-s}$, with

\[
 \sup_{y\in\mathbf T}\|\delta_y\|_{H^{-s}}^2
 =\sum_{k\in\Z}(1+k^2)^{-s}<\infty
 \qquad\left(s>\frac{1}{2}\right).
\]

The Hilbert-space CLT gives $\mathcal B_N\xRightarrow[N\to\infty]{}\mathcal B$, independently of the master, with covariance \eqref{rk:bridge-covariance}. Strong continuity and the translation isometry give, for every compact $K\subset H^{-s}$,

\begin{equation}\label{rk:translation-compact}
 \sup_{\substack{b\in K,\ z\in\R\\|z|\le\delta}}\|\tau_zb-b\|_{H^{-s}}
                         \xrightarrow[\delta\downarrow0]{}0.
\end{equation}

Tightness of $\mathcal B_N$ then gives, for any real path errors $\|\delta_N\|_\infty\xrightarrow[N\to\infty]{\mathbf{P}}0$,

\begin{equation}\label{rk:bridge-translation-error}
 \sup_{0\le t\le T}\|\tau_{L_t+\delta_N(t)}\mathcal B_N
                   -\tau_{L_t}\mathcal B_N\|_{H^{-s}}\xrightarrow[N\to\infty]{\mathbf{P}}0.
\end{equation}

Since $\nu$ has a smooth density, Taylor's identity in $H^{-s}$ reads

\[
 \tau_{a+z}\nu-\tau_a\nu+z\tau_a\partial\nu
 =z^2\int_0^1(1-\theta)\tau_{a+\theta z}\partial^2\nu\dd\theta.
\]

The translation isometry therefore gives, for $a,z\in\R$,

\begin{equation}\label{rk:translation-taylor}
 \|\tau_{a+z}\nu-\tau_a\nu+z\tau_a\partial\nu\|_{H^{-s}}
       \le\frac12z^2\|\partial^2\nu\|_{H^{-s}}.
\end{equation}

Set $\delta_N\coloneqq L^N-L$. Then

\begin{equation}\label{rk:exact-decomposition}
 \mu_t^N-\mu_t
 \ \ \byeq{rk:translated-model}{=}\ \ 
 N^{-\frac{1}{2}}\tau_{L_t^N}\mathcal B_N
                         +(\tau_{L_t^N}-\tau_{L_t})\nu.
\end{equation}

For $\rho\le\tfrac{\alpha}{2}<1-\tfrac{\alpha}{2}$, the rank theorem gives $r_N^{-1}\delta_N\xrightarrow[N\to\infty]{\mathbf{P}} Z_{\lambda\cdot}$ uniformly on $[0,T]$.

\begin{align*}
 \sup_{0\le t\le T}\|r_N^{-1}(\tau_{L_t^N}-\tau_{L_t})\nu-\Xi_t\|_{H^{-s}}\ \ &\byeq{rk:translation-taylor}{\le}\ \ \|\partial\nu\|_{H^{-s}}
      \|r_N^{-1}\delta_N-Z_{\lambda\cdot}\|_\infty +\frac{\|\partial^2\nu\|_{H^{-s}}}{2r_N}\|\delta_N\|_\infty^2 \xrightarrow[N\to\infty]{\mathbf{P}}0.
\end{align*}

This yields

\begin{equation}\label{rk:field-expansion}
 r_N^{-1}(\mu^N-\mu)
   \ \ \overset{\mathclap{\substack{\eqref{rk:bridge-translation-error},\\\eqref{rk:exact-decomposition}}}}{=}\ \ \frac{N^{-\frac{1}{2}}}{r_N}\tau_L\mathcal B_N+\Xi+o_{\Pp}(1).
\end{equation}

The remainder in \eqref{rk:field-expansion} uses the uniform $H^{-s}(\mathbf T)$ norm. The Taylor error is $O_{\Pp}(r_N)$ after division by $r_N$. The bridge coefficient tends to zero below the threshold and equals one at the threshold. For $\rho>\tfrac{\alpha}{2}$, $$\|\delta_N\|_\infty  \overset{\text{Theorem~\ref{rk:tangent}}}{=}  O_{\Pp}(b_N)=o_{\Pp}\left(N^{-\frac{1}{2}}\right),$$ and the decomposition gives $\sqrt N(\mu^N-\mu)=\tau_L\mathcal B_N+o_{\Pp}(1)$ uniformly. For each fixed master, the map

\[
 \left|\begin{aligned}
 H^{-s}(\mathbf T)&\longrightarrow D([0,T];\R\times H^{-s}(\mathbf T))\\
 b&\longmapsto (L,\tau_Lb+\Xi)
 \end{aligned}\right.
\]

is continuous into vector $J_1$, since

\[
 \sup_{0\le t\le T}\|(\tau_{L_t}b+\Xi_t)-(\tau_{L_t}b'+\Xi_t)\|_{H^{-s}}
       =\|b-b'\|_{H^{-s}}.
\]

Since $\mathcal B_N$ is independent of the master, the conditional laws of the critical approximation converge almost surely. Omitting $\Xi$ gives the supercritical case. A vanishing bridge coefficient gives the subcritical case. If a remainder satisfies $\|E_N\|_\infty\xrightarrow[N\to\infty]{\mathbf{P}}0$, then $$\Ee[1\wedge\|E_N\|_\infty\mid\sigma(\Pi)]\xrightarrow[N\to\infty]{L^1(\Pp)}0,$$ which transfers these conditional kernels and joint stable limits. At criticality, for $t,u\in[0,T]$ and $\varphi,\psi\in H^s(\mathbf T)$, the covariance is

\begin{equation*}
 \nu\left(\varphi(\cdot+L_t)\psi(\cdot+L_u)\right)
 -\nu\left(\varphi(\cdot+L_t)\right)\nu\left(\psi(\cdot+L_u)\right).
\end{equation*}

Since $\partial\nu\ne0$ and $\Pp(Z_{\lambda t}=0)=0$, $\|\Xi_t\|_{H^{-s}}=|Z_{\lambda t}|\|\partial\nu\|_{H^{-s}}>0$ almost surely. For $\rho<\tfrac{\alpha}{2}$, write $\sqrt N(\mu_t^N-\mu_t)=\sqrt N r_N(\Xi_t+o_{\Pp}(1))$. Since $\sqrt N r_N\xrightarrow[N\to\infty]{}\infty$ and $\|\Xi_t\|_{H^{-s}}>0$ almost surely,

\[
 \Pp\left(\|\sqrt N(\mu_t^N-\mu_t)\|_{H^{-s}}>M\right)\xrightarrow[N\to\infty]{}1
 \qquad\text{for every }M<\infty,
\]

which proves non-tightness.
\end{proof}

\section{Clock limits and coupling proofs}\label{app:clock}
\subsection{Population and clock fluctuations}
\begin{proof}[Proof of Lemma~\ref{clk:joint-clt}]
The covariance density of the martingale parts is

\begin{equation}\label{clk:bracket}
 \sum_{\substack{1\le a,b\le m\\a\ne b}}r_{ab}(p_t)
 \binom{v_{ab}}1\binom{v_{ab}}1^{\!\top}
 \ \ \byeq{clk:rates}{=}\ \ 
 \begin{pmatrix}Q(p_t)&F(p_t)\\F(p_t)^\top&\lambda(p_t)\end{pmatrix}.
\end{equation}

The compensated counting martingales

\[
 M_t^{ab,N}\coloneqq N^{-\frac12}\left(C_{ab}^N(t)
                 -N\int_0^t r_{ab}(p_u^N)\dd u\right)
\]

are orthogonal, have jumps $N^{-\frac12}$, and satisfy
\begin{subequations}\label{clk:martingale-brackets}

\begin{equation}
 \langle M^{ab,N}\rangle_t=\int_0^t r_{ab}(p_u^N)\dd u.
\end{equation}

The cross bracket is

\begin{equation}
\begin{aligned}
 \left\langle\sum_{\substack{1\le a,b\le m\\a\ne b}}v_{ab}M^{ab,N},
 \sum_{\substack{1\le a,b\le m\\a\ne b}}M^{ab,N}\right\rangle{}_{t}
 &=\sum_{\substack{1\le a,b\le m\\a\ne b}}v_{ab}\int_0^t r_{ab}(p_u^N)\dd u =\int_0^t F(p_u^N)\dd u.
\end{aligned}
\end{equation}

\end{subequations}
Since the rates are bounded on $\Delta_m$, Doob's inequality gives

\[
 \Ee\left[\|M^{ab,N}\|_\infty^2\right]
 \le4\Ee\left[\langle M^{ab,N}\rangle_T\right]
 \ \ \byeq{clk:martingale-brackets}{\le}\ \ 
 4T\sup_{p\in\Delta_m}r_{ab}(p)<\infty.
\]

By the population equation,

\[
 \begin{aligned}
 G_t^N\ \ &\byeq{clk:chain}{=}\ \ 
 G_0^N+\int_0^t\sqrt N\{F(p_u^N)-F(p_u)\}\dd u +\sum_{\substack{1\le a,b\le m\\a\ne b}}v_{ab}M_t^{ab,N}.
\end{aligned}
\]

For $F:\Delta_m\to\mathsf T$ and $\lambda:\Delta_m\to\R$, $\operatorname{Lip}$ denotes the Lipschitz seminorm on $\Delta_m$ with Euclidean distances in the domain and range. Lipschitz continuity and Gronwall's inequality give
\begin{subequations}\label{clk:population-bound}

\begin{equation}
\begin{aligned}
 \sup_{u\in[0,t]}|G_u^N|
 &\le |G_0^N|+\sum_{\substack{1\le a,b\le m\\a\ne b}}|v_{ab}|\|M^{ab,N}\|_\infty +\operatorname{Lip}(F)\int_0^t\sup_{r\in[0,u]}|G_r^N|\dd u
\end{aligned}
\end{equation}

and therefore

\begin{equation}
\begin{aligned}
 \|G^N\|_\infty
 &\le e^{T\operatorname{Lip}(F)}
       \left(|G_0^N|+\sum_{\substack{1\le a,b\le m\\a\ne b}}|v_{ab}|\|M^{ab,N}\|_\infty\right) =O_{\Pp}(1).
\end{aligned}
\end{equation}

\end{subequations}
Hence $\|p^N-p\|_\infty=O_{\Pp}\left(N^{-\frac12}\right)$. Similarly,

\[
 \begin{aligned}
 H_t^N\ \ &\byeq{clk:chain}{=}\ \ 
 \int_0^t\sqrt N\{\lambda(p_u^N)-\lambda(p_u)\}\dd u +\sum_{\substack{1\le a,b\le m\\a\ne b}}M_t^{ab,N},
\end{aligned}
\]

so

\[
 \begin{aligned}
 \|H^N\|_\infty
 &\le T\operatorname{Lip}(\lambda)\|G^N\|_\infty
              +\sum_{\substack{1\le a,b\le m\\a\ne b}}\|M^{ab,N}\|_\infty
 \ \ \byeq{clk:population-bound}{=}\ \ O_{\Pp}(1).
 \end{aligned}
\]

The brackets converge uniformly in probability to the integral of \eqref{clk:bracket}, and vanishing jumps give Lindeberg's condition. Taylor's formula gives

\[
 \begin{aligned}
 \binom{G_t^N}{H_t^N}
 =\binom{G_0^N}{0}
 +\int_0^t
   \begin{pmatrix}DF(p_u)&0\\D\lambda(p_u)&0\end{pmatrix}
   \binom{G_u^N}{H_u^N}\dd u
 +\sum_{\substack{1\le a,b\le m\\a\ne b}}\binom{v_{ab}}1M_t^{ab,N}+\epsilon_t^N,
 \end{aligned}
\]

where the remainder $\epsilon^N\in D([0,T];\mathsf T\times\R)$ satisfies

\[
 \begin{aligned}
 \|\epsilon^N\|_\infty
 &\le C_TN^{-\frac12}\|G^N\|_\infty^2
 \ \ \byeq{clk:population-bound}{=}\ \ O_{\Pp}\left(N^{-\frac12}\right)\xrightarrow[N\to\infty]{\mathbf{P}}0.
 \end{aligned}
\]

The joint CLT for the initial variable and martingales \citep{EthierKurtz1986,Rebolledo1980,JacodShiryaev2003} gives the Gaussian martingales in \eqref{clk:limit-sde}, independent of $G_0$. Continuity of the linear integral equation proves the claim. Independence of the population from $S$ passes to the limit.
\end{proof}
Write $P_t\coloneqq\operatorname{Cov}(G_t)\in\R^{m\times m}$, $c_t\coloneqq\operatorname{Cov}(G_t,H_t)\in\mathsf T$, and $v_t\coloneqq\operatorname{Var}(H_t)\in[0,\infty)$ for $t\in[0,T]$. It\^o's product rule gives, with a martingale $M$,

\[
\begin{aligned}
 \dd(G_tH_t)
 &=H_t\dd G_t+G_t\dd H_t+\dd\langle G,H\rangle_t\\
 &\overset{\mathclap{\eqref{clk:limit-sde},\,\eqref{clk:bracket}}}{=}\ \ \ \
 \left(DF(p_t)G_tH_t
       +G_tG_t^\top D\lambda(p_t)^\top+F(p_t)\right)\dd t
 +\dd M_t.
\end{aligned}
\]
(With $\langle G,H\rangle$ the vector of predictable covariations between the martingale parts of $G$ and $H$.) 
Taking expectations and using the same rule for $GG^\top$ and $H^2$ gives

\begin{equation*}
\begin{aligned}
 \dot P_t&=DF(p_t)P_t+P_tDF(p_t)^\top+Q(p_t),\\
 \dot c_t&=DF(p_t)c_t+P_tD\lambda(p_t)^\top+F(p_t),\\
 \dot v_t&=2D\lambda(p_t)c_t+\lambda(p_t).
 \end{aligned}
\end{equation*}

The initial covariance is $P_0=\operatorname{Cov}(G_0)$, and $c_0=0$ and $v_0=0$. For $0\le r\le t\le T$, let the linear map $\Phi(t,r):\mathsf T\to\mathsf T$ solve $\partial_t\Phi(t,r)=DF(p_t)\Phi(t,r)$ and $\Phi(r,r)=\operatorname{Id}_{\mathsf T}$. For $1\le a,b\le m$ with $a\ne b$ and $0\le r\le t\le T$, define
\begin{subequations}\label{clk:variation-constants}

\begin{equation}
\begin{aligned}
 k_{ab}(t,r)&\coloneqq1+\int_r^tD\lambda(p_u)\Phi(u,r)v_{ab}\dd u =1+O(t-r)
\end{aligned}
\end{equation}

as $r\uparrow t$. Variation of constants gives

\begin{equation}
\begin{aligned}
 G_t&=\Phi(t,0)G_0 +\sum_{\substack{1\le a,b\le m\\a\ne b}}\int_0^t
       \Phi(t,r)v_{ab}\sqrt{r_{ab}(p_r)}\dd B_r^{ab},
\end{aligned}
\end{equation}

and stochastic Fubini gives

\begin{equation}
\begin{aligned}
 H_t&=\int_0^tD\lambda(p_u)\Phi(u,0)G_0\dd u +\sum_{\substack{1\le a,b\le m\\a\ne b}}\int_0^t
       k_{ab}(t,r)\sqrt{r_{ab}(p_r)}\dd B_r^{ab}.
\end{aligned}
\end{equation}

\end{subequations}
For $t\in(0,T]$, choose $\delta\in(0,t)$ such that $k_{ab}(t,r)^2\ge\tfrac14$ for $1\le a,b\le m$, $a\ne b$, and $r\in[t-\delta,t]$. Independence of $G_0$ and the Brownian motions gives

\begin{equation}\label{clk:positive-clock-variance}
 \begin{aligned}
 v_t
\ \ \byeq{clk:variation-constants}{\ge}\ \ 
 \sum_{\substack{1\le a,b\le m\\a\ne b}}\int_{t-\delta}^t k_{ab}(t,r)^2r_{ab}(p_r)\dd r &\ge\frac14\int_{t-\delta}^t\sum_{\substack{1\le a,b\le m\\a\ne b}}r_{ab}(p_r)\dd r\\
 \ \ &\byeq{clk:rates}{=}\ \ \frac14\int_{t-\delta}^t\lambda(p_r)\dd r
 \ge\frac{\lambda_*\delta}{4}>0.
\end{aligned}
\end{equation}

\subsection{The clock tangent}
\begin{proof}[Proof of Proposition~\ref{clk:tangent}]
Set $a_j\coloneqq A_{t_j}$ and $b_N\coloneqq N^{-\frac14}$. Since $A$ is strictly increasing, the $a_j$ are positive and distinct. Fix $R>a_k$. For large $N$, the intervals $I_{j,N}\coloneqq (a_j-b_N,a_j+b_N]$ are disjoint subsets of $[0,R]$. For $u\in[0,R]$, erase their increments from $S$:

\[
 S_u^{\circ,N}\coloneqq S_u-\sum_{j=1}^k
       \{S_{u\wedge(a_j+b_N)}-S_{u\wedge(a_j-b_N)}\}.
\]

Independent increments make $S^{\circ,N}$ independent of the erased increment processes. Stable scaling gives

\[
 \begin{aligned}
 \|S-S^{\circ,N}\|_{\infty,[0,R]}
 &\le\sum_{j=1}^k\sup_{v\in[0,2b_N]}
            |S_{a_j-b_N+v}-S_{a_j-b_N}| =O_{\Pp}\left(b_N^{\frac1\alpha}\right).
\end{aligned}
\]

Write $\mathcal Q_N\coloneqq (G_{t_j}^N,H_{t_j}^N)_{j=1}^k$ and $Z_{j,N}\coloneqq N^{\frac1{2\alpha}}(S_{T_{t_j}^N}-S_{a_j})$. Define the event $E_N\coloneqq\{\max_{1\le j\le k}|H_{t_j}^N|\le N^{\frac14}\}$. Tightness gives $\Pp(E_N)\xrightarrow[N\to\infty]{}1$. On $E_N$, each increment defining $Z_{j,N}$ lies in $I_{j,N}$, since $|T_{t_j}^N-a_j|=N^{-\frac12}|H_{t_j}^N|\le N^{-\frac14}=b_N$. Conditionally on the population, these increments are independent. Symmetry covers either sign of $H_{t_j}^N$, and stable scaling gives

\begin{equation*}
\begin{aligned}
 \Ee\left[e^{iu_jZ_{j,N}}\middle|\mathcal Q_N\right]
 &=\exp\left\{-|T_{t_j}^N-a_j|\,\left|N^{\frac1{2\alpha}}u_j\right|^\alpha\right\} =\exp\{-|H_{t_j}^N||u_j|^\alpha\}.
\end{aligned}
\end{equation*}

Thus, for bounded continuous $\varphi:(\mathsf T\times\R)^k\to\R$ and $\Psi:D([0,R];\R)\to\R$,

\[
 \begin{aligned}
 \Ee\left[\ind_{E_N}\varphi(\mathcal Q_N)\Psi(S^{\circ,N})
                               e^{i\sum_{j=1}^k u_jZ_{j,N}}\right] &=\Ee\left[\ind_{E_N}\varphi(\mathcal Q_N)
             e^{-\sum_{j=1}^k |H_{t_j}^N||u_j|^\alpha}\right]
                         \Ee\left[\Psi(S^{\circ,N})\right].
\end{aligned}
\]

Replace $S^{\circ,N}$ by $S$ and apply Lemma~\ref{clk:joint-clt}. Tightness of $(H_{t_j}^N)_{j=1}^k$ also gives tightness of $(Z_{j,N})_{j=1}^k$: conditionally on $\mathcal Q_N$ and $E_N$, their scales are $|H_{t_j}^N|^{\frac1\alpha}$. The identity identifies every subsequential joint limit with $S|_{[0,R]}$. Bounded continuous functionals are dense in $L^1(\Law(S|_{[0,R]}))$. Conditional expectations as $R\uparrow\infty$ then extend convergence to every bounded $\sigma(S)$-measurable test.
\end{proof}
\subsection{The two empirical centers}
\begin{proof}[Proof of Theorem~\ref{clk:finite-center}]
The identity

\begin{equation*}
\begin{aligned}
 \sqrt N(\mu_t^N-\mu_t^{[N]})
 &=\tau_{L_t^N}\sum_{a=1}^m\sqrt N(p_{t,a}^N-p_{t,a})\delta_{y_a} =\tau_{L_t^N}\sum_{a=1}^mG_{t,a}^N\delta_{y_a}.
\end{aligned}
\end{equation*}

reduces the proof to joint convergence of $L^N$ and $G^N$. Fix $R>A_T+1$, and let $\nu_S$ denote the L\'evy measure of $S$ on $\R$. On $[0,R]$, decompose $S=S^{>\delta}+M^{\le\delta}$ into jumps larger than $\delta>0$ and its compensated small-jump martingale. Symmetry removes the drift, and Doob's inequality gives

\begin{equation}\label{clk:small-jump-bound}
\begin{aligned}
 \Ee\left[\sup_{u\in[0,R]}|M_u^{\le\delta}|^2\right]
 &\le4R\int_{\{z\in\R:|z|\le\delta\}}z^2\nu_S(\dd z) =C_\alpha R\delta^{2-\alpha}.
\end{aligned}
\end{equation}

For fixed $\delta$, the finitely many jump times $u_j$ of $S^{>\delta}$ occupy distinct grid cells for large $N$. Every jump of $T^N$ is $\tfrac1N$ almost surely. For $0<u_j<A_T$, define its crossing time by $t_j^N\coloneqq\inf\left\{t\in[0,T]:T_t^N\ge\frac{\lceil Nu_j\rceil}{N}\right\}$, with $\inf\varnothing\coloneqq\infty$. Whenever this crossing exists,

\[
 |t_j^N-A^{-1}(u_j)|
 \le\lambda_*^{-1}(\|T^N-A\|_\infty+N^{-1}).
\]

All these crossings exist with probability tending to one because $\|T^N-A\|_\infty\xrightarrow[N\to\infty]{\mathbf{P}}0$ and $\dot A\ge\lambda_*$. The large-jump set is locally finite and $\Delta S_{A_T}=0$ almost surely, so the two clocks cross the same large jumps before $T$ with probability tending to one. Align their physical times by increasing piecewise linear homeomorphisms $\theta_N:[0,T]\to[0,T]$. They satisfy $\|\theta_N-\mathrm{id}\|_\infty\xrightarrow[N\to\infty]{\mathbf{P}}0$, and, on this event with $T_T^N\le R$,

\[
 \|(S\circ T^N)\circ\theta_N-S\circ A\|_\infty
                 \le2\sup_{u\in[0,R]}|M_u^{\le\delta}|.
\]

Let $N\to\infty$ and then $\delta\downarrow0$. This proves $L^N\xrightarrow[N\to\infty]{\substack{\eqref{clk:small-jump-bound}\\ \mathbf{P}}}L$ in $J_1$. Independence from $S$ gives joint convergence with $G^N\xRightarrow[N\to\infty]{\text{Lemma~\ref{clk:joint-clt}}} G$. For $x,y,z\in\R$, the translation modulus is

\begin{equation*}
 \begin{aligned}
 \|\tau_x\delta_y-\tau_z\delta_y\|_{H^{-s}}^2
 &=\frac1{2\pi}\int_\R(1+\xi^2)^{-s}
                         |e^{-ix\xi}-e^{-iz\xi}|^2\dd\xi\\
 &\le\frac1{2\pi}\int_\R(1+\xi^2)^{-s}
                         \min\{4,|x-z|^2\xi^2\}\dd\xi
       \xrightarrow[x-z\to0]{}0.
 \end{aligned}
\end{equation*}

On a representation realizing the joint convergence, the same time changes give

\[
 \begin{aligned}
 &\sup_{t\in[0,T]}\left\|
 \sqrt N(\mu^N-\mu^{[N]})_{\theta_N(t)}-\mathcal G_t^L
 \right\|_{H^{-s}}\\
 &\le\|\delta_0\|_{H^{-s}}
               \sum_{a=1}^m\|G_a^N\circ\theta_N-G_a\|_\infty +\sum_{a=1}^m\|G_a\|_\infty
       \sup_{t\in[0,T]}\|\tau_{L^N_{\theta_N(t)}}\delta_{y_a}
                         -\tau_{L_t}\delta_{y_a}\|_{H^{-s}} \xrightarrow[N\to\infty]{}0.
\end{aligned}
\]

Continuity of $G$ gives this limit. This proves \eqref{clk:finite-center-limit}. Joint convergence with $S$, followed by $L^1$ approximation of bounded $\sigma(S)$-measurable tests, gives stability. Independence of $G$ and $S$ gives the conditional covariance.
\end{proof}
\begin{proof}[Proof of Corollary~\ref{clk:regimes}]
The two centers give the decomposition

\begin{equation}\label{clk:empirical-decomposition}
 \mu_t^N-\mu_t^L
 \ \ \byeq{clk:centers}{=}\ \ 
 \tau_{L_t^N}(\nu_t^N-\nu_t)
                   +(\tau_{L_t^N}-\tau_{L_t})\nu_t.
\end{equation}

Since $|\widehat\nu_t(\xi)|\le1$, Fourier transformation gives, as $z\to0$,

\begin{equation}\label{clk:translation-taylor}
\begin{aligned}
 \|(\tau_{\ell+z}-\tau_\ell)\nu_t
                  +z\partial_x\tau_\ell\nu_t\|_{H^{-s}}^2 &=\frac1{2\pi}\int_\R(1+\xi^2)^{-s}
       |e^{-iz\xi}-1+iz\xi|^2|\widehat\nu_t(\xi)|^2\dd\xi =o(z^2).
\end{aligned}
\end{equation}

Indeed, $\tfrac{e^{-iz\xi}-1+iz\xi}{z}\xrightarrow[z\to0]{}0$ and its modulus is bounded by $2|\xi|$, with $\int_\R(1+\xi^2)^{-s}\xi^2\dd\xi<\infty$. The remainder is uniform in $\ell$ and $t$. At each fixed $t=t_j$,

\[
\begin{aligned}
 \mu_t^N-\mu_t^L
 \ \ &\byeq{clk:empirical-decomposition}{=}\ \
 \tau_{L_t^N}(\nu_t^N-\nu_t)+(\tau_{L_t^N}-\tau_{L_t})\nu_t\\
 \ \ &\byeq{clk:translation-taylor}{=}\ \
 N^{-\frac12}\tau_{L_t^N}\sum_{a=1}^mG_{t,a}^N\delta_{y_a}
 -(L_t^N-L_t)\partial_x\mu_t^L
 +o_{\Pp}\left(N^{-\frac1{2\alpha}}\right).
\end{aligned}
\]

By Proposition~\ref{clk:tangent}, the two leading terms have scales $N^{-\frac12}$ and $N^{-\frac1{2\alpha}}$. Dividing by the larger scale gives \eqref{clk:regimes-limit}. For $t\in(0,T]$, $v_t\ \ \byeq{clk:positive-clock-variance}{>}\ \ 0$, hence $\Pp(H_t=0)=0$. Moreover, $\widehat\nu_t(0)=1$ and continuity of $\widehat\nu_t$ imply $\partial_x\nu_t\ne0$. Since stable variables have no atom at zero, $\|\mathcal C_j\|_{H^{-s}}>0$ almost surely. For $\alpha>1$, multiplying the nonzero clock limit by $N^{\frac12-\frac1{2\alpha}}\xrightarrow[N\to\infty]{}\infty$ excludes root-$N$ tightness.
\end{proof}
\begin{proof}[Proof of Proposition~\ref{clk:no-path}]
The chain event times, mapped by $A$, form a countable random set independent of $S$. A L\'evy process has no fixed-time jumps, so almost surely no jump of $L$ is a chain event, simultaneously for all $N$. At every such time,
\begin{subequations}\label{clk:jump-mismatch}

\begin{equation}
 \Delta(L^N-L)_t=-\Delta L_t
\end{equation}

and

\begin{equation}
 \Delta(\mu^N-\mu^L)_t
 =-\tau_{L_{t-}}(\tau_{\Delta L_t}\nu_t-\nu_t).
\end{equation}

\end{subequations}
The stable L\'evy measure has infinite mass. Hence $S$ almost surely has a nonzero jump on $(0,A_T)$. Choose its physical time $t$ and put $z\coloneqq \Delta L_t\ne0$. A finitely supported probability on $\R$ cannot be invariant under a nonzero translation: invariance would imply that its finite support contains $y+nz$ for every integer $n$. Therefore $c_L\coloneqq\tfrac{|z|}{2}>0$ and $c_\mu\coloneqq\tfrac12\|\tau_z\nu_t-\nu_t\|_{H^{-s}}>0$. For any normed space $E$, path $x\in D([0,T];E)$, and $t\in(0,T]$,

\begin{equation}\label{clk:path-jump-bound}
 \|\Delta x_t\|\le\|x_t\|+\|x_{t-}\|\le2\|x\|_\infty.
\end{equation}

Thus, for every $N$,

\[
 \begin{aligned}
 &\|L^N-L\|_\infty
 \ \ \byeq{clk:path-jump-bound}{\ge}\ \ 
 \tfrac12|\Delta(L^N-L)_t|
 \ \ \byeq{clk:jump-mismatch}{=}\ \ \frac{|z|}{2}=c_L
 \end{aligned}
\]

and

\[
 \begin{aligned}
 \|\mu^N-\mu^L\|_{\infty,H^{-s}}
 \ \ &\byeq{clk:path-jump-bound}{\ge}\ \ 
 \frac12\|\Delta(\mu^N-\mu^L)_t\|_{H^{-s}} =\frac12\|\tau_z\nu_t-\nu_t\|_{H^{-s}}=c_\mu.
\end{aligned}
\]

Multiplication by $a_N\xrightarrow[N\to\infty]{}\infty$ proves \eqref{clk:no-path-bound}. Compact sets in each of these topologies have uniformly bounded ranges, so tightness fails.
\end{proof}
\subsection{Dependence at the Cauchy boundary}
\begin{corollary}[A Gaussian--Cauchy mixture]\label{clk:cauchy}
Let $\alpha=1$, $t\in(0,T]$, and $\phi\in C_b^1(\R)$. For $\ell\in\R$, define $a(\ell)\coloneqq(\phi(y_a+\ell))_{a=1}^m\in\R^m$ and $d(\ell)\coloneqq\sum_{a=1}^m p_{t,a}\phi'(y_a+\ell)$. Set $\kappa(\ell)\coloneqq a(\ell)^\top c_t$ and $r(\ell)\coloneqq a(\ell)^\top P_ta(\ell)-\tfrac{\kappa(\ell)^2}{v_t}\ge0$. Let $\xi$ be a standard Cauchy variable independent of $(G,H,S)$, and put $U\coloneqq a(L_t)^\top G_t$, $Z\coloneqq|H_t|\xi$, and $Y\coloneqq U+d(L_t)Z$. Then, stably relative to $\sigma(S)$, $$\sqrt N\left(\mu_t^N(\phi)-\mu_t^L(\phi)\right) \xRightarrow[N\to\infty]{}Y.$$ Conditionally on $L$, $U$ and $Z$ are independent if and only if $\kappa(L_t)=0$ almost surely. For $u\in\R$,

\begin{equation}\label{clk:cauchy-cf}
\begin{aligned}
 \Ee\left[e^{iuY}\middle|L\right] &=e^{-\frac{u^2r(L_t)}2}\Ee_{h\sim\mathcal N(0,v_t)}
   \left[\exp\left\{\frac{iu\kappa(L_t)}{v_t}h
              -|u d(L_t)|\,|h|\right\}\right].
\end{aligned}
\end{equation}

At time $t$, the Gaussian empirical field and $Z$ are conditionally independent if and only if $c_t=0$.
\end{corollary}
\begin{remark*}
These criteria concern $Z$ itself. When $d(L_t)=0$, the observed clock correction vanishes.
\end{remark*}
\begin{proof}
For $1\le a\le m$ and $\ell,z\in\R$, the fundamental theorem of calculus gives

\[
 \begin{aligned}
 \phi(y_a+\ell+z)-\phi(y_a+\ell)
 &=z\phi'(y_a+\ell) +z\int_0^1\{\phi'(y_a+\ell+\theta z)-\phi'(y_a+\ell)\}\dd\theta.
\end{aligned}
\]

The last term is $o(|z|)$, uniformly for $\ell$ in compact sets. Apply Proposition~\ref{clk:tangent}, using tightness of $L_t$ and $\sqrt N(L_t^N-L_t)$, to obtain the scalar limit. Conditional Gaussian regression gives $U=\frac{\kappa(L_t)}{v_t}H_t+R$. Here $R\perp H_t\mid L$ and $R\mid L\sim\mathcal N(0,r(L_t))$. The notation $R\perp H_t\mid L$ means that $R$ and $H_t$ are independent conditionally on the whole path of $L$. The residual $R$ is conditionally independent of $\xi$. Integrating over $R$, $\xi$, and $H_t$ in turn gives

\[
 \begin{aligned}
 \Ee\left[e^{iuY}\middle|L,H_t\right]
 &=e^{\frac{iu\kappa(L_t)H_t}{v_t}}\Ee\left[e^{iuR}\middle|L\right]
       \Ee\left[e^{iu d(L_t)|H_t|\xi}\middle|L,H_t\right] \\ &=\exp\left\{-\frac12u^2r(L_t)
       +\frac{iu\kappa(L_t)}{v_t}H_t-|u d(L_t)|\,|H_t|\right\}.
\end{aligned}
\]

Since $H_t\mid L\sim\mathcal N(0,v_t)$, averaging proves \eqref{clk:cauchy-cf}. If $\kappa(L_t)=0$, Gaussian independence gives $U\perp(H_t,\xi)\mid L$, hence $U\perp Z\mid L$. Conversely, $\Ee\left[e^{iuZ}\middle| U,L\right]=\Ee\left[e^{-|u||H_t|}\middle| U,L\right]$. If $U\perp Z\mid L$, uniqueness of Laplace transforms implies $U\perp|H_t|\mid L$. The Gaussian fourth-moment identity then gives $0=\operatorname{Cov}(U^2,H_t^2\mid L)=2\kappa(L_t)^2$. Finally, the map $g\mapsto\sum_{a=1}^m g_a\delta_{y_a+L_t}$ is injective because the locations are distinct. Equivalently, smooth test functions can take any prescribed values at these locations. The scalar criterion for all such tests is $a^\top c_t=0$ for every $a\in\R^m$, namely $c_t=0$.
\end{proof}
For $m=2$, $y_1=0$ and $y_2=1$, allow only $0\to1$ transitions at rate $\beta>0$, and start all particles in type $0$. Put $\pi_t\coloneqq 1-e^{-\beta t}$. The fraction at $1$ equals the normalized transition count, so $T_t^N=p_{t,2}^N$ and $A_t=\pi_t$. Consequently, $G_t=(-H_t,H_t)$ and $v_t=\pi_t(1-\pi_t)$, which gives $\kappa(\ell)=\pi_t(1-\pi_t)\{\phi(1+\ell)-\phi(\ell)\}$. Thus the critical clock correction can depend on the sampling field.

\subsection{Stable tails and the comparison of couplings}
\begin{lemma}[A quantitative stable tail]\label{clk:stable-tail}
Suppose $\Ee[e^{iuU}]=e^{-|u|^\alpha}$ for $u\in\R$ and $0<\alpha<1$. Define $c_\alpha\coloneqq\tfrac{2\Gamma(\alpha)\sin\left(\frac{\pi\alpha}{2}\right)}\pi$, where $\Gamma$ is the Gamma function. For every $x>0$,

\begin{equation}\label{clk:stable-tail-bound}
 \left|\Pp(|U|>x)-c_\alpha x^{-\alpha}\right|
 \le\frac{\Gamma(2\alpha)}\pi x^{-2\alpha}.
\end{equation}

\end{lemma}
\begin{proof}
The integrable characteristic function gives a density $f_U:\R\to[0,\infty)$ with $f_U(x)=\frac1\pi\operatorname{Re}\int_0^\infty e^{ixt-t^\alpha}\dd t,\, x\in\R$. For $x>0$, rotate the contour into the first quadrant, using the principal branch of $t^\alpha$. The small arc vanishes with its radius. On the large arc, write $t=|t|e^{i\theta}$ with $\theta\in\left[0,\tfrac\pi2\right]$. Then $|e^{ixt-t^\alpha}|=e^{-x|t|\sin\theta-|t|^\alpha\cos(\alpha\theta)}$ and $\cos(\alpha\theta)\ge\cos\left(\tfrac{\pi\alpha}{2}\right)>0$. Hence

\[
 f_U(x)=-\frac1\pi\operatorname{Im}\int_0^\infty
       e^{-xy}\exp\left\{-e^{i\frac{\pi\alpha}{2}}y^\alpha\right\}\dd y.
\]

For $z\in\mathbb C$ with $\operatorname{Re}z\ge0$, $e^{-z}-1+z=z^2\int_0^1(1-u)e^{-uz}\dd u$, so

\begin{equation*}
\begin{aligned}
 |e^{-z}-1+z|
 &\le|z|^2\int_0^1(1-u)e^{-u\operatorname{Re}z}\dd u \le\frac{|z|^2}{2}.
\end{aligned}
\end{equation*}

Using $z=e^{i\frac{\pi\alpha}{2}}y^\alpha$ and the identity $\int_0^\infty e^{-xy}y^a\dd y=\Gamma(a+1)x^{-a-1}$ for $a>-1$ gives

\begin{equation*}
\begin{aligned}
 \left|f_U(x)-\frac{\Gamma(\alpha+1)\sin\left(\frac{\pi\alpha}{2}\right)}\pi
                         x^{-1-\alpha}\right| &\le\frac1\pi\int_0^\infty e^{-xy}
 \left|e^{-e^{i\frac{\pi\alpha}{2}}y^\alpha}-1
                    +e^{i\frac{\pi\alpha}{2}}y^\alpha\right|\dd y\\ & \le\frac1{2\pi}\int_0^\infty e^{-xy}y^{2\alpha}\dd y
 =\frac{\Gamma(2\alpha+1)}{2\pi}x^{-1-2\alpha}.
\end{aligned}
\end{equation*}

Integrate from $x$ to infinity and use symmetry:

\[
 \begin{aligned}
 \left|2\int_x^\infty f_U(y)\dd y
       -\frac{2\Gamma(\alpha+1)\sin\left(\frac{\pi\alpha}{2}\right)}{\pi\alpha}x^{-\alpha}\right| &\le\frac{\Gamma(2\alpha+1)}\pi
                         \int_x^\infty y^{-1-2\alpha}\dd y\\
 &=\frac{\Gamma(2\alpha+1)}{2\pi\alpha}x^{-2\alpha}
 =\frac{\Gamma(2\alpha)}\pi x^{-2\alpha}.
\end{aligned}
\]

The identity $\Gamma(a+1)=a\Gamma(a)$ proves \eqref{clk:stable-tail-bound}.
\end{proof}

\begin{proof}[Proof of Proposition~\ref{clk:same-law}]
Use \eqref{rk:maps} with $c=c_\alpha$, and define $L_t^{N,\mathrm r}\coloneqq J^N_{\lambda t}$ and $L_t^{\mathrm r}\coloneqq S_{\lambda t}$. Integration by parts gives

\begin{equation}\label{clk:stable-cosine-integral}
\begin{aligned}
 \alpha\int_0^\infty(1-\cos(uz))z^{-1-\alpha}\dd z
 &=|u|^\alpha\int_0^\infty w^{-\alpha}\sin w\dd w =\Gamma(1-\alpha)\cos\left(\tfrac{\pi\alpha}{2}\right)|u|^\alpha.
\end{aligned}
\end{equation}

The Gamma reflection identity yields

\[
 \begin{aligned}
 c_\alpha\Gamma(1-\alpha)\cos\left(\tfrac{\pi\alpha}{2}\right)
 &=\frac{\Gamma(\alpha)\Gamma(1-\alpha)\sin(\pi\alpha)}\pi =1.
\end{aligned}
\]

Together with \eqref{clk:stable-cosine-integral}, this gives the stable normalization of the rank-master signal. Lemma~\ref{clk:stable-tail} implies \eqref{rk:tail-bound} with $\rho=\alpha>\tfrac\alpha2$. Localization and Doob's inequality give

\begin{equation*}
 \|L^{N,\mathrm r}-L^{\mathrm r}\|_\infty
               \overset{\substack{\text{Lemma~\ref{rk:localized}}\\ \text{Theorem~\ref{rk:transfer}}}}{=}  O_{\Pp}(b_N)
             \ \ \byeq{rk:rate}{=}\ \ o_{\Pp}\left(N^{-\frac12}\right),
\end{equation*}

with $b_N$ defined in \eqref{rk:rate}. The Hilbert-space CLT gives $\mathcal B_N\coloneqq \sqrt N(\nu^N-\nu)\xRightarrow[N\to\infty]{}\mathcal B$, independently of the master. The decomposition

\begin{equation}\label{clk:translation-decomposition}
 \sqrt N(\mu^N-\mu)
 =\tau_{L^N}\mathcal B_N+\sqrt N(\tau_{L^N}-\tau_L)\nu
\end{equation}

has a second term satisfying

\[
 \begin{aligned}
 \sup_{t\in[0,T]}\|\sqrt N(\tau_{L_t^N}-\tau_{L_t})\nu\|_{H^{-s}}
 &\le\sqrt N\|L^N-L\|_\infty\|\partial\nu\|_{H^{-s}} =o_{\Pp}(1).
\end{aligned}
\]

The first converges by \eqref{rk:translation-compact}. This proves \eqref{clk:same-law-rank}, stably with $S$. For the grid coupling, the marks are independent of $K^N$ and the sample, and

\[
 \Ee\left[e^{iuN^{\frac1\alpha}\left(S_{\frac{k}{N}}-S_{\frac{k-1}{N}}\right)}\right]
       =e^{-N^{-1}\left|N^{\frac1\alpha}u\right|^\alpha}=e^{-|u|^\alpha}.
\]

Independent increments prove equality of microscopic laws. The Poisson CLT gives $$\sqrt N\left(\tfrac{K^N}{N}-\lambda\mathrm{id}\right)\xRightarrow[N\to\infty]{}(B_{\lambda t})_{t\in[0,T]},$$ independently of $S$, where $B$ is standard Brownian motion. The argument that removes increments gives, jointly at finitely many deterministic times and stably with $S$, $$L_t^{N,\mathrm g}-L_t^{\mathrm g}   \overset{\substack{\text{proof of}\\ \text{Proposition~\ref{clk:tangent}}}}{=}  O_{\Pp}\left(N^{-\frac1{2\alpha}}\right)=o_{\Pp}\left(N^{-\frac12}\right).$$ Substitution into \eqref{clk:translation-decomposition} proves the finite-time CLT. Grid alignment gives $L^{N,\mathrm g}\xrightarrow[N\to\infty]{\substack{\eqref{clk:small-jump-bound}\\ \mathbf{P}}}L^{\mathrm g}$ in $J_1$. Combining this with the independent bridge CLT and \eqref{rk:translation-compact} proves \eqref{clk:same-law-grid-center}, stably with $S$. Almost surely, no jump of $L$ is a Poisson event for any $N$. The stable L\'evy measure has infinite mass on noninteger jump sizes, so there is a jump $z\notin\Z$ on $(0,\lambda T)$ almost surely. At its physical time $t$, for every $N$, $\Delta(\mu^{N,\mathrm g}-\mu^{\mathrm g})_t=-\tau_{L_{t-}}(\tau_z\nu-\nu)$. Since $\widehat\nu(1)=\tfrac\varepsilon2\ne0$, translation invariance would imply $e^{2\pi iz}=1$, contradicting $z\notin\Z$. Thus

\[
 \begin{aligned}
 \sup_{u\in[0,T]}\|\mu_u^{N,\mathrm g}-\mu_u^{\mathrm g}\|_{H^{-s}}
 &\ge\frac12\|\Delta(\mu^{N,\mathrm g}-\mu^{\mathrm g})_t\|_{H^{-s}} =\frac12\|\tau_z\nu-\nu\|_{H^{-s}}>0
\end{aligned}
\]

for every $N$. This proves \eqref{clk:same-law-obstruction} and excludes $J_1$ tightness.
\end{proof}

\section*{\normalfont\scshape AI Assistance Declaration}
During the preparation of this manuscript, the author used AI tools for language editing and manuscript organization. The author takes responsibility for the mathematical statements, proofs, computations, references, and final content.

\begingroup
\renewcommand{\bibfont}{\footnotesize}
\setlength{\bibsep}{1pt}

\endgroup

\end{document}